\documentclass[twoside,11pt]{article}
\usepackage{jmlr2e}
\usepackage{graphicx}
\usepackage{hyperref, url, breakurl}       % hyperlinks
\usepackage{algorithm}
\usepackage{algorithmic}
\usepackage{amssymb}
\usepackage{amsmath}
\usepackage{paralist}
\usepackage[pdftex,dvipsnames]{xcolor}  % Coloured text etc.
\usepackage{caption}

\newtheorem{thm}{Theorem}
\newtheorem{lem}{Lemma}
\newtheorem{cor}{Corollary}

\newtheorem{rem}{Remark}
\newtheorem{ass}{Assumption}

\DeclareMathOperator{\R}{\mathbb{R}} 
\DeclareMathOperator{\Ocal}{\mathcal{O}} 
\newcommand{\Exp}[1]{\mathbb{E}\left[#1\right]} 
\newcommand{\set}[1]{\left\{#1\right\}}
\newcommand{\sets}[1]{\{#1\}}
\newcommand{\norm}[1]{\left\Vert#1\right\Vert}
\newcommand{\norms}[1]{\Vert#1\Vert}
\newcommand{\dom}[1]{\mathrm{dom}\left(#1\right)}
\newcommand{\iprods}[1]{\langle#1\rangle}

\newcommand{\mytxtbi}[1]{\textbf{\textit{#1}}}

\firstpageno{1}

\newcommand\blfootnote[1]{%
  \begingroup
  \renewcommand\thefootnote{}\footnote{#1}%
  \addtocounter{footnote}{-1}%
  \endgroup
}

\usepackage{lastpage}
\jmlrheading{22}{2021}{1-\pageref{LastPage}}{11/20; Revised
6/21}{9/21}{20-1238}{L. M. Nguyen, Q. Tran-Dinh, D. T. Phan, P. H. Nguyen, and M. van Dijk}
\ShortHeadings{A Unified Convergence Analysis for Shuffling-Type Gradient Methods}{L. M. Nguyen, Q. Tran-Dinh, D. T. Phan, P. H. Nguyen, and M. van Dijk}

\begin{document}

\title{A Unified Convergence Analysis for Shuffling-Type Gradient Methods}
       
\author{\name{Lam M. Nguyen} \email{LamNguyen.MLTD@ibm.com} \\
	\addr IBM Research, Thomas J. Watson Research Center \\ Yorktown Heights, NY10598, USA.
      	\AND
	\name{Quoc Tran-Dinh} \email{quoctd@email.unc.edu} \\
	\addr Department of Statistics and Operations Research\\
	The University of North Carolina at Chapel Hill \\ Chapel Hill, NC27599, USA.
	\AND
	\name{Dzung T. Phan} \email{phandu@us.ibm.com}\\
	\addr IBM Research, Thomas J. Watson Research Center \\ Yorktown Heights, NY10598, USA.
	\AND
	\name{Phuong Ha Nguyen} \email{phuongha.ntu@gmail.com}\\
	\addr eBay Inc., San Jose, CA 95125.
	\AND
	\name{Marten van Dijk} \email{marten.van.dijk@cwi.nl}\\
        \addr Centrum Wiskunde $\&$ Informatica, Amsterdam, Netherlands.
       }

\editor{Prateek Jain}

   \maketitle

%%%%%%%%%%%%%%%%%%%%%%%%%%%%%%%%%%%%%%%%%
%%% + Abstract.
%%%%%%%%%%%%%%%%%%%%%%%%%%%%%%%%%%%%%%%%%
\begin{abstract}
In this paper, we propose a unified convergence analysis for a class of generic shuffling-type gradient methods for solving finite-sum optimization problems.
Our analysis works with any sampling without replacement strategy and covers many known variants such as randomized reshuffling, deterministic or randomized single permutation, and cyclic and incremental gradient schemes.
We focus on two different settings: strongly convex and nonconvex problems, but also discuss the non-strongly convex case.
Our main contribution consists of new non-asymptotic and asymptotic convergence rates for a wide class of shuffling-type gradient methods in both nonconvex and convex settings. 
We also study  uniformly randomized shuffling variants with different learning rates and model assumptions.
While our rate in the nonconvex case is new and significantly improved over existing works under standard assumptions, the rate on the strongly convex one matches the existing best-known rates prior to this paper up to a constant factor without imposing a bounded gradient condition. 
Finally, we empirically illustrate our theoretical results via two numerical examples: nonconvex logistic regression and neural network training examples.
As byproducts, our results suggest some appropriate choices for diminishing learning rates in certain shuffling variants.
\end{abstract}


\blfootnote{Corresponding author: Lam M. Nguyen}

%%%%%%%%%%%%%%%%%%%%%%%%%%%%%%%%%%%%%%%%%
%%% 1. Introduction.
%%%%%%%%%%%%%%%%%%%%%%%%%%%%%%%%%%%%%%%%%
\section{Introduction}\label{sec:intro}
This paper aims at analyzing convergence rates of a general class of shuffling-type gradient methods for solving the following well-known finite sum minimization problem:
\begin{equation}\label{ERM_problem_01}
    \min_{w \in \mathbb{R}^d} \left\{ F(w) := \frac{1}{n} \sum_{i=1}^n f(w ; i)  \right\},
\tag{P}
\end{equation}
where $f(\cdot; i) : \R^d\to\R$ is  smooth and possibly nonconvex for $i \in [n] := \set{1,\cdots,n}$.

Problem \eqref{ERM_problem_01} covers a wide range of convex and nonconvex models in machine learning and statistical learning, including logistic regression, multi-kernel learning, conditional random fields, and neural networks. 
Especially, it covers \textit{empirical risk minimization} as a special case. 
Very often, \eqref{ERM_problem_01} lives in a high dimensional space, and/or it has a large number of components $n$.
Therefore, deterministic optimization methods relying on full gradients are usually inefficient to solve \eqref{ERM_problem_01}, see, e.g., \citep{Bottou2018,sra2012optimization}.

The stochastic gradient descent (SGD) method has been widely used to solve \eqref{ERM_problem_01} due to its efficiency in dealing with large-scale problems in big data regimes.
The first variant of SGD (called \textit{stochastic approximation method}) was introduced by \citep{RM1951}.
In the last fifteen years, there has been a tremendous progress of research in SGD, where various stochastic and randomized-based algorithms have been proposed, making it be one of the most active research areas in optimization as well as in machine learning.
In addition, due to the deep learning revolution, research on SGD for nonconvex optimization for deep learning also becomes extremely active nowadays.

SGD is also a method of choice to solve the following stochastic optimization problem:
\begin{equation}\label{Expectation_problem_01}
    \min_{w \in \mathbb{R}^d} \Big\{ F(w) := \mathbb{E}_{(x,y) \sim \mathcal{D}} \big[ f(w; x, y) \big]  \Big\},
\end{equation}
where $\mathcal{D}$ is some probability distribution. 
Clearly, \eqref{ERM_problem_01} can be cast into a special case of \eqref{Expectation_problem_01}.

To solve \eqref{ERM_problem_01}, at each iteration $k$ (for $k := 0,1,\cdots, K$), SGD chooses an index $i_k \in [n]$ (or a minibatch) at random and updates an iterate sequence $\sets{w_k}$ as $w_{k+1} := w_k - \eta_k \nabla f(w_k ; i_k)$ from a given starting point $w_0$, which is up to $n$ times ``component gradient'' cheaper than one iteration of a full gradient method with the update $w_{k+1} := w_{k} -  \frac{\eta_k}{n} \sum_{i=1}^n f(w_k ; i)$, where $\eta_k > 0$ is called a learning rate at the $k$-th iteration and $i_k = (x_k, y_k)$, a single sample or a mini-batch of the input data $(x, y)$. 
Although the first variant of SGD was introduced in the 1950s, its convergence rate was investigated much later, see, e.g., \citep{Polyak1992}.
The convergence rate of SGD for solving \eqref{Expectation_problem_01} under strong convexity is $\Ocal( K^{-1})$ \citep{Nemirovski2009,Polyak1992,Nguyen2018_sgdhogwild}, and for finding a stationary point of \eqref{Expectation_problem_01} in the nonconvex case is $\Ocal\big( K^{-1/2}\big)$ \citep{ghadimi2013stochastic}, where $K$ is the total iteration number.
In particular, these rates also apply to \eqref{ERM_problem_01}.

%%% Motivation.
\paragraph{Motivation.}
This paper is motivated by a number of observations as follows.

\indent{$\bullet$} Firstly, shuffling gradient-type methods are widely used in practice. 
The classical SGD scheme (we refer to it as the \textit{standard SGD} method in this paper) for solving \eqref{ERM_problem_01} relies on an i.i.d. sampling scheme to select components $\nabla f(\cdot; i)$ for updating the iterates $w_k$.
In practice, however,  other mechanisms for selecting components $\nabla f(\cdot; i)$ such as \textit{randomized [re]shuffling techniques} are more desirable to use for implementing stochastic gradient algorithms.
Shuffling strategies  are easier and faster to implement in practice. 
They have been implemented in several well-known packages such as TensorFlow and PyTorch.
In addition, it has been recognized that shuffling-type methods often decrease the training loss faster than standard SGD, see, e.g.,  \citep{bottou2009curiously,bottou2012stochastic,JMLR:v18:17-632}. 
However, convergence guarantee of shuffling-type algorithms has just recently emerged.
This motivates us to conduct a unified analysis for a general class of shuffling-type algorithms.

\indent{$\bullet$} Secondly, convergence analysis of shuffling-type schemes and cyclic strategies is much more challenging than that of the standard SGD or its variants due to the lack of independence.
Hitherto, there is a limited number of theoretical works that analyze the convergence rates of shuffling techniques, where a majority has focussed on the strongly convex case \citep{gurbuzbalaban2015random,haochen2018random,Safran2019HowGoodSGDShuffling,nagaraj2019sgd}.  
To the best of our knowledge, hitherto and prior to our work, only \citep{li2019incremental,meng2019convergence} have studied convergence rates of shuffling-type gradient methods for solving nonconvex instances of \eqref{ERM_problem_01}.
There exists no unified analysis that can cover a wide class of shuffling-type gradient algorithms under different assumptions ranging from strongly convex to nonconvex cases.
In this paper, we will provide some key elements to form a unified analysis framework for shuffling-type gradient methods, which can be applied to different variants.
While we only focus on the strongly convex and nonconvex cases without specifying shuffling strategy, we believe that our framework can be customized to take into account additional or alternative assumptions to achieve a possibly better convergence rate (see, e.g., \citep{mishchenko2020random} as an example).

\indent{$\bullet$} Thirdly, existing shuffling-type gradient algorithms for the strongly convex case of \eqref{ERM_problem_01} such as \citep{ahn2020sgd,gurbuzbalaban2015random,haochen2018random,Safran2019HowGoodSGDShuffling,nagaraj2019sgd} require a \textbf{bounded gradient assumption}.
Although the bounded gradient condition is widely used and accepted for strongly convex problems, it leads to implicitly imposing a ball constraint on \eqref{ERM_problem_01}.
More precisely, if $F$ is $\mu$-strongly convex and $G$-bounded gradient as $\sup_{i\in [n]}\norm{\nabla{f}(w ; i)}  \leq G$ for $\forall w \in \dom{F}$ (the domain of $F$), then it is easy to show that $\frac{\mu}{2}\norm{w - w^{\star}}^2 \leq F(w) - F(w^{\star}) \leq \frac{1}{2\mu}\norm{\nabla{F}(w)}^2 \leq \frac{G^2}{2\mu}$ for all $w \in\dom{F}$, where $w^{\star}$ is the unique minimizer of $F$  \citep{Nguyen2018_sgdhogwild}.
This expression implies that $\norm{w - w^{\star}} \leq \frac{G}{\mu}$ as an implicit ball constraint on  \eqref{ERM_problem_01}. 
However, we usually do not know $w^{\star}$ to quantify this condition in practice.
Note that since $F$ is strongly convex, it is also coercive. 
Thus any sublevel set $\{ w \in \dom{F} : F(w) \leq F(w_0)\}$ (for a fixed $w_0$) is bounded, and consequently, $G$ exists. 
Nevertheless, it is still difficult to quantify $G$ over this sublevel set since we do not know the size of this sublevel set explicitly.
If $G$ is quantified inappropriately, it will change problem \eqref{ERM_problem_01} from the unconstrained to the constrained setting which may not be equivalent to \eqref{ERM_problem_01}.
In addition, a projection is required to guarantee the feasibility of the iterates in this case, adding another computational cost, see, e.g., \citep{nagaraj2019sgd}.
A recent work in \citep{mishchenko2020random} can also avoid the bounded gradient condition in their analysis, but only focuses on the case where each component $f(\cdot; i)$ is convex.

\indent{$\bullet$} Fourthly, for the strongly convex case, prior to this work, existing analysis relies on a set of strong assumptions.
Such assumptions often include: strong convexity, $L$-smoothness, Lipschitz Hessian, and bounded gradient.
HaoChen and Sra  prove an $\Ocal((nT)^{-2} + T^{-3})$ convergence rate for randomized reshuffling scheme in \citep{haochen2018random},  where $T$ is the number of epochs (i.e. the number of passes over $n$ component functions $f(\cdot; i)$), but under stronger assumptions than ours.
Further improvement can be found in \citep{ahn2020sgd,mishchenko2020random}.
The rates in these papers nearly match the lower bound proven in  \citep{Safran2019HowGoodSGDShuffling} or a sharper one in \citep{rajput2020closing}. 
Our rate is $\Ocal(T^{-2})$ in the general case, and $\Ocal(1/(nT^2))$ in the randomized reshuffling case, but only requires $f(\cdot; i)$ to be $L$-smooth and and $F$ to be strongly convex with bounded variance.
Hence, it is unclear if one can fairly compare ours and \citep{haochen2018random} since they consider two different classes of problems due to the use of different sets of assumptions.  

\indent{$\bullet$} Finally, prior to our work, convergence analysis of shuffling-type gradient schemes has not been rigorously investigated for the nonconvex setting of \eqref{ERM_problem_01}.
Existing works only focus on special variants such as incremental gradient or using nonstandard criterion \citep{li2019incremental,meng2019convergence}.
Our work is the first analyzing convergence rates of the general shuffling-type gradient scheme in the nonconvex case under standard assumptions, which achieves the best known $\Ocal{(T^{-2/3})}$ or $\Ocal(n^{-1/3}T^{-2/3})$ rate in epoch.

%%% Contribution.
\paragraph{Contribution.}
In this paper, we develop a new and unified convergence analysis framework for general shuffling-type gradient methods to solve \eqref{ERM_problem_01} and apply it to different shuffling variants in both nonconvex and strongly convex settings under standard assumptions.
More specifically, our contribution can be summarized as follows:
\begin{compactitem}
    \item[$\mathrm{(a)}$]
    We prove $\Ocal(1/T^2)$-convergence rate in epoch of a generic shuffling-type gradient scheme for the strongly convex case without imposing ``gradient boundedness'' and/or Lipschitz Hessian assumptions, e.g., in \citep{gurbuzbalaban2015random,haochen2018random}. 
    In addition, our analysis does not require convexity of each component function as in some existing works. 
    Similar to  \citep{gurbuzbalaban2015random} our rate also can be viewed as $\Ocal(1/t^2)$ for any $1 \leq t\leq T$ without fixing $T$ a priori as in other works.      
    
    \item[$\mathrm{(b)}$]
    If either a general bounded variance condition is imposed on $F$ (see Assumption~\ref{ass_general_bounded_variance}) 
    or each component $f(\cdot; i)$   is convex for all $i \in [n]$, then by using a uniformly randomized reshuffling strategy, 
    our convergence rate in the strongly convex case is improved to $\Ocal(1/(nT^2))$, which matches the best-known results in recent works, 
    including \citep{ahn2020sgd,mishchenko2020random}.
    Our latter case (i.e. when $f(\cdot; i)$ is convex) holds for both constant and diminishing learning rates.

    \item[$\mathrm{(c)}$] We prove $\Ocal(T^{-2/3})$-convergence rate in epoch for the constant stepsizes 
    and $\tilde{\Ocal}(T^{-2/3})$-convergence rate\footnote{The notation $\tilde{\Ocal}(\cdot)$ hides all logarithmic terms of the input $(\cdot)$ compared to the standard $\Ocal(\cdot)$ notation.} for the diminishing stepsizes of a general shuffling-type gradient method (Algorithm~\ref{sgd_replacement}) to approximate a stationary point of the nonconvex problem  \eqref{ERM_problem_01}, where $T := K/n$ is number of epochs.
   Our rate is significantly improved over $\Ocal(T^{-1/2})$ rate of the special incremental gradient method in \citep{li2019incremental}.
   To the best of our knowledge, these are the first improved non-asymptotic rates for SGD with shuffling for both constant and diminishing learning rates under standard assumptions. 
   When a uniformly randomized reshuffling strategy is used, our rate is improved by a factor of $n^{1/3}$ to $\Ocal(n^{-1/3}T^{-2/3})$.
    
    \item[$\mathrm{(d)}$]
    We establish asymptotic convergence to a stationary point  under diminishing learning rate scheme. 
    We theoretically and empirically show that the shuffling-type gradient algorithm achieves the best performance with the learning rate $\eta_t = \Ocal(t^{-1/3})$, where $t$ is the epoch counter. 
    Our learning rate closely relates to a ``scheduled'' one, i.e. it is constant within each epoch $t$ and decreases w.r.t. $t$.
    When a uniformly randomized reshuffling strategy is used, our rate is also improved by a factor of $n^{1/3}$.
\end{compactitem}

%% Comparison.
\paragraph{Comparison.}
Our theoretical results are new and different from existing works in several aspects.
For the nonconvex case, \citep{meng2019convergence} only proves $\Ocal(1/\sqrt{nT} + \log(n)/n)$-convergence rate to a neighborhood of a stationary point of a randomized reshuffling variant, \citep{li2019incremental}  shows $\Ocal(T^{-1/2})$ convergence rate under the bounded subgradients and weak convexity conditions for an incremental subgradient variant. 
Our convergence rate is $\Ocal(T^{-2/3})$ in epoch,  and hence is $T^{-1/6}$ factor  better  than \citep{li2019incremental} but using the smoothness of $f(\cdot; i)$ instead of a weak convexity condition. 
Compared to standard SGD, our rate is $\Ocal(n^{2/3}K^{-2/3})$ in the total of iterations, while the rate of SGD is $\Ocal(K^{-1/2})$.
Our bound is only better than SGD if $n < \Ocal( K^{1/4} )$ under any shuffling strategy.
If a randomized reshuffling strategy is used, our rate is improved to $\Ocal(n^{1/3}K^{-2/3})$.
This rate is better than SGD if $n < \Ocal(K^{1/2})$.
Note that our results and the ones of standard SGD are using different assumptions which may lead to the dependency on some constants for both bounds. 
In theory, shuffling gradient methods seem not better than SGD when $n$ is large.
However, these methods are often implemented in practice, especially in machine learning, as we have mentioned earlier.
Our detailed discussion on the comparison with SGD is given in Remark~\ref{rem_nonconvex}.
Further comparison between randomized reshuffling methods and SGD, GD, and other deterministic shuffling schemes for strongly convex problems can be found, e.g., in \citep{haochen2018random,Safran2019HowGoodSGDShuffling}.

%%% Related work.
\paragraph{Related work.} 
Let us briefly review the most related works to our methods in this paper.
The random shuffling-type method has been empirically studied in early works such as \citep{bottou2009curiously} and also discussed in \citep{bottou2012stochastic}. 
Its cyclic variant, known as an incremental gradient method was proposed even much earlier, see \citep{nedic2001incremental}, where the convergence analysis was given in \citep{nedic2001convergence} for a subgradient variant, and in \citep{gurbuzbalaban2015convergence} for gradient  variants.
These results are only for convex problems.
Other incremental gradient variants can be found, e.g., in \citep{Defazio2014,defazio2014finito} known as SAGA-based methods.
Our method is more general since it covers different shuffling variants in both deterministic and randomized worlds.
When a randomized reshuffling strategy is used, we can improve our results to match the best known rates in both convex and nonconvex settings.

In \citep{gurbuzbalaban2015random}, the authors showed that if $T$ is large, the randomized shuffling gradient method asymptotically converges with $\Ocal(T^{-2})$ rate under a proper stepsize.
However, this rate was only shown for strongly convex problems with bounded gradient/sequence, smoothness, and Lipschitz Hessian.
These assumptions all together are very restrictive to hold in practice. 
Under the same conditions, \citep{haochen2018random} improved the convergence rate to  $\Ocal((nT)^{-2} + T^{-3})$ non-asymptotically, but in the regime of $T/\log(T) \geq \Ocal(n)$.  
Another related work is \citep{nagaraj2019sgd}, which achieves $\tilde{\Ocal}(1/(nT^2))$ convergence rates without Lipschitz Hessian when $T$ is above some order of the condition number.
Such a paper still requires $F$ to have uniformly bounded gradient on its domain.
Recently, an $\Omega(T^{-2} + n^2T^{-3})$ lower bound is proved in \citep{Safran2019HowGoodSGDShuffling} under the same assumptions as \citep{haochen2018random}.
Another $\Omega(nT^{-2})$ shaper lower bound  is recently established in \citep{rajput2020closing} when $n > \Ocal(T)$.

In \citep{ying2017convergence}, the authors replaced  the i.i.d. sampling scheme by a randomized reshuffling strategy and established that variance reduced methods such as SAGA and SVRG still have a linear convergence rate for strongly convex problems but using an  unusual energy function.
Unfortunately, it is unclear how to transform such a criterion to standard convergence criteria such as loss residuals or solution distances.
The stochastic gradient method and its variance-reduced variants (e.g., SAG \citep{SAG}, SAGA \citep{Defazio2014}, SVRG \citep{SVRG}, and SARAH \citep{Nguyen2017sarah}) for solving \eqref{ERM_problem_01} under strong convexity and smoothness structures usually have theoretically  linear convergence rates compared to a sublinear rate in shuffling gradient variants. 
These variance-reduced techniques have also been widely exploited  in nonconvex settings, but often require full gradient evaluations due to the use of double loops \citep{Pham2019,Reddi2016a,Tran-Dinh2019a}.
However, these algorithms are different from the standard shuffling-type gradient method we study in this paper.

In \citep{shamir2016without}, a convergence rate to a neighborhood of the optimal value of an SGD variant using a without-replacement sampling strategy is studied for general convex problems.
Clearly, this type of convergence is different from ours, and requires $n$ to be large to get a suitable bound.
If  \eqref{ERM_problem_01} is generalized linear and strongly convex, then a faster $\Ocal(\log(K)/K)$ rate is achieved.
Another recent work is \citep{meng2019convergence} which considers different distributed SGD variants with shuffling for  strongly convex, general convex, and nonconvex problems.
The authors can only show convergence to a neighborhood of an optimal solution or a stationary point as in  \citep{shamir2016without}.
In addition, their rates are much slower than existing results for the strongly convex case, and also slower than ours, while requiring stronger assumptions.
After the first draft of this paper was online, \citep{mishchenko2020random} have made a step further by improving the convergence rates of the randomized reshuffling variant for the strongly convex case, but require convexity of each component function $f(\cdot; i)$.
Moreover, they only focus on constant stepsize, while our results cover both constant and diminishing stepsizes.

%%% Paper outline
\paragraph{Paper outline.}
The rest of this paper is organized as follows.
Section~\ref{sec:RSGD} describes our general shuffling gradient algorithm to solve \eqref{ERM_problem_01}. We state our main assumptions and provides necessary mathematical tools in Section~\ref{sec:basic_tools}.
Sections~\ref{sec_analysis_01} and \ref{sec_analysis_02} analyze convergence for the nonconvex and convex cases, respectively.
Several numerical experiments are presented in Section \ref{sec_experiment}. For the sake of presentation, all details and proofs are deferred to Appendix.

%%%%%%%%%%%%%%%%%%%%%%%%
%%% 2. Randomized Shuffling GD Algorithm
%%%%%%%%%%%%%%%%%%%%%%%%
\section{The Shuffling-Type Gradient Algorithm and Technical Lemmas}\label{sec:RSGD}
Let us first describe our generic shuffling-type gradient algorithm for \eqref{ERM_problem_01}.
Next, we state two standard assumptions imposed on \eqref{ERM_problem_01}, which will be used in this paper.
Finally, we prove two technical lemmas that serve as key steps for our convergence analysis.

%%%% 2.1. The generic shuffling-type gradient algorithm
\subsection{The generic shuffling-type gradient algorithm}\label{subsec:algorithm}
Shuffling-type gradient  methods are widely used in practice due to their efficiency \citep{bottou2009curiously}. 
Moreover, these methods have been investigated in many recent papers, including \citep{gurbuzbalaban2015random,haochen2018random,nagaraj2019sgd}.  
In this paper, we analyze convergence rates for a wide class of shuffling-type schemes to solve \eqref{ERM_problem_01} in both convex and nonconvex settings as described in Algorithm~\ref{sgd_replacement}.

\begin{algorithm}[hpt!]
   \caption{(Generic Shuffling-Type Gradient Algorithm for Solving \eqref{ERM_problem_01})}\label{sgd_replacement}
\begin{algorithmic}[1]
   \STATE {\bfseries Initialization:} Choose an initial point $\tilde{w}_0\in\dom{F}$.
   \FOR{$t=1,2,\cdots,T $}
   \STATE Set $w_0^{(t)} := \tilde{w}_{t-1}$;
   \STATE Generate any permutation  $\pi^{(t)}$ of $[n]$ (either deterministic or random);
   \FOR{$i = 1,\cdots, n$}
    \STATE Update $w_{i}^{(t)} := w_{i-1}^{(t)} - \eta_i^{(t)} \nabla f ( w_{i-1}^{(t)} ; \pi^{(t)} ( i ) )$; 
   \ENDFOR
   \STATE Set $\tilde{w}_t := w_{n}^{(t)}$;
   \ENDFOR
\end{algorithmic}
\end{algorithm} 

Note that $\pi^{(t)}(i)$ is the $i$-th element of $\pi^{(t)}$ for $i \in [n]$. 
Each outer iteration $t$ of Algorithm~\ref{sgd_replacement} can be counted for one epoch.
The inner loop updates the iterate sequence $\sets{w_i^{(t)}}_{i=1}^{n}$ using only one component per iteration as in SGD by shuffling the objective components.
Our analysis will be done in epoch-wise (i.e. the convergence guarantee is on $F(\tilde{w}_t) - F_{*}$ evaluated at the outer iterate $\tilde{w}_t$ instead of $w_i^{(t)}$).
Depending on the choice of $\pi^{(t)}$ we obtain different variants, especially the following methods:
\begin{compactitem}
\item If $\pi^{(t)} = \sets{1,2,\cdots, n}$ or some fixed permutation of $\sets{1,2,\cdots,n}$ for all epochs $t$, then Algorithm~\ref{sgd_replacement} is equivalent to a cyclic gradient method.
This method can also be viewed as the incremental gradient scheme studied in \citep{nedic2001incremental}, and recently in \citep{li2019incremental}.
\item If $\pi^{(t)}$ is randomly generated one time and repeatedly used at each iteration $t$, then Algorithm~\ref{sgd_replacement} becomes a single shuffling variant \citep{Safran2019HowGoodSGDShuffling}.
\item If $\pi^{(t)}$ is randomly generated at each epoch $t$, then  Algorithm~\ref{sgd_replacement} reduces to a randomized reshuffling scheme, broadly used in practice, see, e.g., \citep{JMLR:v18:17-632}.
\end{compactitem}
The randomized reshuffling schemes have been studied, e.g., in \citep{gurbuzbalaban2015random,haochen2018random,nagaraj2019sgd}, but their convergence analysis has mainly been investigated for the strongly convex case and often under a strong set of assumptions.

%%% 2.2. Standard assumptions
\subsection{Model assumptions}\label{subsec:assumptions}
Our analysis throughout the paper relies on the following standard assumptions of \eqref{ERM_problem_01}.
\begin{ass}\label{ass_basic}
Assume that problem \eqref{ERM_problem_01} satisfies the following conditions:
\begin{enumerate}[$($i$)$] 
\item $\dom{F} := \sets{x \in \R^d : F(x) < +\infty} \neq\emptyset$ and $F_{*} := \inf_{w\in\R^d}F(w) > -\infty$.
\item $f(\cdot; i)$ is $L$-smooth for all $i \in [n]$, i.e. there exists a constant $L \in (0, +\infty)$ such that:
\begin{equation}\label{eq:Lsmooth_basic}
\norms{ \nabla f(w;i) - \nabla f(\hat{w} ;i)} \leq L \norms{ w - \hat{w}}, \quad \forall w, \hat{w} \in \dom{F}. 
\end{equation}
\end{enumerate}
\end{ass}
Note that Assumption~\ref{ass_basic}(\textit{i}) is required in any algorithm to guarantee the well-definedness of \eqref{ERM_problem_01}.
Assumption~\ref{ass_basic}({\emph{ii}}) is standard in gradient-type methods.
Hence, we refer to Assumption~\ref{ass_basic} as  the standard assumption required throughout the paper.
Assumption \ref{ass_basic}(ii) implies that the objective function $F$ is also $L$-smooth. 
Moreover, as proven in \citep{nesterov2004}, we have 
\begin{equation}\label{eq:Lsmooth}
F(w) \leq F(\hat{w}) + \iprods{\nabla F(\hat{w}), w - \hat{w}} + \frac{L}{2}\norms{w - \hat{w}}^2, \quad \forall w, \hat{w} \in \dom{F}.
\end{equation}
Apart from Assumption~\ref{ass_basic}, we also require the following one in most of our results.
%%% Assumption 2.
\begin{ass}\label{ass_general_bounded_variance}
There exists two constants $\Theta \in [0, +\infty)$ and $\sigma \in (0, +\infty)$ such that
\begin{equation}\label{eq:general_bounded_variance}
    \frac{1}{n} \sum_{i=1}^n \norms{\nabla f(w; i) - \nabla F(w)}^2 \leq \Theta \| \nabla F(w) \|^2 \ + \ \sigma^2, \quad \forall w\in\dom{F}. 
\end{equation}
\end{ass}
Note that if $\Theta =0$, then Assumption~\ref{ass_general_bounded_variance} reduces to the standard bounded variance assumption $\mathbb{E}_i\big[\norms{ f(w; i) - \nabla{F}(w) }^2\big] \leq \sigma^2$, which is often used in nonconvex problems, see, e.g., \citep{ghadimi2013stochastic,Pham2019}.

%%%% 2.3. Basic Assumptions and Mathematical Tools
\subsection{Technical lemmas for convergence analysis}\label{sec:basic_tools}
The following two lemmas provide key estimates for our convergence analysis in this paper.
We first state them here and provide their proof in Appendix~\ref{subsec:general_results}.

%% Lemma 1.
\begin{lem}\label{lem_general_framework}
Let $\sets{Y_t}_{t\geq 1}$ be a nonnegative sequence in $\R$ and $q$ be a positive integer number.
Let  $\rho > 0$ and $D > 0$ be two given constants and $0 < \eta_t \leq \frac{1}{\rho}$ be given for all $t \geq 1$.
Assume that, for all $t \geq 1$, we have
\begin{equation}\label{eq:induc_ineq}
    Y_{t+1}  \leq (1 - \rho \cdot \eta_t) Y_{t} +  D \cdot \eta_t^{q+1}.  
\end{equation}
If we choose $\eta_t := \frac{q}{\rho(t + \beta)}$ for all $t\geq 1$, where $\beta \geq q-1$, then we have 
\begin{equation}\label{eq_rate_t2}
    Y_{t+1} \leq \frac{\beta \cdots (\beta - q + 1)}{(t+\beta - q+1) \cdots (t+\beta)}Y_1 + \frac{q^{q+1}D\log(t+\beta)}{\rho^{q+1}(t+\beta-q+1)\cdots (t+\beta)}.
\end{equation}
If we choose $\eta_t := \eta \in (0, \rho^{-1})$ for all $t\geq 1$, then we have 
\begin{equation}\label{eq_rate_t3}
 Y_{t+1} \leq (1 - \rho \eta)^{t} Y_1 + \frac{D\eta^{q}[ 1 - (1 - \rho \eta)^{t}]}{\rho} \leq Y_1\exp( -\rho\eta t) + \frac{D\eta^q}{\rho}.
\end{equation}
\end{lem}

%%%  Lemma 2.
\begin{lem}\label{lem_general_framework_02}
Let $\sets{Y_t}_{t\geq 1}$ and $\sets{Z_t}_{t\geq 1}$ be two nonnegative sequences in $\R$ and $m$ and $q$ be two positive numbers such that $q > m$.
For positive constants $\rho, \alpha, \beta, \gamma$, and $D$, assume that
\begin{equation}\label{eq_lem_general_framework_02}
    Y_{t+1}  \leq Y_{t} - \rho \eta_t^{m} \cdot Z_{t} + \eta_t^{q} \cdot D, \quad \text{where}\quad \eta_t := \frac{\gamma}{(t + \beta)^{\alpha}}\quad\text{and}\quad \alpha m \leq \frac{1}{2}.
\end{equation}
Suppose that $Y_t \leq C + H \log (t + \theta)$ for some $C > 0$, $H \geq 0$, $\theta > 0$, and $1 + \theta - \beta > (1-\alpha m)e^{\frac{\alpha m}{1-\alpha m}}$ for all $t\geq 1$,  $($where $e$ is the natural number$)$.
Then, we have
\begin{equation}\label{eq:key_concl}
\hspace{-0.5ex}
\arraycolsep=0.2em
\begin{array}{lcl}
\dfrac{1}{T} {\displaystyle\sum_{t = 1}^{T}} Z_{t} {~}& \leq & \dfrac{1}{T} \bigg[\dfrac{(1+\beta)^{\alpha m} Y_1}{\rho \gamma^m} + \dfrac{C (T - 1 + \beta)^{\alpha m}}{2 \rho \alpha m \gamma^m} + \dfrac{H (T - 1 + \beta)^{\alpha m} \log(T + \theta)}{2\rho \alpha m \gamma^m} \bigg] \vspace{1ex}\\
& & + {~}  \dfrac{D \gamma^{q - m}}{\rho}\cdot \dfrac{A(T)}{T},
\end{array}
\hspace{-2ex}
\end{equation}
where
\begin{equation*}
A(T) := \left\{\begin{array}{ll}
\log(T+\beta) - \log(\beta) &\text{if}~\alpha(q - m) = 1, \vspace{1ex}\\
\frac{(T + \beta)^{1 - \alpha(q - m)}}{1 - \alpha ( q - m)} &\text{otherwise}.
\end{array}\right.
\end{equation*}
\end{lem}

Lemmas~\ref{lem_general_framework} and \ref{lem_general_framework_02} are independent of interest, and play an important role not only for deriving the results for  shuffling-type gradient methods but also for applying our new convergence analysis to the standard SGD algorithm (see Appendix~\ref{sec_standard_sgd} for more details). 

%%% Type of guarantees.
\begin{rem}[\textbf{Types of guarantee}]\label{re:type_of_convergence}
Since we can choose the permutation $\pi^{(t)}$ of Algorithm~\ref{sgd_replacement} either deterministically or randomly, our bounds in the sequel will hold either deterministically or in expectation, respectively.
Without loss of generality, we write these results in the context of expectation taken overall the randomness generated by the algorithm.
\end{rem}

%%%%%%%%%%%%%%%%%%%%%%%%%%%%%%%%%%%%%
%%%% 4. Convergence Analysis for Convex Case.
%%%%%%%%%%%%%%%%%%%%%%%%%%%%%%%%%%%%%
\section{Convergence Analysis for  Convex Case}\label{sec_analysis_01}
In this section, we mainly consider two cases. 
In the first case, we only assume that the sum function $F$ is strongly convex while some components $f(\cdot; i)$ for $i\in [n]$ are not necessarily convex.
In the second case, we assume that $F$ is strongly convex and each $f(\cdot; i)$ for $i\in [n]$ is also convex.
Let us state the strong convexity assumption of $F$ as follows.

%%% Assumption A.4.
\begin{ass}[$\mu$-strong convexity]
\label{ass_stronglyconvex}
The objective function $F$ of \eqref{ERM_problem_01}  is $\mu$-strongly convex on $\dom{F}$, i.e. there exists a constant $\mu \in (0, +\infty)$ such that
\begin{equation}\label{eq:stronglyconvex_00}
F(w) \geq F(\hat{w}) + \iprods{\nabla F(\hat{w}), w - \hat{w}} + \frac{\mu}{2}\norms{w - \hat{w}}^2,\quad \forall w, \hat{w} \in \dom{F}.
\end{equation}
\end{ass}
It is well-known from the literature \citep{nesterov2004,Bottou2018} that Assumption \ref{ass_stronglyconvex} implies the existence and uniqueness of the optimal solution $w_{*}$ of \eqref{ERM_problem_01}, and
\begin{equation}\label{eq:stronglyconvex}
\frac{\mu}{2}\norms{w - w_{*}}^2 \leq  F(w) - F(w_{*})  \leq \frac{1}{2\mu}\norms{\nabla F(w)}^2, \quad \forall w \in \dom{F}. 
\end{equation}
As we have mentioned above, Assumption~\ref{ass_stronglyconvex} only requires the sum function $F$ to be strongly convex, but some components $f(\cdot; i)$ can even be nonconvex.

Under Assumption~\ref{ass_stronglyconvex}, since problem \eqref{ERM_problem_01} has a unique optimal solution $w_{*}$, we introduce the following variance of $F$ at $w_{*}$:
\begin{equation}\label{defn_finite}
\sigma_{*}^2 := \frac{1}{n}\sum_{i=1}^{n}\Vert \nabla{f}(w_{*}; i) \Vert^2 \in [0, +\infty).
\end{equation}
Now, we state the convergence of Algorithm~\ref{sgd_replacement} with a constant learning rate in Theorem~\ref{thm_main_result_02_const}, whose proof is given in Appendix~\ref{subsec:appendix_Th12_proof}.
%%%
%%% Theorem 1.
\begin{thm}\label{thm_main_result_02_const}
Assume that Assumptions \ref{ass_basic} and \ref{ass_stronglyconvex} hold $($but some $f(\cdot; i)$ for $i\in [n]$ are not necessarily convex$)$, $\sigma_{*}^2$ is defined by \eqref{defn_finite}, and $\kappa := \frac{L}{\mu}$ is the condition number of $F$.
Let $\sets{ w_i^{(t)}}_{t=1}^T$ be generated by Algorithm~\ref{sgd_replacement} after $T$ epochs with a constant stepsize $\eta_i^{(t)} := \frac{6\log(T)}{\mu n T}$ and any shuffling strategy $\pi^{(t)}$. 
Then, for $T \geq 1$ such that $T  \geq 12\kappa^2\log(T)$, we have
\begin{align}\label{eq:scvx_bound1_const}
\hspace{-0.0ex}
\mathbb{E}\big[ F(\tilde{w}_T) - F(w_{*}) \big] & \leq \frac{1}{T^2} \bigg[  \big( F(\tilde{w}_0) - F(w_{*}) \big) \, + \, \frac{54(\mu^2 + L^2)\sigma_{*}^2 \log(T)^2}{\mu^3} \bigg] \nonumber\\
& =  \Ocal\left(\frac{\log(T)^2}{T^2}\right).
\hspace{-1ex}
\end{align}
Assume additionally that $\pi^{(t)}$ is sampled uniformly at random without replacement from $[n]$ and Assumption~\ref{ass_general_bounded_variance} holds.
Then, by choosing  a learning rate $\eta_i^{(t)}  := \frac{4\log(\sqrt{n}T)}{\mu n T}$ for $T \geq 1$ such that $T \geq \frac{8L\sqrt{\Theta/n+1}}{\mu^2}\log(\sqrt{n}T)$, we have
\begin{equation}\label{eq:scvx_bound2_const0}
\mathbb{E}\big[ F(\tilde{w}_t) - F(w_{*}) \big]  \leq   \frac{1}{nT^2}\bigg[  \big[ F(\tilde{w}_0) - F(w_{*}) \big] + \frac{2L^2\sigma^2 \log(\sqrt{n}T)^2}{\mu^3} \bigg].
\end{equation}
Consequently, the convergence rate of $\big\{ \mathbb{E}\big[ F(\tilde{w}_T) - F(w_{*}) \big] \big\}$ in this case is $\Ocal\left(\frac{\log(T)^2}{n T^2}\right)$.

Alternatively, assume additionally that $\pi^{(t)}$ is sampled uniformly at random without replacement from $[n]$ and each $f(\cdot; i)$ is convex for $i \in [n]$.
Then, by choosing  a constant learning rate $\eta_i^{(t)}  := \frac{2\log(\sqrt{n}T)}{\mu n T}$ for $T \geq 1$ such that $\log(T\sqrt{n}) \leq \frac{T}{2}\min\big\{ 1, \frac{\sqrt{5}-1}{\kappa}\big\}$, we have
\begin{equation}\label{eq:scvx_bound2_const}
\mathbb{E}\big[ F(\tilde{w}_t) - F(w_{*}) \big]  \leq  \frac{L}{2} \mathbb{E}\big[ \norms{\tilde{w}_t - w_{*}}^2 \big] \leq  \frac{L}{2nT^2}\bigg[ \mathbb{E}\big[ \norms{\tilde{w}_{0} - w_{*}}^2 \big] + \frac{8L\sigma_{*}^2 \log(\sqrt{n}T)^2}{3 \mu^3} \bigg].
\end{equation}
Consequently, the convergence rate of both $\big\{ \mathbb{E}\big[ F(\tilde{w}_T) - F(w_{*}) \big] \big\}$ and $\big\{ \mathbb{E}\big[ \norms{\tilde{w}_T - w_{*}}^2 \big] \big\}$ in this case is $\Ocal\left(\frac{\log(T)^2}{n T^2}\right)$.
\end{thm}

The condition $\log(T\sqrt{n}) \leq \frac{T}{2}\min\big\{ 1, \frac{\sqrt{5}-1}{\kappa}\big\}$  on $T$ of \eqref{eq:scvx_bound2_const} shows that $\frac{T}{\log(T\sqrt{n})} \geq \Ocal(\kappa)$, where $\kappa := \frac{L}{\mu}$ is the condition number of $F$.
This condition has been shown in previous works, and aligns with recent results in \citep{ahn2020sgd,mishchenko2020random}.
However, unlike \citep{mishchenko2020random}, we have new results stated in \eqref{eq:scvx_bound1_const} and \eqref{eq:scvx_bound2_const0}.

Next, we prove the following result for the strongly convex case using diminishing learning rates.
The detailed proof of this theorem is given in Appendix~\ref{subsec:appendix_Th12_proof}.

%%% Theorem 2.
\begin{thm}\label{thm_main_result_02}
Assume that Assumptions \ref{ass_basic} and \ref{ass_stronglyconvex} hold $($but some $f(\cdot; i)$ for $i\in [n]$ are not necessarily convex$)$, $\sigma_{*}^2$ is defined by \eqref{defn_finite}, and $\kappa := \frac{L}{\mu}$ is the condition number of $F$.
Let $\sets{ w_i^{(t)}}$ be generated by Algorithm~\ref{sgd_replacement} with $\eta_i^{(t)} := \frac{\eta_t}{n}$ and any shuffling strategy $\pi^{(t)}$ for solving \eqref{ERM_problem_01}. 
Let  $\eta_t$ be updated by  $\eta_t  := \frac{6}{\mu(t + \beta)}$ for all $t \geq 1$, where $\beta \geq 12\kappa^2 - 1$. 
Then, we have
\begin{align}\label{eq:scvx_bound1}
\hspace{-0.5ex}
\arraycolsep=0.2em
\mathbb{E}\big[ F(\tilde{w}_t) - F(w_{*}) \big] & \leq  \dfrac{\beta(\beta - 1)}{(t+\beta)(t+\beta-1)}\bigg[ \big( F(\tilde{w}_0) - F(w_{*}) \big) + \dfrac{216(L^2 + \mu^2)\sigma_{*}^2 \log(t + \beta)}{\mu^3\beta(\beta - 1)} \bigg] \nonumber\\
& =  \Ocal\left(\dfrac{\log(t)}{t^2}\right).
\hspace{-3ex}
\end{align}
If, additionally, $\pi^{(t)}$ is sampled uniformly at random without replacement from $[n]$, $f(\cdot; i)$ $($for all $i \in [n]$$)$ are convex, and $L \leq \frac{\sqrt{5} - 1}{2}$, then by choosing  $\eta_t  := \frac{2}{\mu(t + 1 + 1/n)}$ for $t \geq 1$, we have 
\begin{equation}\label{eq:scvx_bound2}
\hspace{-0.0ex}
\arraycolsep=0.1em
\begin{array}{lcl}
\mathbb{E}\big[ F(\tilde{w}_t) - F(w_{*}) \big] & \leq & \dfrac{L}{2} \mathbb{E}\big[ \norms{\tilde{w}_t - w_{*}}^2 \big] \vspace{0.5ex}\\
& \leq & \dfrac{2L}{n(t+1/n)(t+ 1/n + 1)}\bigg[ \mathbb{E}\big[ \norms{\tilde{w}_0 - w_{*}}^2 \big] + \dfrac{L\sigma_{*}^2\log(t + \beta)}{3\mu^3} \bigg] \vspace{1ex}\\
& = & \Ocal\left(\dfrac{\log(t)}{nt^2}\right).
\end{array}
\hspace{-3ex}
\end{equation}
\end{thm}
The condition $L \leq \frac{\sqrt{5} - 1}{2}$ in Theorem~\ref{thm_main_result_02} holds without loss of generality.
Indeed, if it does not hold, then by rescaling $f(\cdot; i)$ by $\frac{\sqrt{5} - 1}{2L} \cdot f(\cdot; i)$, we obtain this condition.
Clearly, with a randomized reshuffling strategy, the convergence rate of Algorithm~\ref{sgd_replacement} is better than a general shuffling one.
This rate nearly matches the lower bound in previous works, e.g., in \citep{nagaraj2019sgd,rajput2020closing}.
While our proof of \eqref{eq:scvx_bound1} is new, the proof of \eqref{eq:scvx_bound2} is inspired by \citep{mishchenko2020random}, but it is rather different (see Lemma~\ref{le:scvx_key_bound2} in the appendix). 
Moreover, our bound \eqref{eq:scvx_bound2} is established for a diminishing stepsize $\eta_t$.

Since the total number of iterations is $K := nT$, if we write the convergence rates in terms of $K$, then for any shuffling strategy, we have $\mathbb{E}\big[ F(\tilde{w}_T) - F(w_{*})\big] \leq \Ocal\left( n^2 K^{-2}\right)$.
However, for a randomized reshuffling one, we have $\mathbb{E}\big[ F(\tilde{w}_T) - F(w_{*})\big] \leq \Ocal\left( nK^{-2} \right)$, matching the result in \citep{haochen2018random,mishchenko2020random}.
As mentioned earlier, our assumptions for \eqref{eq:scvx_bound2} are as in \citep{mishchenko2020random} and  weaker than those in \citep{haochen2018random,nagaraj2019sgd}. Here, we use diminishing learning rates in Theorem~\ref{thm_main_result_02} instead of constant ones as in \citep{haochen2018random,mishchenko2020random,nagaraj2019sgd}.
Note that  Algorithm~\ref{sgd_replacement} covers much broader class of algorithms compared to existing methods in the literature.

%%% Remark: Non-strongly convex case
\begin{remark}[Non-strongly convex case]\label{re:non-strongly_convex}
Similar to \citep{mishchenko2020random}, we can use \eqref{eq:scvx_key_bound2} of Lemma~\ref{le:scvx_key_bound2} to prove the following convergence rate for Algorithm~\ref{sgd_replacement} under only convexity of $f(\cdot; i)$ for all $i\in [n]$.
We state this result as follows without proof.
\begin{equation*}
\mathbb{E}\big[F(\hat{w}_T) - F(w_{*}) \big] \leq \frac{1}{\Sigma_T} \sum_{t=1}^T\eta_t \mathbb{E}\big[ F(\tilde{w}_{t-1}) - F(w_{*}) \big] \leq \frac{1}{2\Sigma_T}\norms{\tilde{w}_0 - w_{*}}^2 + \frac{L\sigma_{*}^2}{3n\Sigma_T} \cdot \sum_{t=1}^T\eta_t^3,
\end{equation*}
where $\Sigma_T := \sum_{t=1}^T\eta_t$ and $\hat{w}_T := \frac{1}{\Sigma_T}\sum_{t=1}^T \eta_t \tilde{w}_{t-1}$ is a weighted averaging sequence.
\begin{compactitem}
\item If we choose $\eta := \frac{\gamma n^{1/3}}{T^{1/3}} \leq \frac{1}{2L}$ as a constant learning rate, then we have
\begin{equation*}
\mathbb{E}\big[ F(\hat{w}_T) - F(w_{*}) \big] \leq \frac{1}{n^{1/3} T^{2/3}} \cdot \bigg[ \frac{1}{2\gamma} \norms{\tilde{w}_0 - w_{*}}^2 + \frac{\gamma^2 L\sigma_{*}^2}{3} \bigg].
\end{equation*}
\item If we choose $\eta_t := \frac{\gamma n^{1/3}}{(t + \beta)^{1/3}}$ for $1 \leq t \leq T$ as a diminishing learning rate  for some $\gamma > 0$ and $\beta \geq 1$ such that $\beta \geq 8L^3\gamma^3 n - 1$, then we have
\begin{equation*}
\mathbb{E}\big[ F(\hat{w}_T) - F(w_{*}) \big] \leq \frac{1}{n^{1/3} T^{2/3}} \cdot \bigg[ \frac{1}{2\gamma} \norms{\tilde{w}_0 - w_{*}}^2 + \frac{\gamma^2 L\sigma_{*}^2 \log(T + \beta)}{3} \bigg].
\end{equation*}
\end{compactitem}
The constant learning rate case stated here is discussed in \citep{mishchenko2020random}.
However, we add a new statement on diminishing learning rate that may be practically favorable. 
\end{remark}

%%%%%%%%%%%%%%%%%%%%%%%%%%%%%%%%%%%%%%%%
%%%% C. Convergence Analysis for nonconvex Case.
%%%%%%%%%%%%%%%%%%%%%%%%%%%%%%%%%%%%%%%%
\section{Convergence Analysis for Nonconvex Case}\label{sec_analysis_02}
We now provide convergence analysis for Algorithm~\ref{sgd_replacement} to solve nonconvex smooth instances of \eqref{ERM_problem_01}.
We consider two cases in Subsection~\ref{subsec:general_ncvx} and \ref{subsec:gradient_dominance_ncvx}, respectively.

%%% 2.1. The general case
\subsection{The general case}\label{subsec:general_ncvx}
We state our first result on the nonconvex case, whose proof is given in  Appendix \ref{apdx:proof_Th1_Corr_12}.

%%% Theorem 3.
\begin{thm}\label{thm_nonconvex_01}
Suppose that Assumptions~\ref{ass_basic}  and \ref{ass_general_bounded_variance} hold for \eqref{ERM_problem_01}. 
Let $\sets{\tilde{w}_t}_{t = 1}^T$ be generated by Algorithm~\ref{sgd_replacement} with any shuffling strategy for $\pi^{(t)}$ and any learning rate $\eta_i^{(t)} = \frac{\eta_t}{n}  = \frac{\eta}{n}$ such that $0 < \eta \leq \frac{1}{L\sqrt{2(3\Theta+2)}}$.
Then, we have
\begin{equation}\label{eq_thm_nonconvex_01}
    \frac{1}{T} \sum_{t=1}^T \mathbb{E}\big[ \norms{\nabla F ( \tilde{w}_{t-1} )}^2  \big] \ \leq \ \frac{4}{T \eta} \big[ F ( \tilde{w}_{0} ) - F_{*} \big] \ + \  6 L^2 \sigma^2 \eta^2.
\end{equation}
If, additionally, $\pi^{(t)}$ is sampled uniformly at random without replacement from $[n]$ and $\eta_t$ is chosen such that $0 < \eta \leq \frac{1}{L\sqrt{2(\Theta/n + 1)}}$, then we have
\begin{equation}\label{eq_thm_nonconvex_01_RR}
    \frac{1}{T} \sum_{t=1}^T  \mathbb{E}\big[ \norms{\nabla F ( \tilde{w}_{t-1} )}^2 \big] \ \leq \ \frac{4}{T \eta} \big[ F ( \tilde{w}_{0} ) - F_{*} \big] \ + \  \frac{4L^2 \sigma^2 \eta^2}{n}.
\end{equation}
\end{thm}
If $L$, $\sigma$, and $\Theta$ are known, then we can choose the following learning rate to get a concrete bound as stated in Corollary~\ref{cor_nonconvex_01}, whose proof can be found in Appendix \ref{apdx:proof_Th1_Corr_12}.

%% Corollary 1.
\begin{cor}\label{cor_nonconvex_01}
Let $\sets{\tilde{w}_t}_{t=1}^T$ be generated by Algorithm~\ref{sgd_replacement} for solving \eqref{ERM_problem_01}.
For a given $\epsilon$ such that $0 < \epsilon \leq 2\sigma^2$, under the same conditions as of \eqref{eq_thm_nonconvex_01} in Theorem~\ref{thm_nonconvex_01}, if we choose a constant learning rate $\eta := \frac{\sqrt{\epsilon}}{2L \sigma \sqrt{3\Theta + 2}}$, then to guarantee 
$\frac{1}{T} \sum_{t=1}^T \mathbb{E}\big[ \norms{ \nabla F ( \tilde{w}_{t-1} ) }^2 \big] \leq \epsilon$ for \eqref{ERM_problem_01}, it requires at most $T := \Big\lfloor \frac{16 L \sigma (3\Theta+2)^{3/2} [ F ( \tilde{w}_{0} ) - F_{*} ]}{(6\Theta + 1)} \cdot \frac{1}{\epsilon^{3/2}} \Big\rfloor$ outer iterations. 
As a result, the total number of gradient evaluations  is at most $\mathcal{T}_{\nabla{f}} := \Big\lfloor \frac{16 L \sigma (3\Theta+2)^{3/2} [ F ( \tilde{w}_{0} ) - F_{*} ]}{(6\Theta + 1)} \cdot \frac{n}{\epsilon^{3/2}} \Big\rfloor$. 

If, in addition, $\pi^{(t)}$ is sampled uniformly at random without replacement from $[n]$, then by choosing $\eta := \frac{\sqrt{n \epsilon}}{2L\sigma\sqrt{2(\Theta/n + 1)}}$ for $0 < \epsilon \leq \frac{4\sigma^2}{n}$, to guarantee $\frac{1}{T} \sum_{t=1}^T  \mathbb{E}\big[ \norms{ \nabla F ( \tilde{w}_{t-1} ) }^2\big] \leq \epsilon$, it requires at most $T := \left\lfloor  \frac{16(\Theta/n + 1)^{3/2} [ F ( \tilde{w}_{0} ) - F_{*} ] }{(2\Theta/n +1)} \cdot \frac{L\sigma }{\sqrt{n} \epsilon^{3/2}} \right\rfloor $ outer iterations.
Consequently, the total number of gradient evaluations is at most $\mathcal{T}_{\nabla{f}} := \left\lfloor  \frac{16(\Theta/n + 1)^{3/2} [ F ( \tilde{w}_{0} ) - F_{*} ] }{(2\Theta/n +1)} \cdot \frac{ L\sigma \sqrt{n}}{\varepsilon^{3/2}} \right\rfloor$.
\end{cor}

%%% Remark 1.
\begin{rem}\label{rem_nonconvex}
\normalfont
To obtain the same bound, the total complexity of the standard SGD is $\Ocal( L_F \sigma_S^2 \varepsilon^{-2})$ for solving \eqref{Expectation_problem_01} under the bounded variance $\mathbb{E}_i[ \norms{ \nabla f(w; i) - \nabla F(w) }^2 ] \leq \sigma_S^2$ for some $\sigma_S > 0$ and the $L_F$-smoothness of $F$. 
Note that the standard SGD only requires $F$ to be $L_F$-smooth while we impose the smoothness on individual realizations. 
Therefore, $L_F$ and $L$ may be different \citep{ghadimi2013stochastic}. 
For a rough comparison, if $n < \Ocal \Big( \frac{L_F \sigma_S^2}{L \sigma} \cdot \frac{1}{\varepsilon^{1/2}} \Big)$, then Algorithm~\ref{sgd_replacement} with any shuffling strategy seems to have advantages over the standard  SGD method in the nonconvex setting. 
In addition, if a randomized reshuffling strategy is used, then for $n \leq \frac{4\sigma^2}{\varepsilon}$, the complexity of  Algorithm~\ref{sgd_replacement}  is $\Ocal\left( \frac{L\sigma\sqrt{n}}{\varepsilon^{3/2}} \right)$, which is better than SGD by a factor $\frac{\sigma}{\varepsilon^{1/3}}$.
From this point of view, it seems that Algorithm~\ref{sgd_replacement} with a general shuffling strategy is theoretically less efficient than SGD when a low accuracy solution is desirable (i.e. $\varepsilon$ is not too small) or when $n \gg 1$.
However, we believe  that Algorithm~\ref{sgd_replacement} allows more flexible strategy to choose $f(\cdot; i)$ rather than that i.i.d. sampling.
For instance, choosing a randomized reshuffling strategy significant improves the complexity of Algorithm~\ref{sgd_replacement}.
We also note that our convergence guarantee is completely different from \citep{meng2019convergence} as mentioned earlier.
Nevertheless, Assumptions~\ref{ass_basic} and \ref{ass_general_bounded_variance} are very standard and hold for various applications in machine learning.
\end{rem}

Let us propose another choice of the learning rate $\eta_t$ in the following result, who proof can also be founded in Appendix~\ref{apdx:proof_Th1_Corr_12}.

% % Corollary 2.
\begin{cor}\label{cor_nonconvex_01_01}
Let $\sets{\tilde{w}_t}_{t=1}^T$ be generated by Algorithm~\ref{sgd_replacement}.
Under the same conditions as of \eqref{eq_thm_nonconvex_01} in Theorem~\ref{thm_nonconvex_01}, if we choose a constant learning rate $\eta := \frac{\gamma}{T^{1/3}}$ for some $\gamma > 0$ and $T \geq 1$ such that $T^{1/3} \geq \gamma L\sqrt{2(3\Theta + 2)}$, then we have
\begin{equation}\label{eq_nonconvex_01}
\frac{1}{T} \sum_{t=1}^T \mathbb{E} \big[ \Vert \nabla F ( \tilde{w}_{t-1} ) \Vert^2 \big]  \leq \frac{1}{T^{2/3}}\left[ \frac{4 \big( F ( \tilde{w}_{0} ) - F_{*} \big) }{\gamma} \ + \ 6 L^2 \sigma^2 \gamma^2 \right] = \Ocal\left( \frac{1}{T^{2/3}} \right).
\end{equation}
If, in addition, $\pi^{(t)}$ is sampled uniformly at random without replacement from $[n]$, then by choosing $\eta := \frac{\gamma n^{1/3}}{T^{1/3}}$ for some $\gamma > 0$ such that $T^{1/3} \geq \gamma Ln^{1/3}\sqrt{2(\Theta/n+1)}$, we have
\begin{equation}\label{eq_nonconvex_01b}
    \frac{1}{T} \sum_{t=1}^T  \mathbb{E}\big[ \norms{\nabla F ( \tilde{w}_{t-1} )}^2 \big] \leq  \frac{4}{n^{1/3}T^{2/3}} \left[  \frac{ \left( F ( \tilde{w}_{0} ) - F_{*} \right)}{\gamma} \ + \ L^2 \sigma^2 \gamma^2 \right] = \Ocal\left(\frac{1}{n^{1/3}T^{2/3}} \right).
\end{equation}
\end{cor}

Since the total number of iterations is $K := nT$, if we express \eqref{eq_nonconvex_01} in terms of $K$, then we have $\frac{1}{T} \sum_{t=1}^T \mathbb{E}[ \Vert \nabla F ( \tilde{w}_{t-1} ) \Vert^2 ] \leq \frac{n^{2/3} \Delta_0 }{K^{2/3}}$, where $\Delta_0 := F ( \tilde{w}_{0} ) - F_{*}$.
Alternatively, we can express \eqref{eq_nonconvex_01b} in terms of $K$ as $\frac{1}{T} \sum_{t=1}^T \mathbb{E}[\Vert \nabla F ( \tilde{w}_{t-1} ) \Vert^2 ] \leq \frac{n^{1/3} \Delta_0}{K^{2/3}}$.
%Clearly, while the former rate remains suboptimal compared to the best known complexity of variance reduction SGD schemes, the latter one is better than recent complexity bounds $\Ocal(n^{1/2}K^{-2/3})$ of variance reduction SGD-type algorithms, e.g., in \citep{Pham2019} by a factor $n^{1/6}$, but when $T$ is chosen such that $T \geq \Ocal(n)$. 

The following theorem characterizes an asymptotic convergence for general diminishing stepsizes, whose proof can be found in Appendix~\ref{apdx:proof_Th2_and_3}. 

%%% Theorem 4.
\begin{thm}\label{thm_nonconvex_02}
Suppose that Assumptions \ref{ass_basic} and \ref{ass_general_bounded_variance} hold for \eqref{ERM_problem_01}. 
Let $\sets{\tilde{w}_t}_{t\geq 1}$ be generated by Algorithm~\ref{sgd_replacement} with diminishing learning rate $\eta_i^{(t)} = \frac{\eta_t}{n}$ such that $\sum_{t=1}^{\infty} \eta_t = \infty$ and $\sum_{t=1}^{\infty} \eta_t^3 < \infty$. 
Then, w.p.1. $($i.e. almost surely$)$, we have $\liminf\limits_{t\to\infty} \norms{ \nabla F ( \tilde{w}_{t-1} )}^2 = 0$. 
\end{thm}

Now, if we vary $\eta_t$, then Theorem~\ref{thm_nonconvex_04} shows how $\eta_t$ affects our rates (see Appendix~\ref{subsec:new_results}). 

%%% Theorem 5.
\begin{thm}\label{thm_nonconvex_04}
Suppose that Assumptions \ref{ass_basic} and \ref{ass_general_bounded_variance} hold for \eqref{ERM_problem_01}. 
Let $\sets{\tilde{w}_t}_{t=1}^T$ be  generated by Algorithm~\ref{sgd_replacement} with  $\eta_i^{(t)} = \frac{\eta_t}{n}$, where $\eta_t := \frac{\gamma}{(t + \beta)^\alpha}$, for some $\gamma > 0$, $\beta > 0$, and $\frac{1}{3} < \alpha < 1$. 
If a generic shuffling strategy is used, then we let $D := \frac{3}{2}L^2\sigma^2$ and assume that $\gamma L\sqrt{2(3\Theta + 2)} \leq (\beta + 1)^{\alpha}$.
Otherwise, if a uniformly randomized reshuffling strategy is used, then we set $D := \frac{L^2\sigma^2}{n}$ and assume that $\gamma L\sqrt{2(\Theta/n + 1)} \leq (\beta + 1)^{\alpha}$.
Let $C := [F( \tilde{w}_{0} ) -F_{*}] + \frac{D \gamma^3}{(3\alpha - 1)\beta^{3\alpha - 1}} > 0$ be a given constant. 
Then, the following statements hold:
\begin{compactitem}
    \item If $\alpha = \frac{1}{2}$, then the following bound holds:
    \begin{equation*}
    \arraycolsep = 0.2em
        \begin{array}{lcl}
        \dfrac{1}{T} \displaystyle\sum_{t = 1}^{T}\mathbb{E}\big[ \norms{ \nabla F( \tilde{w}_{t-1} )  }^2 \big]
         & \leq & \dfrac{4 (1+\beta)^{1/2} \left[ F( \tilde{w}_{0} ) -F_{*} \right]}{ \gamma} \cdot \dfrac{1}{T}  \ + \ \dfrac{4 C}{ \gamma} \cdot \dfrac{(T - 1 + \beta)^{1/2}}{T}  \vspace{1ex}\\ 
         && + {~} 4D \gamma^{2} \cdot \dfrac{\log(T+\beta) - \log(\beta)}{T}. 
         \end{array}
    \end{equation*}
    \item If $\alpha \neq \frac{1}{2}$, then we have
    \begin{equation*}
    \arraycolsep = 0.2em
        \begin{array}{lcl}
        \dfrac{1}{T} \displaystyle \sum_{t = 1}^{T} \mathbb{E}\big[ \norms{ \nabla F( \tilde{w}_{t-1} )  }^2  \big]
        & \leq & \dfrac{4 (1+\beta)^{\alpha} \left[ F( \tilde{w}_{0} ) - F_{*} \right]}{ \gamma} \cdot \dfrac{1}{T} \ + \ \dfrac{2 C}{\alpha \gamma} \cdot \dfrac{(T - 1 + \beta)^{\alpha}}{T}  \vspace{1ex}\\ 
        && + {~} \dfrac{4D \gamma^{2}}{(1 - 2\alpha)} \cdot \dfrac{(T + \beta)^{1 - 2\alpha}}{T}. 
         \end{array}
    \end{equation*}
\end{compactitem}
If a uniformly randomized reshuffling strategy is used, then by replacing $\gamma$ by $n^{1/3}\gamma$, we have
\begin{compactitem}
\item If $\alpha = \frac{1}{2}$, then the convergence rate of Algorithm~\ref{sgd_replacement} is $\Ocal\left(\frac{1 \ + \ T^{1/2} \ + \ \log(T)}{n^{1/3}T}\right)$.
\item If $\alpha \neq \frac{1}{2}$, then the convergence rate of Algorithm~\ref{sgd_replacement} is $\Ocal\left(\frac{1 \ + \ T^{1-2\alpha} \ + \ \log(T)}{n^{1/3}T}\right)$.
\end{compactitem}
\end{thm}

%%% Remark 4.
\begin{rem}\label{rem_thm_nonconvex_04}
\normalfont
In Theorem~\ref{thm_nonconvex_04}, if we choose $\alpha := \frac{1}{3} + \delta$ for some $0 < \delta < \frac{1}{6}$, then we have
\begin{equation*}
\arraycolsep = 0.2em
\begin{array}{lcl}
    \dfrac{1}{T} \displaystyle\sum_{t = 1}^{T}  \mathbb{E}\big[ \norms{ \nabla F( \tilde{w}_{t-1} )  }^2 \big]
     & \leq & \dfrac{4 (1+\beta)^{\frac{1}{3} + \delta} \left[ F( \tilde{w}_{0} ) - F_{*} \right] }{ \gamma} \cdot \dfrac{1}{T} \ + \ \dfrac{2C}{\gamma (\frac{1}{3} + \delta)} \cdot \dfrac{(T - 1 + \beta)^{\frac{1}{3} + \delta}}{T}  \vspace{1ex}\\
     && + {~} \dfrac{12D \gamma^{2}}{1 - 6\delta} \cdot \dfrac{(T + \beta)^{\frac{1}{3} - 2\delta}}{T}, 
\end{array}    
\end{equation*}
where $C := [F( \tilde{w}_{0} ) -F_{*}] + \frac{D \gamma^3}{3 \delta \beta^{3 \delta}}$. 
Hence, the convergence rate of Algorithm~\ref{sgd_replacement} is $\Ocal\big( T^{-(\frac{2}{3} - \delta)} \big)$ in general.
If a randomized reshuffling strategy is used, then this rate is $\Ocal\left( n^{-1/3}T^{-(\frac{2}{3} - \delta)} \right)$.
\end{rem}

For the extreme case $\alpha := \frac{1}{3}$, we have the following result (see Appendix~\ref{subsec:new_results}). 

% % Theorem 6.
\begin{thm}\label{thm_nonconvex_04_02}
Suppose that Assumptions \ref{ass_basic} and \ref{ass_general_bounded_variance} hold for \eqref{ERM_problem_01}. 
Let $\sets{\tilde{w}_t}_{t=1}^T$ be generated by Algorithm~\ref{sgd_replacement} with  $\eta_i^{(t)} := \frac{\eta_t}{n}$, where $\eta_t := \frac{\gamma}{(t + \beta)^{1/3}}$ for some $\gamma > 0$ and $\beta > 0$. 
If any shuffling strategy is used, then let $D := \frac{3}{2}L^2\sigma^2$ and assume that $\gamma L\sqrt{2(3\Theta + 2)} \leq (\beta + 1)^{1/3}$.
Otherwise,  if a uniformly randomized reshuffling strategy is used, then let $D := \frac{L^2\sigma^2}{n}$ and assume that $\gamma L\sqrt{2(\Theta/n + 1)} \leq (\beta + 1)^{1/3}$.
Let $C := [F( \tilde{w}_{0} ) -F_{*}] + \frac{D \gamma^3}{(1 + \beta)} > 0$ be a given constant. 
Then the following bound holds:
\begin{equation*}
    \arraycolsep=0.2em
     \begin{array}{lcl}
     \dfrac{1}{T} \displaystyle\sum_{t = 1}^{T}  \mathbb{E}\big[ \norms{ \nabla F( \tilde{w}_{t-1} )  }^2  \big] & \leq & 
     \dfrac{4 (1+\beta)^{1/3} \left[ F( \tilde{w}_{0} ) -F_{*} \right] }{ \gamma} \cdot \dfrac{1}{T} \ + \ 12D \gamma^{2} \cdot \dfrac{(T + \beta)^{1/3}}{T}  \vspace{1ex}\\ 
     && + {~} 6D \gamma^{2} \cdot \dfrac{(T - 1 + \beta)^{1/3} \log(T + 1 + \beta)}{T}  \ + \ \dfrac{6 C}{ \gamma} \cdot \dfrac{(T - 1 + \beta)^{1/3}}{T}.
    \end{array}
\end{equation*}
Consequently, the convergence rate of Algorithm~\ref{sgd_replacement} is  $\Ocal\left(\frac{1}{T} + \frac{1}{T^{2/3}} + \frac{\log(T)}{T^{2/3}}\right)$.
In addition, if a uniformly randomized reshuffling strategy is used, then by replacing $\gamma$ by $n^{1/3}\gamma$ into the above estimate, the convergence rate of Algorithm~\ref{sgd_replacement} is $\Ocal\left(\frac{1}{n^{1/3}T} + \frac{1}{n^{1/3}T^{2/3}} + \frac{\log(T)}{n^{1/3}T^{2/3}}\right)$, provided that $\gamma Ln^{1/3}\sqrt{2(\Theta/n + 1)} \leq (\beta + 1)^{1/3}$.
\end{thm}

%%% Remark 2.
\begin{rem}
\normalfont
The choice of the learning rates in Theorems~\ref{thm_nonconvex_04} and \ref{thm_nonconvex_04_02} is not necessarily dependent on the smoothness constant $L$ and $\Theta$ of Assumption~\ref{ass_general_bounded_variance}. 
Since $\eta_t$ is diminishing and both $L$ and $\Theta$ are finite, by choosing $\beta$ large, the condition $\eta_t  \leq \frac{1}{L\sqrt{2(3\Theta + 2)}}$ or $\eta_t  \leq \frac{1}{L\sqrt{2(\Theta/n + 1)}}$ automatically holds. 
Therefore, the choices of $\gamma$, $\beta$, and $\alpha$ makes our results more flexible to adjust in particular practical implementation.
We believe that a diminishing or scheduled diminishing learning rate is more favorable in practice than a constant one.
\end{rem}

%%% 3.2. Convergence under gradient dominance
\subsection{Convergence under gradient dominance}\label{subsec:gradient_dominance_ncvx}
We can further improve convergence rates of Algorithm~\ref{sgd_replacement} in the nonconvex case by imposing the following gradient dominance condition. 
\begin{ass}\label{ass_PLcond}
A function $F$ is said to be $\tau$-gradient dominant if there exists a constant $\tau \in (0, +\infty)$ such that 
\begin{equation}\label{eq_grad_dominant}
    F(w) - F_{*} \leq \tau \| \nabla F(w) \|^2, \quad \forall w\in\dom{F},
\end{equation}
where $F_{*} := \inf_{w\in\R^d}F(w)$.  
\end{ass}

This assumption is well-known and widely used in the literature, see, e.g., \citep{polyak_condition,nesterov2006cubic,Polyak1964}.
It is also weaker than a strong convexity assumption. 
When $F_{*}$ is achievable (i.e. $F_{*} = F(w_{*})$ for some $w_{*}$), then we can observe that every stationary point $w_{*}$ of the $\tau$-gradient dominant function $F$ is a global minimizer. 
However, such a function $F$ is not necessarily convex. 

The following theorem states the convergence rate of Algorithm~\ref{sgd_replacement} under gradient dominance, whose proof is deferred to Appendix~\ref{apdx:proof_of_Th5_Cor34}.

%%% Theorem 3.
\begin{thm}\label{thm_nonconvex_03}
Suppose that Assumptions \ref{ass_basic}, \ref{ass_general_bounded_variance}, and \ref{ass_PLcond} hold for \eqref{ERM_problem_01}. 
Let $\sets{ w_i^{(t)}}$ be generated by Algorithm~\ref{sgd_replacement}  for solving \eqref{ERM_problem_01} using $\eta_i^{(t)} := \frac{\eta_t}{n}$ and any shuffling strategy. 
Let  $\eta_t$ be updated as  $\eta_t  := \frac{2}{t + \beta}$ for some $\beta \geq \max\sets{ 2L\sqrt{2(3\Theta + 2)} - 1, 1}$.
Then, for all $t \geq 1$, we have
\begin{equation}\label{eq:nonconvex_03_general}
    \mathbb{E}\big[ F(\tilde{w}_t) - F_{*} \big]  \leq  \frac{1}{(t+\beta - 1)(t + \beta)}\Big[ \beta(\beta-1)\big( F(\tilde{w}_0) - F_{*}\big) \ + \ 768 \cdot \tau^3 L^2\sigma^2 \log(t + \beta)  \Big].
\end{equation}
Consequently, the convergence rate of $\big\{ \mathbb{E}\big[ F(\tilde{w}_t) - F_{*} \big] \big\}$  is $\Ocal\left( \frac{\log(t)}{t^2} \right)$.

If, in addition, $\pi^{(t)}$ is uniformly sampled at random without replacement from $[n]$ and $L\sqrt{2(\Theta/n+1)} \leq 1$, then by choosing $\eta_t := \frac{2}{t + 1 + 1/n}$, for all $t \geq 1$, we have
\begin{equation}\label{eq:nonconvex_03_RR}
    \mathbb{E}\big[ F(\tilde{w}_t) - F_{*} \big] \leq  \frac{2}{n(t+1/n)(t + 1 +  1/n)}\Big[  \big( F(\tilde{w}_0) - F_{*}\big)  \ + \   265 \cdot \tau^3 L^2\sigma^2 \log(t + \beta)  \Big].
\end{equation}
Consequently, the convergence rate of  $\big\{ \mathbb{E}\big[ F(\tilde{w}_t) - F_{*} \big] \big\}$ is $\Ocal\left( \frac{\log(t)}{n t^2} \right)$.
\end{thm}

As we mentioned earlier, the condition $L\sqrt{2(\Theta/n+1)} \leq 1$ for \eqref{eq:nonconvex_03_RR} is not restrictive.
One can always scale $f(\cdot;i)$ to guarantee this condition. 
The rate stated in \eqref{eq:nonconvex_03_general} for the general shuffling strategy is $\widetilde{\Ocal}(1/t^2)$ for any $1 \leq t\leq T$ without fixing $T$ a priori. 
If, in addition, a randomized reshuffling strategy is used, then our rate in \eqref{eq:nonconvex_03_RR} is improved to $\widetilde{\Ocal}(1/(nt^2))$.
Our rates in both cases are for diminishing learning rates.
Note that \citep{haochen2018random} provide $\Ocal(1/(nT^2))$ convergence rate but under a constant learning rate and stronger assumptions, including Lipschitz Hessian continuity.
A recent work in \citep{ahn2020sgd} also considers this gradient dominance case and achieves the $\Ocal(1/(nT^2))$ rate but still requires a bounded gradient condition and using constant learning rate.

%%%%%%%%%%%%%%%%%%%%%%%%%%%%%%%%%%%%%%
%%% 5. Numerical Experiments
%%%%%%%%%%%%%%%%%%%%%%%%%%%%%%%%%%%%%%
\section{Numerical Experiments}\label{sec_experiment}
In this section, we provide various numerical experiments to illustrate the theoretical convergence results of Algorithm~\ref{sgd_replacement} for solving nonconvex problem instances of \eqref{ERM_problem_01}. 
We only focus on the nonconvex setting since the convex case has been intensively studied in previous works, e.g., in \citep{ahn2020sgd,gurbuzbalaban2015random,haochen2018random,mishchenko2020random}.
We implement Algorithm~\ref{sgd_replacement} in Python and compare between different variants. 
Our code is available online at \href{https://github.com/lamnguyen-mltd/shuffling}{\textcolor{blue}{https://github.com/lamnguyen-mltd/shuffling}}.
For each experiment, we conduct 10 runs and reported the average results. 

%%% 5.1. Nonconvex logistic regression example.
\subsection{Nonconvex logistic regression example}\label{sec_exp_nonconvex_classification}
We consider the following well-studied binary classification problem with nonconvex $F$:
\begin{equation}\label{loss_experiments}
   \min_{w \in \mathbb{R}^d} \bigg\{ F(w)  := \frac{1}{n} \sum_{i=1}^n \Big[ \log(1 + \exp(- y_i x_i^{\top} w )) + \frac{\lambda}{2} \sum_{j=1}^d \frac{w_j^2}{1 + w_j^2}   \Big]   \bigg\},
\end{equation}
where $w \in \R^d$ is the vector of model parameters and $w_j$ is the element of $w$, $\sets{(x_i, y_i)}_{i=1}^n$ is a set of training examples, and $\lambda > 0$ is a given regularization parameter.

We first conduct experiments to demonstrate the performance of Algorithm~\ref{sgd_replacement} on two classification datasets \textit{w8a} ($n = 49,749$ samples) and  \textit{ijcnn1} ($n = 91,701$) from LIBSVM \citep{LIBSVM}. 
Since we only aim at testing the nonconvexity of each $f_i$ instead of  statistical properties, we simply choose $\lambda := 0.01$, but other values of $\lambda$ also work. 
The input data $x_i$ ($i \in [n]$)  have been scaled in the range of $[0,1]$.

We apply Algorithm~\ref{sgd_replacement}  with mini-batch size of 1 and $\eta_i^{(t)} := \frac{\eta_t}{n}$ to solve \eqref{loss_experiments}, where $\eta_t := \frac{\gamma}{(t+\beta)^{\alpha}}$ and $\pi^{(t)}$ is generated randomly to obtain an SGD variant with randomized reshuffling strategy. 
We experiment using different configurations: $\alpha = \set{1/3, 1/2, 1}$ and $\gamma / n = \set{0.001, 0.005, 0.01}$, respectively on the two datasets: \textit{w8a} and \textit{ijcnn1}. 

Figures~\ref{fig_LR_w8a} and \ref{fig_LR_ijcnn1} show our comparison on the loss value $F(\tilde{w}_t)$ and the test accuracy using different configurations: $\alpha = \set{1/3, 1/2, 1}$ and $\gamma / n = \set{0.001, 0.005, 0.01}$, respectively on the \textit{w8a} and \textit{ijcnn1} datasets. The choices $\alpha = 1$ and $\gamma/n = 0.001$ do not give good training performances, hence we omit them in our plots. 
\begin{figure}[ht!]
\begin{center}
\includegraphics[width=0.49\textwidth]{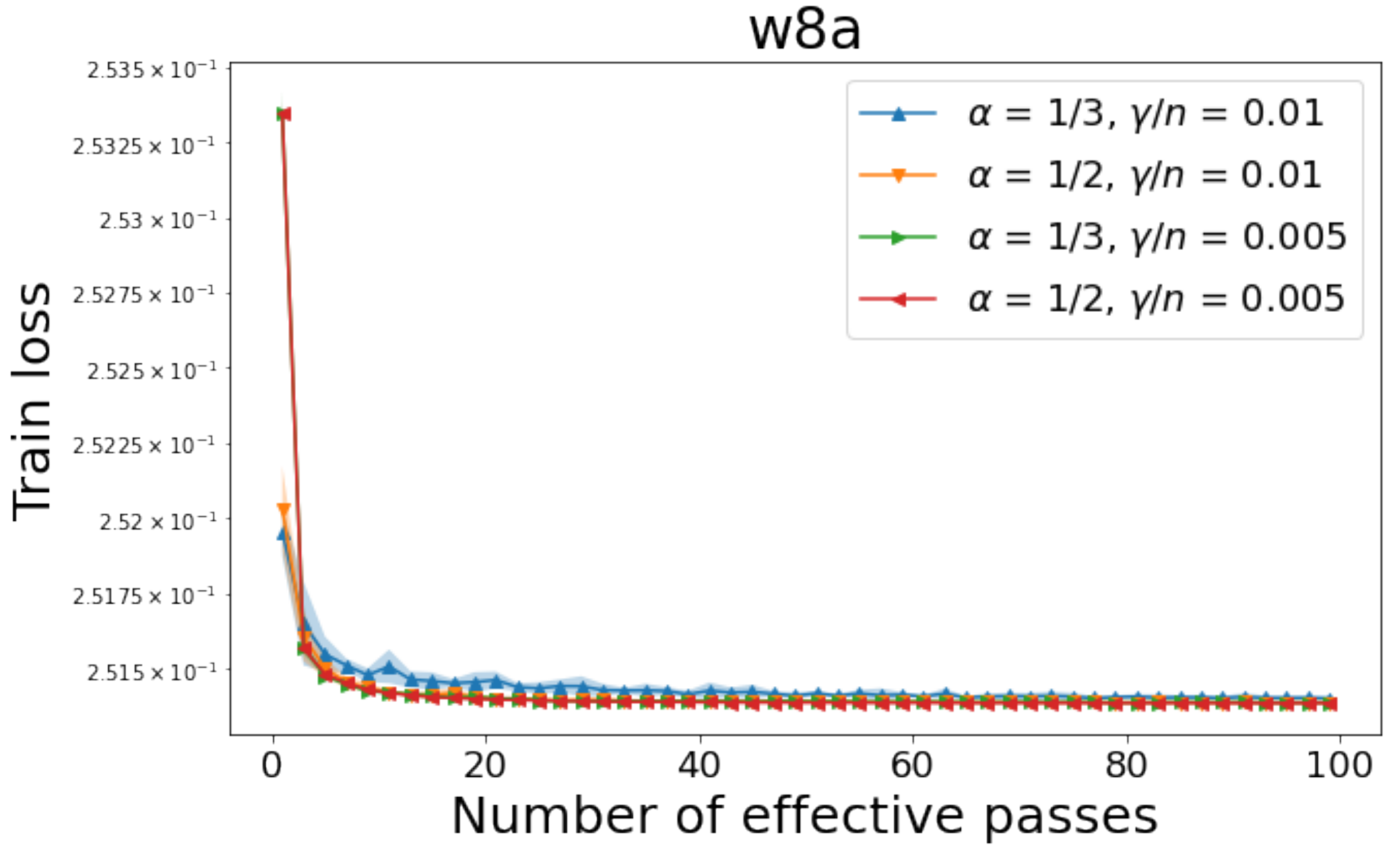}
\includegraphics[width=0.49\textwidth]{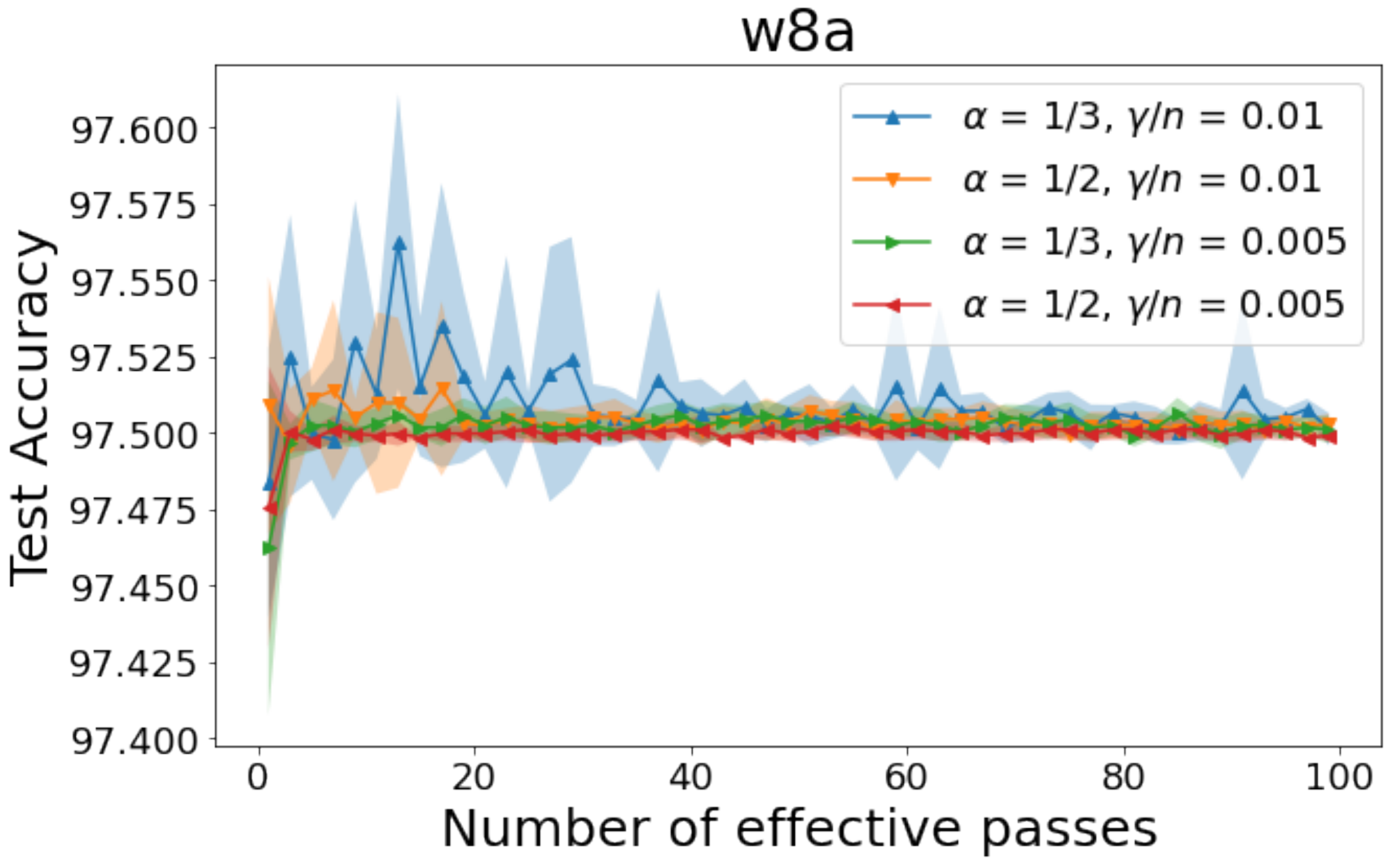}
\caption{The behavior of the train loss $F(\tilde{w}_t)$ and the test accuracy (starting from the $2^{nd}$ epoch) of \eqref{loss_experiments} produced by different values of $\alpha$ and $\gamma / n$ in Algorithm~\ref{sgd_replacement} using the \textit{w8a} dataset.}
\label{fig_LR_w8a}
\end{center}
%\vspace{-2ex}
\end{figure} 
\begin{figure}[ht!]
\vspace{-2ex}
\begin{center}
\includegraphics[width=0.49\textwidth]{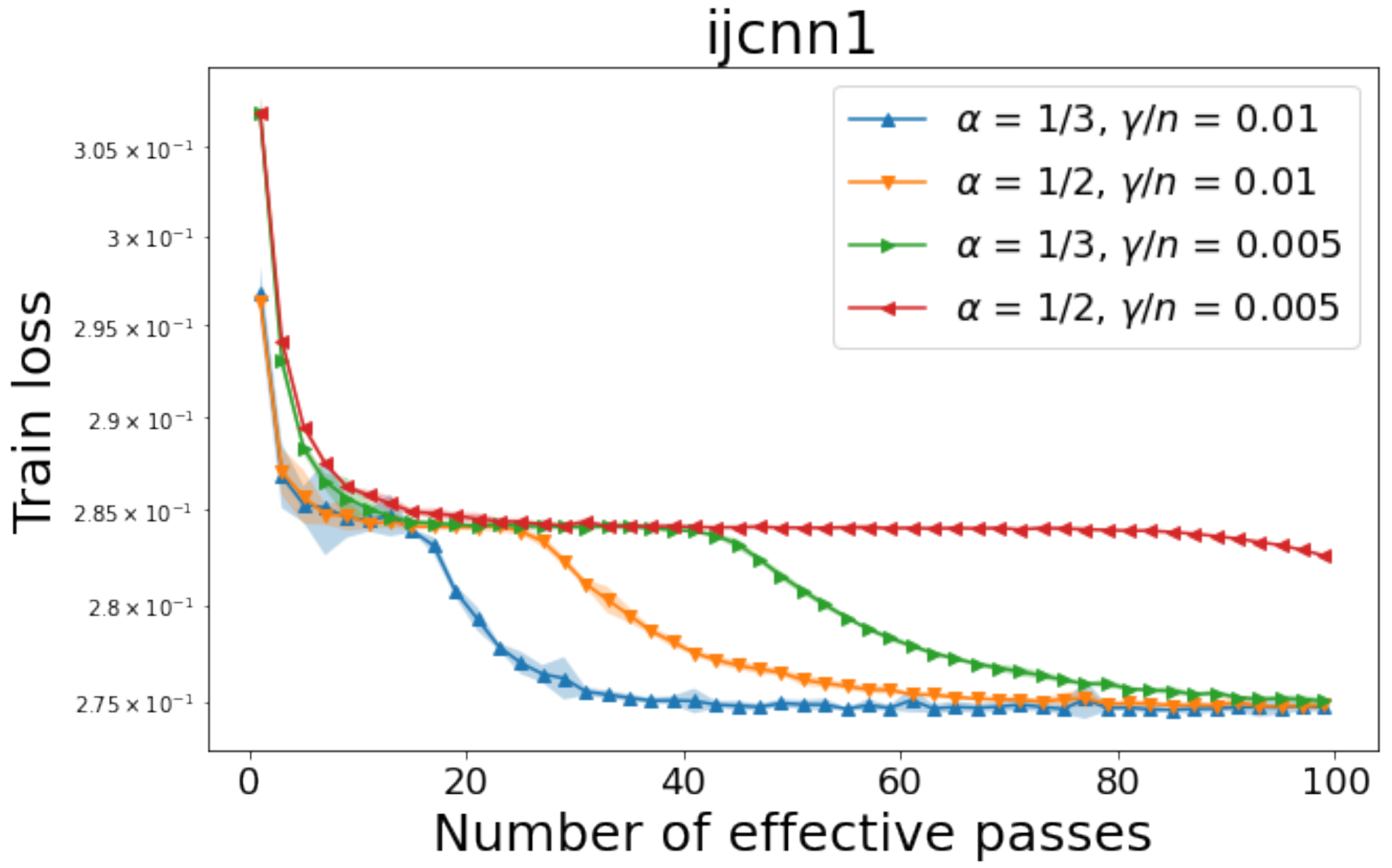}
\includegraphics[width=0.49\textwidth]{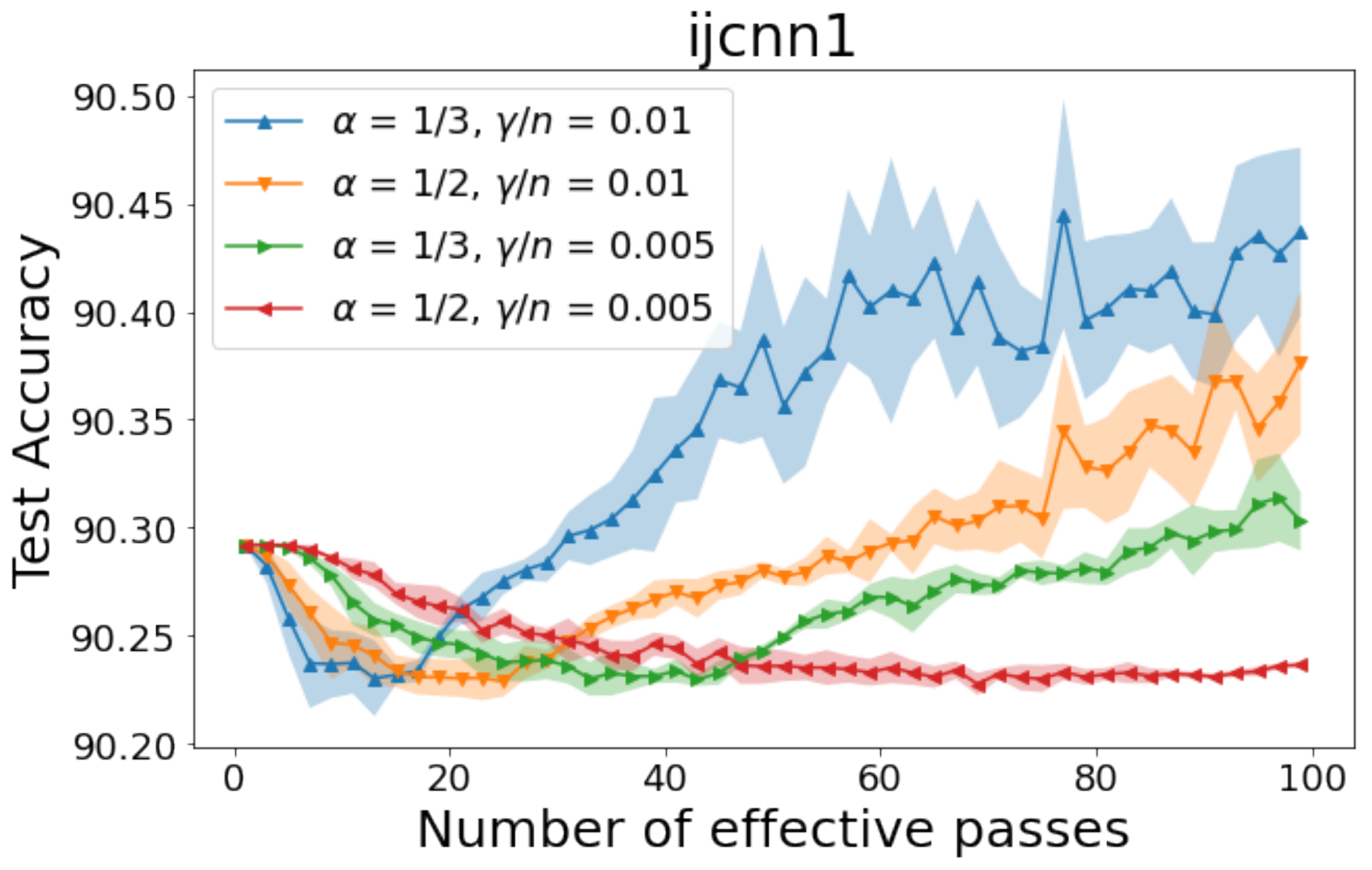}
\caption{The behavior of the train loss $F(\tilde{w}_t)$ and the test accuracy (starting from the $2^{nd}$ epoch) of \eqref{loss_experiments} produced by different values of $\alpha$ and $\gamma / n$ in Algorithm~\ref{sgd_replacement} using the \textit{ijcnn1} dataset.}
\label{fig_LR_ijcnn1}
\end{center}
%\vspace{-2ex}
\end{figure} 

%%% Discussion
\textbf{Discussion.} 
We observe from Figures~\ref{fig_LR_w8a} and \ref{fig_LR_ijcnn1} that the value $\alpha := 1/3$ used in the learning rate of Algorithm~\ref{sgd_replacement} usually gives the best performance.
If we fix $\alpha := 1/3$ and use different ratios $\gamma/n$, then as showed in these plots, $\gamma/n = 0.01$ seems to work best. 
Note that we plotted the confidence intervals in every figure, however these intervals for train loss can barely be seen because the loss values for different random seeds do not deviate much from the mean value in this experiment. 
The configuration $\alpha = 1/3$ and $\gamma/n = 0.01$ for the datasets \textit{w8a} and \textit{ijcnn1} has larger confidence intervals for the test accuracy and can be seen easily in Figures~\ref{fig_LR_w8a} and \ref{fig_LR_ijcnn1}.

%%% 5.2. Fully connected neural network example
\subsection{Fully connected neural network training example}\label{sec_exp_neural_network}
Our second example is to test Algorithm~\ref{sgd_replacement} on a neural network training problem. 
We perform this test on a fully connected neural network with two hidden layers of $300$ and $100$ nodes, followed by a fully connected output layer which fits into a soft-max cross-entropy loss. 
We use PyTorch to train this model on the well-known \texttt{MNIST} dataset with $n = 60,000$ \citep{MNIST}. 
This data set has $10$ classes corresponding to $10$ soft-max output nodes.  
We also conduct another test on the \texttt{CIFAR-10} dataset ($n = 50,000$ samples and $10$ classes) \citep{CIFAR10}. We scale the datasets by standardization.  
To accelerate the performance, we run Algorithm~\ref{sgd_replacement} with a mini-batch size of $256$ instead of single sample.
We use the same setting as in the previous experiment and do not use any weight decay or any data augmentation techniques.

We apply Algorithm~\ref{sgd_replacement} with a learning rate $\eta_i^{(t)} := \frac{\eta_t}{n}$, where $\eta_t := \frac{\gamma}{(t+\beta)^{\alpha}}$ to solve this training problem. 
We repeatedly run the algorithm $10$ times and report the average results in our figures.
These plots compare the algorithms on different values of $\alpha = \set{1/3, 1/2, 1}$ and $\gamma / n = \set{0.01, 0.05, 0.1, 0.5}$, respectively, on the two datasets. 

Figures~\ref{fig_NN_mnist} and \ref{fig_NN_cifar10} show our comparison on the loss value $F(\tilde{w}_t)$ and test accuracy using \texttt{MNIST} and \texttt{CIFAR-10} datasets, respectively. 
For each dataset, we only plot the results of experiments that yield the best  training performance. 

\begin{figure}[ht!]
\begin{center}
\includegraphics[width=0.49\textwidth]{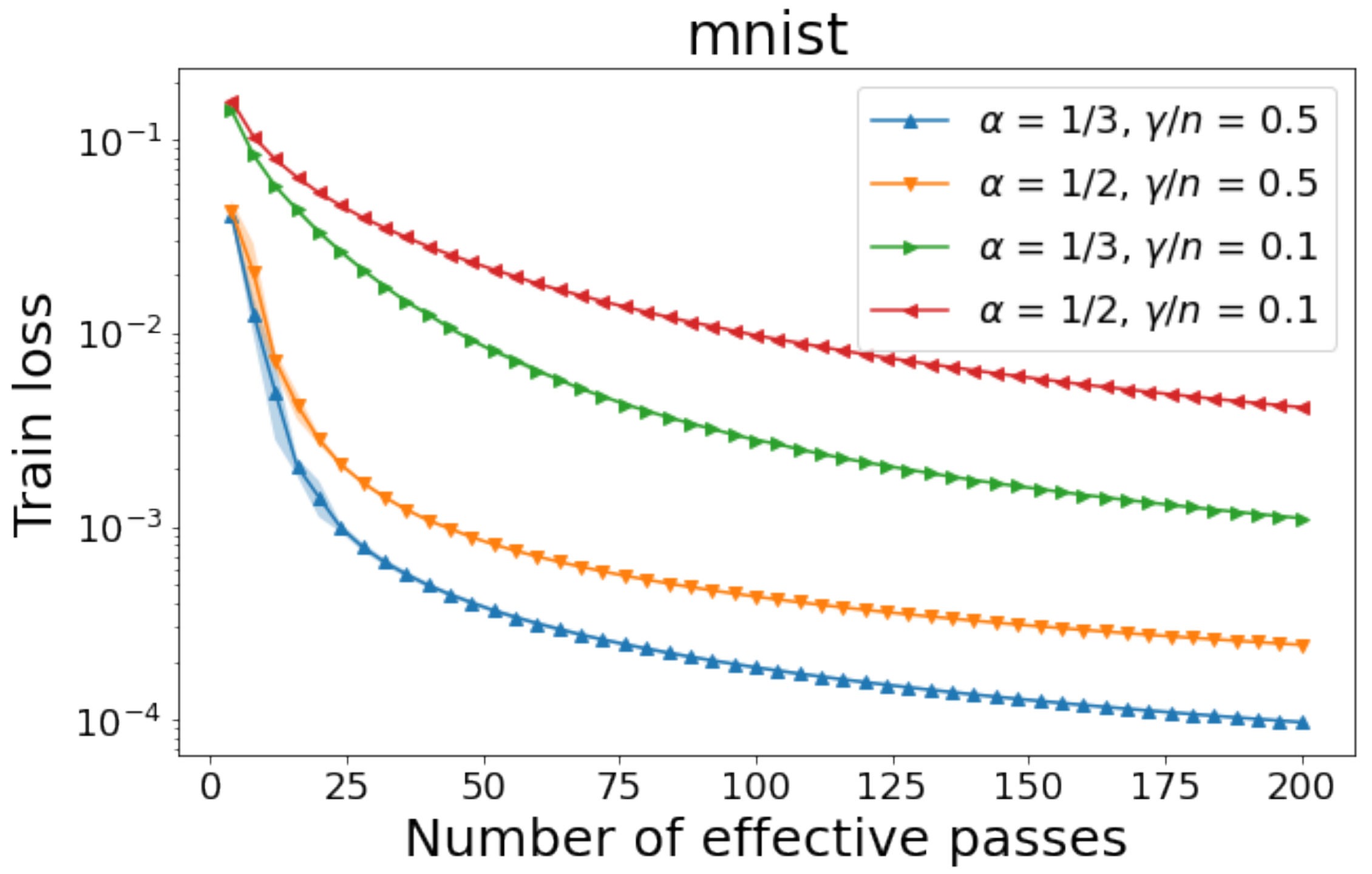}
\includegraphics[width=0.49\textwidth]{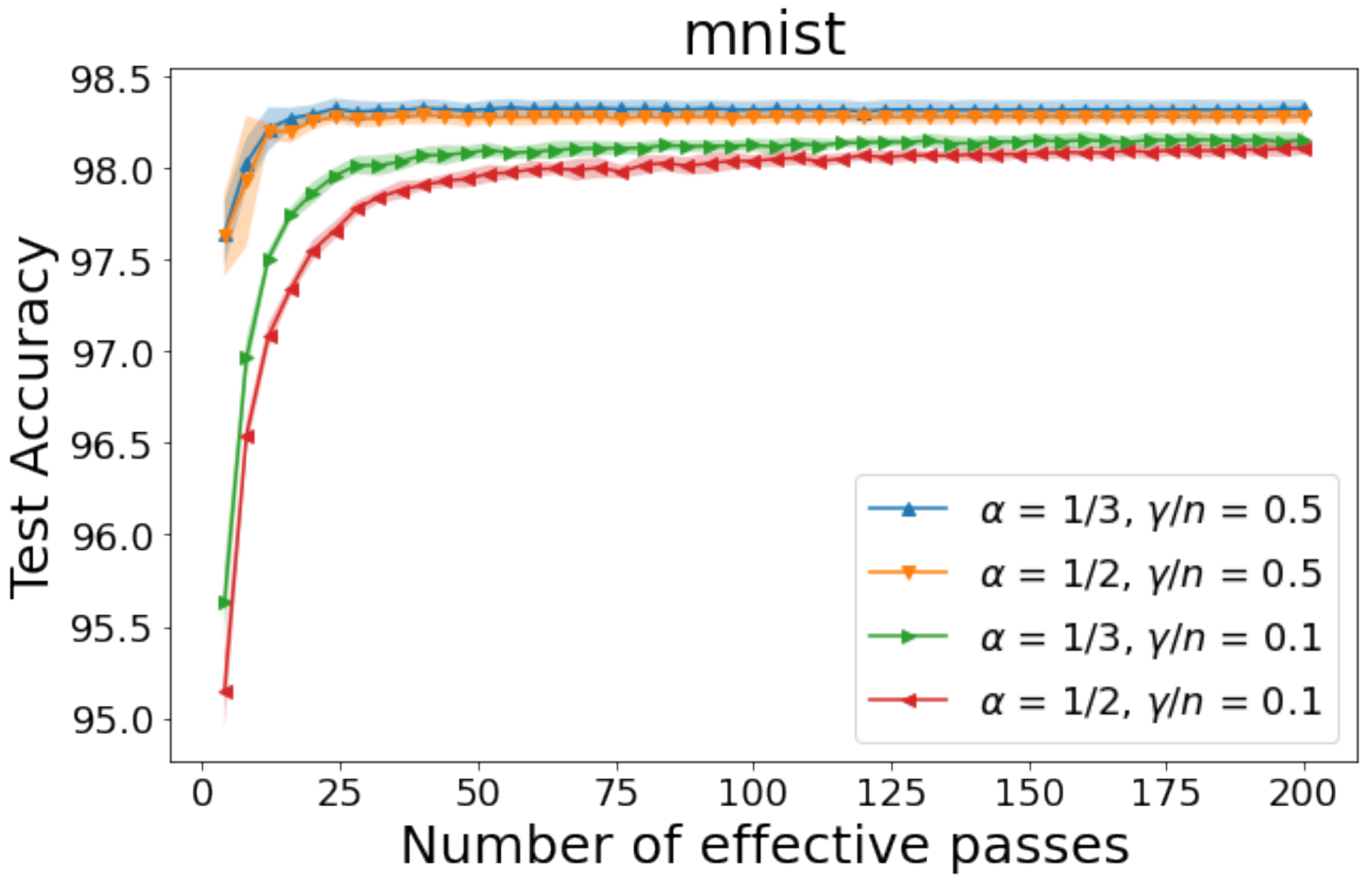}
\caption{The behavior of the train loss $F(\tilde{w}_t)$ and the test accuracy (from the $4^{th}$ epoch) produced by Algorithm~\ref{sgd_replacement} for solving a neural network training problem  on different values of $\alpha$ and $\gamma / n$ using the \textit{MNIST} dataset.}
\label{fig_NN_mnist}
\end{center}
%\vspace{-2ex}
\end{figure} 
\begin{figure}[ht!]
\vspace{-2ex}
\begin{center}
\includegraphics[width=0.49\textwidth]{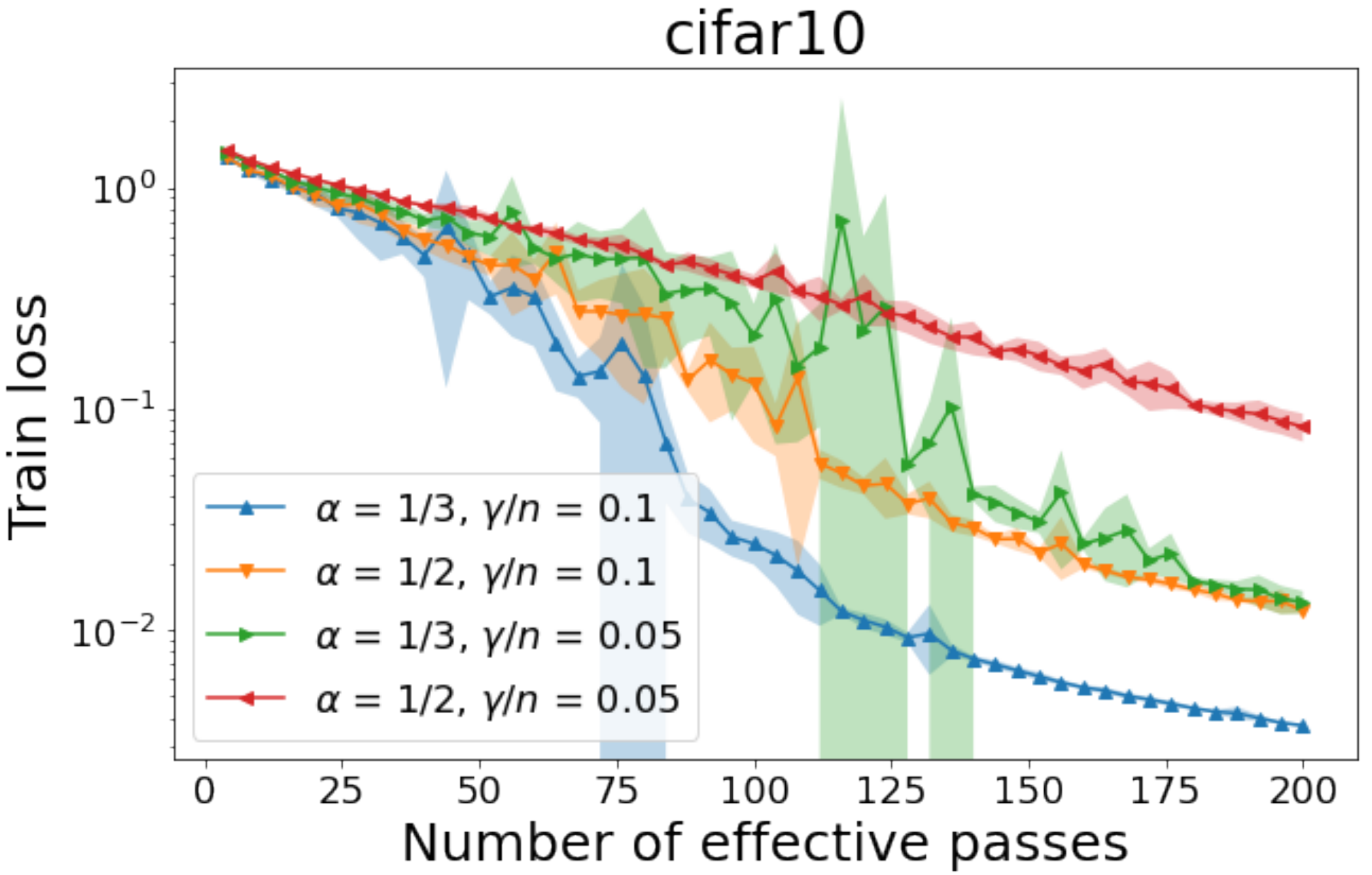}
\includegraphics[width=0.49\textwidth]{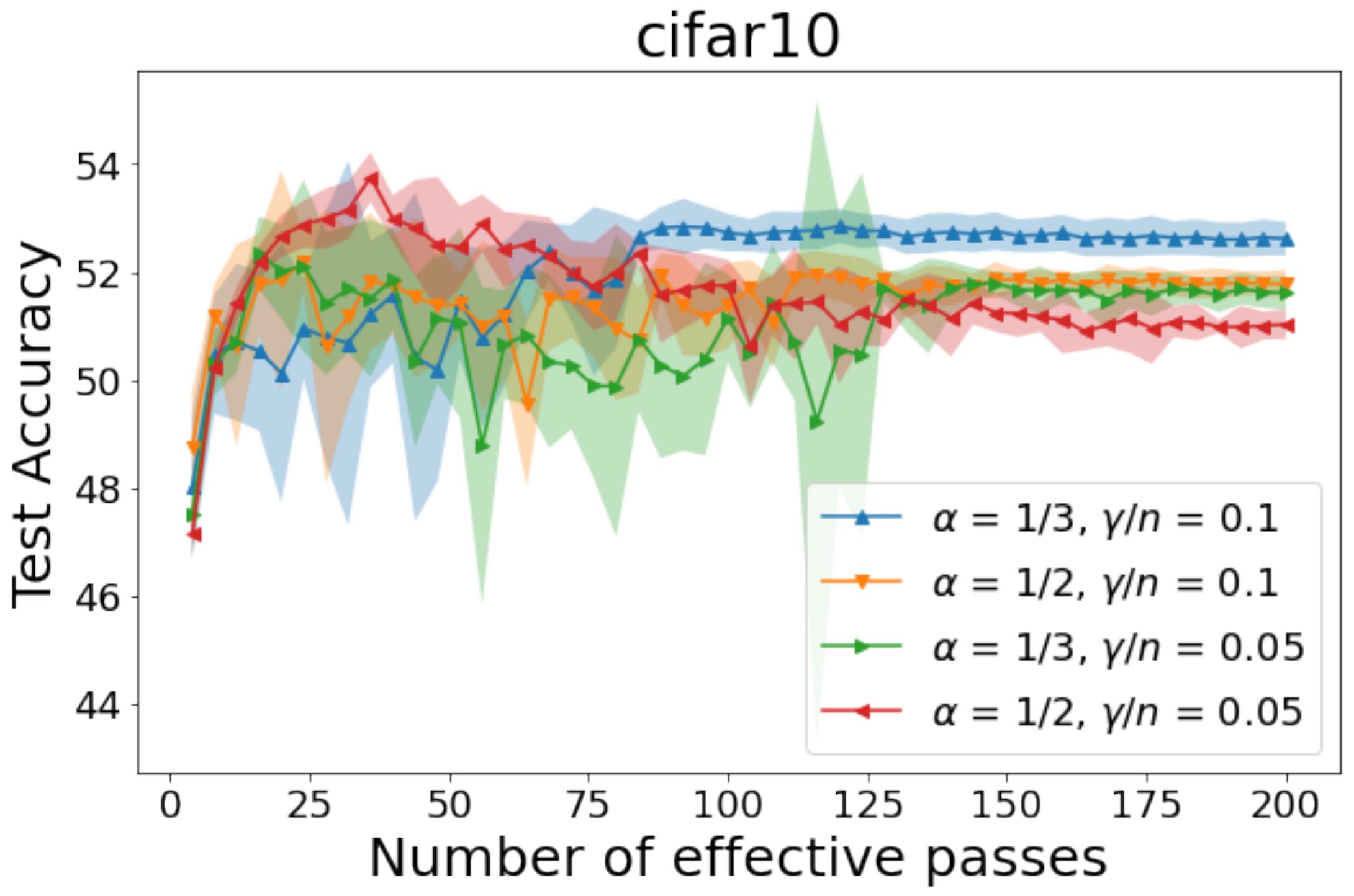}
\caption{The behavior of the train loss $F(\tilde{w}_t)$ and the test accuracy (from the $4^{th}$ epoch) produced by Algorithm~\ref{sgd_replacement} for solving a neural network training problem  on different values of $\alpha$ and $\gamma / n$ using the \textit{CIFAR-10} dataset.}
\label{fig_NN_cifar10}
\end{center}
%\vspace{-2ex}
\end{figure} 

%%% Discussion.
\textbf{Discussion.} We observe again from Figures \ref{fig_NN_mnist} and \ref{fig_NN_cifar10} that $\alpha = 1/3$ works best when fixing $\gamma/n$. 
Once we fix $\alpha := 1/3$ and test on $\gamma/n$, the ratios  $\gamma/n=0.5$ and $\gamma/n=0.1$ give the best performance for the \texttt{MNIST} and \texttt{CIFAR-10} datasets, respectively. Similarly to the nonconvex logistic regression datasets, the train loss values and the test accuracy for \texttt{MNIST} do not deviate much in most cases. 
On the other hand, the \texttt{CIFAR-10} dataset's performance is known to be noisy and we observe this phenomenon again in our experiments. 
The test accuracy for the choice $\alpha := 1/3$ is not ideal in the early stage.
However, it gives the best performance and becomes stable toward the end of the training process.

%5.3 Comparing 3 schemes
\subsection{Comparing different shuffling schemes}
In this last subsection, we compare different shuffling strategies for Algorithm \ref{sgd_replacement}: Random Reshuffling (RR), Shuffle Once (SO), and Incremental Gradient (IG). 
The behaviors of Algorithm~\ref{sgd_replacement} for the \texttt{CIFAR-10} dataset are particularly interesting since its optimization problem is challenging to train. 
We experimented with the same neural network model as in the previous experiment, but with a constant learning rate $\eta_i^{(t)} = \eta_t/n := 0.05$. 
The train loss $F(\tilde{w}_t)$  and the test accuracy are shown in Figure \ref{fig_cifar10_schemes}. 

\begin{figure}[ht!]
\centering
\includegraphics[width=0.49\textwidth]{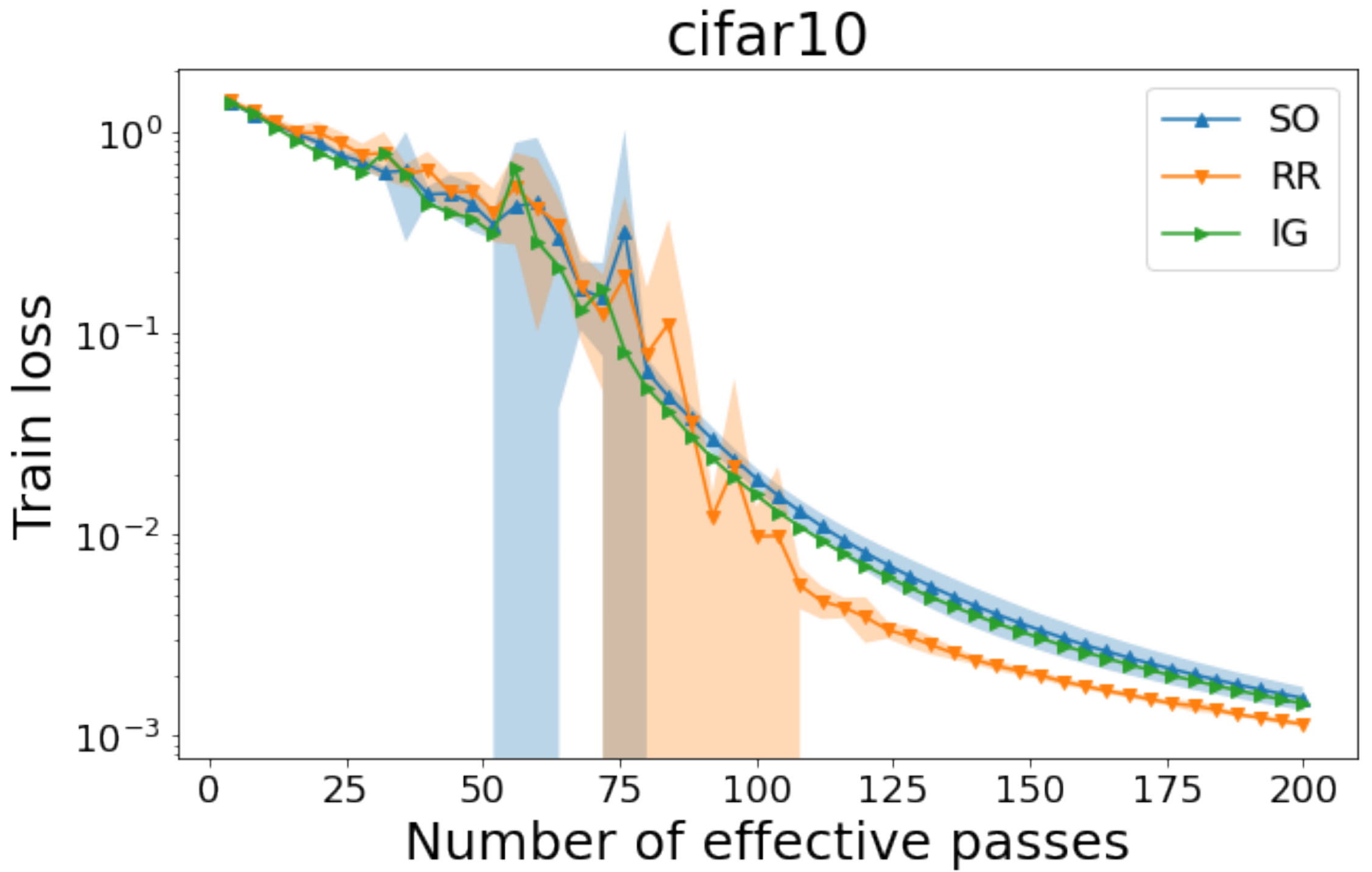}
\includegraphics[width=0.49\textwidth]{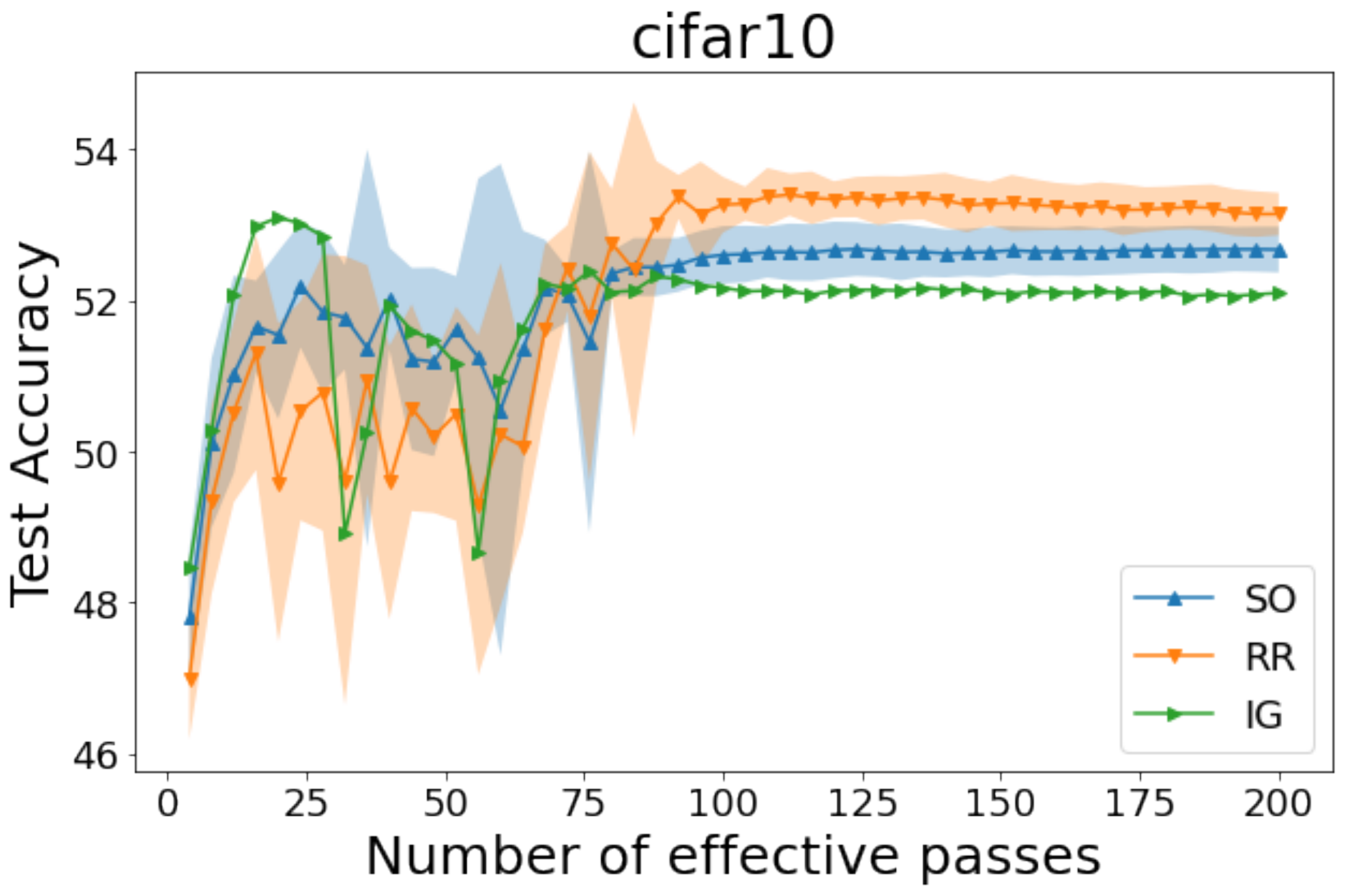}
\caption{The train loss $F(w)$ and test accuracy produced by Algorithm~\ref{sgd_replacement} for the neural network training problem using different shuffling schemes on the \texttt{CIFAR-10} dataset.}
\label{fig_cifar10_schemes}
%\vspace{-1ex}
\end{figure} 

Note that we start all experiments with the same initialization and run these algorithms for $10$ different random seeds. 
The only exception is the Incremental Gradient (IG) scheme where a deterministic permutation $\pi^{(t)}  := \{1, 2, \cdots, n \}$ is used for all $t\geq 1$. 
For this reason, the IG scheme has only one run and we do not see its confidence intervals in Figure \ref{fig_cifar10_schemes}. 

%%% Discussion.
\textbf{Discussion.} 
We observe that the Random Reshuffling scheme works efficiently toward the end of the training process. 
All shuffling schemes we test here are comparable at the early stage, but the deviation seems to decrease along the training epochs. 

%%%%%%%%%%%%%%%%%%%%%%%%%%%%%%%%%%%%%
%%% 6. Conclusions.
%%%%%%%%%%%%%%%%%%%%%%%%%%%%%%%%%%%%%
\section{Conclusions}\label{sec_conclusion}
We have conducted an intensive convergence analysis for a wide class of shuffling-type gradient methods for solving a finite-sum minimization problem.
In the strongly convex case, we have established $\Ocal(T^{-2})$ convergence rate under just strong convexity of the sum function and the smoothness for any shuffling strategy. 
When a randomized reshuffling strategy is used, our rate has been improved to $\Ocal(n^{-1}T^{-2})$, matching the results in the literature but under different assumptions.
For the nonconvex case, we have proved a non-asymptotic $\Ocal(T^{-2/3})$ convergence rate of our algorithm with any shuffling strategy under standard assumptions, which is significantly better than some previous works such as \citep{li2019incremental}.
When a randomized reshuffling strategy is used, our rate has been improved to $\Ocal(n^{-1/3}T^{-2/3})$, matching the recent result in \citep{mishchenko2020random}.
We have also considered these rates in both constant and diminishing learning rates, and investigated an asymptotic convergence.
We believe that our results provide a unified analysis for shuffling-type algorithms using both randomized and deterministic sampling strategies, where it covers the well-known incremental gradient scheme as a special case.
We have conducted different numerical experiments to highlight some theoretical aspects of our results.
We believe that our  analysis framework could be extended to study non-asymptotic convergence rates of SGDs and minimax algorithms, including adaptive SGD variants such as Adam \citep{KingmaB14} and AdaGrad \citep{AdaGrad} under shuffling strategies.

%%% Acknowledgments.
\section*{Acknowledgements}\label{sec_acknowledgement}
The authors would like to thank Trang H. Tran for her valuable comments on some technical proofs and her great help during the revision of this paper. The work of Q. Tran-Dinh has partly been supported by the National Science Foundation (NSF), grant no. DMS-1619884, the Office of Naval Research (ONR), grant no. N00014-20-1-2088 (2020-2023), and The Statistical and Applied Mathematical Sciences Institute (SAMSI).

\appendix
%%%%%%%%%%%%%%%%%%%%%%%%%%%%%%%%%
%%% Appendix A. Some key lemmas for convergence analysis.
%%%%%%%%%%%%%%%%%%%%%%%%%%%%%%%%%
\section{Key Technical Lemmas for Convergence Analysis}\label{sec:apdx:A1}
This appendix provides the full proof of different technical lemmas using for our convergence analysis in the entire paper.
However, let us first outline the key idea of our analysis.

\subsection{The outline of our convergence analysis}\label{subsec:outline_proof}
Let us briefly outline the key steps of our convergence analysis to help the readers easily follow our main proofs.
\begin{compactitem}
\item The key step of our analysis is to form a ``quasi-descent'' inequality between $\mathbb{E}\big[ F(\tilde{w}_t) - F_{*}\big]$,  $\eta_t\mathbb{E}\big[\Vert\nabla{F}(\tilde{w}_{t-1})\Vert^2\big]$, and $\eta_t^3$ as, e.g., in \eqref{eq_rec_01b} or \eqref{eq_001}, relying on \eqref{eq_key_thm_01_est} of  Lemma~\ref{le:F_bound}.
\item For the strongly convex case, we can form a ``quasi-descent'' inequality (see \eqref{eq:scvx_key_bound2} of Lemma~\ref{le:scvx_key_bound2}) between $\mathbb{E}\big[ F(\tilde{w}_t) - F_{*}\big]$, $\mathbb{E}\big[ \norms{\tilde{w}_t - w_{*}}^2 \big]$, and $\eta_t^3$.
\item To obtain such a desired bound, we need to upper bound the average deviation $\frac{1}{n}\sum_{j=0}^{n-1}\norms{w_j^{(t)} - w_0^{(t)}}^2$ between the inner iterates $w_j^{(t)}$ and its epoch iterate $w^{(t)}_0$ via $\Vert\nabla{F}(\tilde{w}_{t-1})\Vert^2$ as in Lemma~\ref{le:ncvx_keybound1}.
\item The final step is to apply either Lemma~\ref{lem_general_framework} or Lemma~\ref{lem_general_framework_02} to obtain our results.
\end{compactitem}
To improve our convergence rate for uniformly randomized reshuffling variants, we exploit Lemma~\ref{le:lem_key_rr} below from  \cite[Lemma 1]{mishchenko2020random}.

%%% A.1. General frameworks.
\subsection{General lemmas}\label{subsec:general_results}
In order to improve our theoretical results for randomized reshuffling variants, we will use \cite[Lemma 1]{mishchenko2020random}, which is stated as follows. 

%%% Mishchenko's lemma.
\begin{lem}\label{le:lem_key_rr}
Let $X_1, \cdots, X_n$ be $n$ given vectors in $\R^d$, $\bar{X} := \frac{1}{n} \sum_{i=1}^n X_i$ be their average, and $\sigma^2 := \frac{1}{n} \sum_{i=1}^n \|X_i -\bar{X}\|^2$ 
be their population variance. 
Fix any $k \in \{1,\cdots, n\}$, let $X_{\pi_1}, \cdots, X_{\pi_k}$ be sampled uniformly without replacement from $\{X_1, \cdots, X_n\}$ and $\bar{X}_{\pi} := \frac{1}{k}\sum_{i=1}^k X_{\pi_i}$ be their average. 
Then, we have
\begin{align*}
    \mathbb{E} [\bar{X}_\pi] = \bar{X} \qquad \text{and} \qquad \mathbb{E} \left[ \|\bar{X}_\pi - \bar{X}\|^2 \right] = \frac{n-k}{k(n-1)} \sigma^2.
\end{align*}
\end{lem}

Now, let us prove Lemma~\ref{lem_general_framework} and Lemma~\ref{lem_general_framework_02} in Subsection~\ref{sec:basic_tools} of the main text.

% Begin of proof.
\vspace{1ex}
\begin{proof}[\mytxtbi{The proof of Lemma~\ref{lem_general_framework}}]
Let us choose $\eta_t := \frac{q}{\rho(t + \beta)}$ for all $t \geq 1$, where $\beta \geq q-1$.
Then, $0 < \rho\eta_t \leq 1$ for all $t\geq 1$.
From the condition $Y_{t+1} \leq (1 - \rho \cdot \eta_t)Y_t + D\cdot \eta_t^{q+1}$ in \eqref{eq:induc_ineq}, by induction, we can show that
\begin{equation}\label{eq:lm1_estimate21}
Y_{t+1} \leq \prod_{i=1}^t(1-\rho \cdot \eta_i) Y_1 + D\sum_{i=1}^t \eta_i^{q+1}\prod_{j=i+1}^t(1 - \rho \cdot \eta_j).
\end{equation}
Using $\eta_t := \frac{q}{\rho(t + \beta)}$, we can directly compute the first coefficient as
\begin{equation*}
\arraycolsep=0.2em
\begin{array}{lcl}
C_t & := & \prod_{i=1}^t(1 - \rho \cdot \eta_i) = \prod_{i=1}^t(1 - \frac{q}{i + \beta}) = \prod_{i=1}^t\frac{i + \beta - q}{i + \beta}  \vspace{1ex}\\
& = & \frac{\beta + 1 - q}{\beta + 1}\cdot \frac{\beta + 2 - q}{\beta + 2} \cdots \frac{\beta + 1}{\beta + q +  1} \cdots \frac{t + \beta - q}{t + \beta} \vspace{1ex}\\
& = & \frac{(\beta + 1 - q)\cdots \beta}{(t+\beta - q + 1)\cdots (t+ \beta)}.
\end{array}
\end{equation*}
Similarly, we can show that, for any $1 \leq i \leq t$, we have
\begin{equation*}
\arraycolsep=0.2em
\begin{array}{lcl}
E_{i,t} &:= & \eta_i^{q+1}\prod_{j=i+1}^t(1 - \rho \cdot \eta_j) = \frac{q^{q+1}}{\rho^{q+1}(i+\beta)^{q+1}} \cdot \frac{(i+\beta + 1 - q)\cdots (i + \beta)}{(t+\beta + 1 - q)\cdots (t+\beta)} \vspace{1ex}\\
& \leq & \frac{q^{q+1}}{\rho^{q+1}(t+\beta+ 1 - q) \cdots (t+\beta)} \cdot \frac{1}{(i+\beta)}.
\end{array}
\end{equation*}
Therefore, we obtain 
\begin{equation*}
\arraycolsep=0.2em
\begin{array}{lcl}
\sum_{i=1}^t\eta_i^{q+1}\prod_{j=i+1}^t(1 - \rho \cdot \eta_j) & = & \sum_{i=1}^tE_{i,t} \leq  \frac{q^{q+1}}{\rho^{q+1}(t+\beta + 1 - q)\cdots (t+\beta)}\sum_{i=1}^t\frac{1}{i+\beta} \vspace{1ex}\\
& \leq &  \frac{q^{q+1}\log(t + \beta)}{\rho^{q+1}(t+\beta-q+1)\cdots (t+\beta)} 
\end{array}
\end{equation*}
Substituting this sum and $C_t$ above into \eqref{eq:lm1_estimate21}, we finally obtain \eqref{eq_rate_t2}, i.e.: 
\begin{equation*}
\arraycolsep=0.2em
\begin{array}{lcl}
Y_{t+1} \leq  \frac{(\beta + 1 - q)\cdots \beta}{(t+\beta  - q + 1)\cdots (t+ \beta)} \cdot Y_1  \ + \  \frac{D q^{q+1}\log(t + \beta)}{\rho^{q+1}(t+\beta-q+1)\cdots(t+\beta)},
\end{array}
\end{equation*}
If we choose $\eta_t := \eta \in (0, \rho^{-1})$ for all $t\geq 1$, then $0 < \rho \eta_t \leq 1$, $C_t = (1 - \rho\eta)^t$, and $E_{i,t} = (1 - \rho \eta)^{t-i}\eta^{q+1}$.
Hence, we get $\sum_{i=1}^tE_{i,t} = \eta^{q+1}\sum_{i=1}^t(1 - \rho \eta)^{t-i} = \frac{\eta^{q}[ 1 - (1 - \rho \eta)^t]}{\rho}$.
Substituting these $C_t$ and $E_{i,t}$ into \eqref{eq:lm1_estimate21}, we obtain $Y_{t+1} \leq (1 - \rho \eta)^{t} Y_1 + \frac{D\eta^{q}[ 1 - (1 - \rho \eta)^{t}]}{\rho}$, which proves the first inequality of \eqref{eq_rate_t3}.
Now, since $1 - \rho\eta \leq \exp(-\rho\eta)$ for $0 \leq \rho\eta \leq 1$ and $1 - (1 - \rho \eta)^{t} \leq 1$, we can easily prove the second inequality of \eqref{eq_rate_t3}.
\end{proof}
%%% End of proof.

Next, we prove the following elementary results, which will be used to prove Lemma~\ref{lem_general_framework_02}. 

%%% Lemma 3.
\begin{lem}\label{lem_sup_gen_framework_02}
The following statements hold:
\begin{compactitem}
\item[$\mathrm{(a)}$]
For any $0 \leq \nu \leq \frac{1}{2}$ and $s > 0$, we have $(s + 1)^{\nu} - s^{\nu} \leq \frac{1}{2s^{1-\nu}}$.

\item[$\mathrm{(b)}$] 
For any $c  > 0$, $\theta > 0$, $\beta > 0$, and $1 + \theta - \beta > ce^{\frac{1-c}{c}}$, the function $f(t) := \frac{\log(t + 1 + \theta)}{(t + \beta)^c}$ is monotonically decreasing on $[0, +\infty)$. 

\item[$\mathrm{(c)}$] 
Suppose that $f$ is a real-valued and monotonically decreasing function on $[a, +\infty)$ such that $f(x) \geq 0$ for all $x \in [a, +\infty)$. 
Then, for any integers $t$ and $t_0$ such that $t \geq t_0 \geq a$, we have
\begin{equation}\label{integral_03}
    \sum_{i = t_0 + 1}^{t} f(i) \leq \int_{t_0}^{t} f(x) dx \leq \sum_{i=t_0}^{t-1} f(i). 
\end{equation}
\end{compactitem}
\end{lem}

%%% Proof of Lemma 3.
\begin{proof}
(a)~If $2\nu \leq 1$, then $\left(\frac{s+1}{s}\right)^{1-2\nu} \geq 1$, which is equivalent to $\frac{s+1}{s} \geq \left(\frac{s+1}{s}\right)^{2\nu}$.
This leads to $(s+1)^{\nu}s^{1-\nu} - s^{\nu}(s+1)^{1-\nu} \leq 0$.
Hence, we have
\begin{equation*}
(s + 1)^{\nu} - s^{\nu}  = \frac{1 + (s+1)^{\nu}s^{1-\nu} - s^{\nu}(s+1)^{1-\nu}}{(s + 1)^{1-\nu} + s^{1-\nu}} \leq \frac{1}{(s + 1)^{1-\nu} + s^{1-\nu}} \leq \frac{1}{2s^{1-\nu}},
\end{equation*}
which proves assertion (a).

\indent{(b)}~Our goal is to show that $f'(t) < 0$ for all $t \geq 0$. 
We can directly compute $f'(t)$ as
\begin{equation*}
    f'(t) = (t + \beta)^{-c-1} \left[  1 - \frac{1 + \theta - \beta}{t + 1 + \theta} - c \cdot \log(t + 1 + \theta)  \right] = (t + \beta)^{-c-1} g(t + 1 + \theta), 
\end{equation*}
where $g(\tau) :=  1 - \frac{1 + \theta - \beta}{\tau} - c \log(\tau)$.
We consider $g(\tau)$ for $\tau > 0$.
It is obvious to show that $g'(\tau) = \frac{1 + \theta - \beta}{\tau^2} - \frac{c}{\tau} = \frac{(1 + \theta - \beta) - c\tau}{ \tau^2}$ and $g''(\tau) = \frac{c\tau - 2(1 + \theta - \beta)}{\tau^3}$.
Hence, $g'(\tau) = 0$ has a unique solution $\tau^{*} := \frac{1+\theta - \beta}{c} > 0$ and $g''(\tau^{*}) = -\frac{c^3}{(1 + \theta - \beta)^2} < 0$.
Consequently, $g$ attains its unique local maximum at $\tau^{*}$. 
Moreover, for $\tau \geq \tau^{*}$, we have $g'(\tau) \leq 0$. Hence, $g$ is nonincreasing on $[\tau_{*}, +\infty)$, which leads to
\begin{equation*}
g(\tau) \leq g(\tau^{*}) = 1 - c - c \log\left(\frac{1 + \theta - \beta}{c}\right) < 0.
\end{equation*}
Here, the last inequality holds since $1 + \theta - \beta > ce^{\frac{1-c}{c}}$.
Since $f'(t) = (t + \beta)^{-c-1}g(t + 1 + \theta)$, where $(t + \beta)^{-c-1} > 0$ for any $t\geq 0$ and $c$, we have $f'(t) < 0$ for all $t \geq 0$.
Hence, $f$ is monotonically decreasing on $[0, +\infty)$.

\indent{(c)}~If $f$ is monotonically decreasing and nonnegative on $[a, +\infty)$, then $f(i+1) \leq \int_{i}^{i+1} f(x) dx \leq f(i)$ for any integer $i \geq a$.
Hence, summing this inequality from $i := t_0$ to $t-1$, we have
\begin{align*}
    \sum_{i = t_0 + 1}^{t}f(i) = \sum_{i = t_0}^{t-1} f(i+1) \leq \sum_{i = t_0}^{t-1}\int_{i}^{i+1}f(x)dx = \int_{t_0}^{t} f(x) dx \leq \sum_{i=t_0}^{t-1}f(i),
\end{align*}
which proves \eqref{integral_03}.
\end{proof}
%%% End of Lemma 3.

%%% The proof of Lemma 2.
\begin{proof}[\mytxtbi{The proof of Lemma~\ref{lem_general_framework_02}}]
From the inequality \eqref{eq_lem_general_framework_02} and $\eta_t := \frac{\gamma}{(t+\beta)^{\alpha}}$, we have
\begin{align*}
    Z_t  \leq \frac{1}{\rho \eta_t^m} (Y_t - Y_{t+1} ) + \frac{D \eta_t^{q-m}}{\rho}  = \frac{(t+\beta)^{\alpha m}}{\rho \gamma^m} (Y_t - Y_{t+1} ) + \frac{D \gamma^{q-m}}{\rho} \cdot \frac{1}{(t + \beta)^{\alpha(q-m)}}. 
\end{align*}
Next, using Lemma~\ref{lem_sup_gen_framework_02}(a) with $s := t + \beta$ and $\nu := m\alpha$ we have
\begin{equation}\label{eq:proof1t}
(t+\beta + 1)^{\alpha m} - (t + \beta)^{\alpha m} \leq \frac{1}{2(t+\beta)^{1 - \alpha m}},
\end{equation}
because we assume that $m\alpha \leq \frac{1}{2}$. 
Summing up the first inequality from $t = 1,\cdots, T$ and taking average, we have
\begin{equation*}
\arraycolsep=0.1em
\begin{array}{lcl}
    \frac{1}{T} \sum_{t = 1}^{T} Z_{t} & \leq & 
     \frac{1}{\rho \gamma^m} \cdot \frac{1}{T} \sum_{t=1}^{T} (t+\beta)^{\alpha m} (Y_t - Y_{t+1} ) + \frac{D \gamma^{q-m}}{\rho} \cdot \frac{1}{T} \sum_{t=1}^{T} \frac{1}{(t + \beta)^{\alpha(q-m)}} \vspace{1ex}\\
    & = &  \frac{1}{\rho \gamma^m} \cdot \frac{1}{T} \left [  (1+\beta)^{\alpha m} Y_1 - (T + \beta)^{\alpha m} Y_{T+1}  \right] 
     + \frac{D \gamma^{q-m}}{\rho} \cdot \frac{1}{T} \sum_{t=1}^{T} \frac{1}{(t + \beta)^{\alpha(q-m)}} \vspace{1ex}\\
  && + {~}  \frac{1}{\rho \gamma^m} \cdot \frac{1}{T} \sum_{t=1}^{T-1} \left(  (t + 1 + \beta)^{\alpha m} - (t + \beta)^{\alpha m}  \right) Y_{t+1} \vspace{1ex}\\
    & \overset{\eqref{eq:proof1t}}{\leq}  & \frac{(1+\beta)^{\alpha m} Y_1}{\rho \gamma^m} \cdot \frac{1}{T} + \frac{1}{2 \rho \gamma^m} \cdot \frac{1}{T} \sum_{t=1}^{T-1} \frac{C + H \log(t + 1 + \theta)}{(t+\beta)^{1 - \alpha m}} + \frac{D \gamma^{q-m}}{\rho} \cdot \frac{1}{T} \sum_{t=1}^{T} \frac{1}{(t + \beta)^{\alpha(q-m)}} \vspace{1ex}\\
    & \overset{\eqref{integral_03}}{\leq} & \frac{(1+\beta)^{\alpha m} Y_1}{\rho \gamma^m} \cdot \frac{1}{T} + \frac{C}{2 \rho \gamma^m} \cdot \frac{1}{T} \int_{t=0}^{T-1} \frac{dt}{(t+\beta)^{1 - \alpha m}} + \frac{H}{2 \rho \gamma^m} \cdot \frac{1}{T} \int_{t=0}^{T-1} \frac{\log (t + 1 + \theta)}{(t+\beta)^{1 - \alpha m}} dt \vspace{1ex}\\
    && + {~} \frac{D \gamma^{q-m}}{\rho} \cdot \frac{1}{T} \int_{t=0}^{T} \frac{dt}{(t + \beta)^{\alpha(q-m)}},
\end{array}
\end{equation*}
where the second inequality follows since $0 \leq Y_t \leq C + H \log (t + \theta)$ for some $C > 0$, $H \geq 0$, and $\theta > 0$, for all $t \geq 1$, and $\alpha m \leq \frac{1}{2}$. 
The last inequality follows since $\frac{\log (t + 1 + \theta)}{(t+\beta)^{1 - \alpha m}}$ is nonnegative and monotonically decreasing on $[0,\infty)$ according to Lemma~\ref{lem_sup_gen_framework_02}(b) with $1 - \alpha m \geq \frac{1}{2} > 0$ and $1 + \theta - \beta > (1 - \alpha m)e^{\frac{\alpha m}{1 - \alpha m}}$, and both $\frac{1}{(t +\beta)^{1-\alpha m}}$ and $\frac{1}{(t+\beta)^{\alpha(q-m)}}$ are also nonnegative and monotonically decreasing on $[0,\infty)$. 
Note that
\begin{equation*}
	\begin{array}{lcl}
   	 \int_{t=0}^{T-1} \frac{\log (t + 1 + \theta)}{(t+\beta)^{1 - \alpha m}} dt & = & \frac{1}{\alpha m} (t + \beta)^{\alpha m} \log(t + 1 + \theta) \Big |_{t=0}^{T-1} 
	 - \frac{1}{\alpha m} \int_{t=0}^{T-1} \frac{(t+\beta)^{\alpha m}}{(t+1 + \theta)}  dt \vspace{1ex}\\
	    & \leq & \frac{1}{\alpha m} (T - 1 + \beta)^{\alpha m} \log(T + \theta). 
	\end{array}
\end{equation*}
Therefore, we consider two cases:

\noindent $\bullet$~If $\alpha(q - m) = 1$, we have
    \begin{equation*}
    	\begin{array}{lcl}
        \frac{1}{T} \sum_{t = 1}^{T} Z_{t} & \leq & \frac{(1+\beta)^{\alpha m} Y_1}{\rho \gamma^m} \cdot \frac{1}{T} \ + \ \frac{C}{2 \rho \alpha m \gamma^m} \cdot \frac{(T - 1 + \beta)^{\alpha m} - \beta^{\alpha m}}{T} \vspace{1ex}\\
       && + {~}\frac{H}{2\rho \alpha m \gamma^m} \cdot \frac{(T - 1 + \beta)^{\alpha m} \log(T + \theta)}{T}  \ + \ \frac{D \gamma^{q - m}}{\rho} \cdot \frac{\log(T+\beta) - \log(\beta)}{T}  \vspace{1ex}\\
        & \leq &\frac{(1+\beta)^{\alpha m} Y_1}{\rho \gamma^m} \cdot \frac{1}{T} + \frac{C}{2 \rho \alpha m \gamma^m} \cdot \frac{(T - 1 + \beta)^{\alpha m}}{T}  \vspace{1ex}\\
       && + {~} \frac{H}{2\rho \alpha m \gamma^m} \cdot \frac{(T - 1 + \beta)^{\alpha m} \log(T + \theta)}{T} \ + \ \frac{D \gamma^{q - m}}{\rho} \cdot \frac{\log(T+\beta) - \log(\beta)}{T}. 
    \end{array}
    \end{equation*}
\noindent $\bullet$~If $\alpha(q - m) \neq 1$, we have
    \begin{equation*}
    \begin{array}{lcl}
        \frac{1}{T} \sum_{t = 1}^{T} Z_{t} & \leq &  \frac{(1+\beta)^{\alpha m} Y_1}{\rho \gamma^m} \cdot \frac{1}{T} \ + \ \frac{C}{2 \rho \alpha m \gamma^m} \cdot \frac{(T - 1 + \beta)^{\alpha m} -  \beta^{\alpha m}}{T}   + \frac{H}{2\rho \alpha m \gamma^m} \cdot \frac{(T - 1 + \beta)^{\alpha m} \log(T + \theta)}{T}  \vspace{1ex}\\
        && + {~} \frac{D \gamma^{q-m}}{\rho (1 - \alpha ( q - m))} \cdot \frac{(T + \beta)^{1 - \alpha(q - m)} - \beta^{1 - \alpha(q - m)}}{T} \vspace{1ex}\\
        & \leq  & \frac{(1+\beta)^{\alpha m} Y_1}{\rho \gamma^m} \cdot \frac{1}{T} \ + \ \frac{C}{2 \rho \alpha m \gamma^m} \cdot \frac{(T - 1 + \beta)^{\alpha m}}{T}  \ + \ \frac{H}{2\rho \alpha m \gamma^m} \cdot \frac{(T - 1 + \beta)^{\alpha m} \log(T + \theta)}{T}  \vspace{1ex}\\
        && + {~} \frac{D \gamma^{q-m}}{\rho (1 - \alpha ( q - m))} \cdot \frac{(T + \beta)^{1 - \alpha(q - m)}}{T}. 
    \end{array}
 \end{equation*}
Here, the result is obtained by directly computing the integrals.
Hence, \eqref{eq:key_concl} is proved.
\end{proof}
%%% End of the proof.

%%% A.2. Key lemmas for convex problem.
\subsection{Key estimates }\label{subsec:key_lemmas}
This appendix provides four technical lemmas for the next steps of our analysis.
However, let us first state the following facts.
\begin{compactitem}
\item Let $w_{*}$ be a stationary point of $F$, i.e. $F(w_{*}) = 0$. Then, if $F$ is convex, then $w_{*} \in \arg \min_{w \in \mathbb{R}^d} F(w)$.
Consequently, for any permutation $\pi^{(t)}$ of $[n]$, we have
\begin{equation}\label{eq:fact1}
\sum_{j=0}^{n-1}\nabla{f}(w_{*}; \pi^{(t)}(j+1)) = 0.
\end{equation}
\item For any $1 \leq i \leq n$, from the update of $w_i^{(t)}$ in Algorithm~\ref{sgd_replacement}, we have
\begin{equation}\label{eq:fact2}
\arraycolsep=0.2em
\begin{array}{lcl}
w_i^{(t)} & = &  \tilde{w}_{t-1} - \frac{\eta_t}{n}\sum_{j=0}^{i-1} \nabla{f}(w_j^{(t)}; \pi^{(t)}(j+1)), \vspace{1ex}\\
\tilde{w}_t & = &  \tilde{w}_{t-1} - \frac{\eta_t}{n} \sum_{j=0}^{n-1} \nabla{f}(w_j^{(t)}; \pi^{(t)}(j+1)).
\end{array}
\end{equation}
\item Let us recalled $\sigma_{*}^2 := \frac{1}{n}\sum_{i=0}^n\norms{\nabla{f}(w_{*}; i)}^2$, the variance of $F$, defined by \eqref{defn_finite}.
\end{compactitem}

Now, we first upper bound $\frac{1}{n} \sum_{i=0}^{n-1} \norms{ w_{i}^{(t)} - w_{0}^{(t)} }^2$ in the following lemma.
This lemma only requires Assumption \ref{ass_basic}(ii) to hold without convexity.

%%% Lemma 2.
\begin{lem}\label{lem_bound_weights}
Suppose that Assumption \ref{ass_basic}$\mathrm{(ii)}$ holds for \eqref{ERM_problem_01}. 
Let $\sets{w_i^{(t)}}$ be generated by  Algorithm~\ref{sgd_replacement} with the learning rate $\eta_i^{(t)} := \frac{\eta_t}{n} > 0$ for a given positive sequence $\sets{\eta_t}$. 
Then
\begin{equation}\label{eq_lem_bound_02}
\arraycolsep=0.2em
\begin{array}{lcl}
\norms{ w_{i}^{(t)} - w_{0}^{(t)} }^2 &\leq  & \frac{2 L^2 \eta_t^2 \cdot i }{n^2} \sum_{j=0}^{i-1} \norms{w_{j}^{(t)} - w_{*} }^2 + \frac{ 2 \eta_t^2 (n-i)}{n} \cdot \sigma_{*}^2.  \vspace{1ex}\\
 \norms{ w_{i}^{(t)} - w_{*} }^2  &\leq &  2  \norms{ w_{0}^{(t)} - w_{*} }^2  +  \frac{4 L^2 \eta_t^2 \cdot i }{n^2} \cdot \sum_{j=0}^{i-1}  \norms{ w_{j}^{(t)} - w_{*} }^2 + \frac{4\eta_t^2 (n-i) \sigma_{*}^2}{n}.
 \end{array}
\end{equation}
If, in addition, $0 < \eta_t \leq \frac{1}{2L}$, then, for any $1 \leq i \leq n$, we have
\begin{equation}\label{eq_lem_bound_02b}
\arraycolsep=0.2em
\begin{array}{lcl}
\sum_{j=0}^{i-1}\norms{w_j^{(t)} - w_{*}}^2 \leq 4 i \cdot \big[ \norms{w_0^{(t)} - w_{*}}^2 + 2 \eta_t^2\sigma_{*}^2 \big].
\end{array}
\end{equation}
Consequently, if $0 < \eta_t \leq \frac{1}{2L}$ for all $t \geq 1$, then we have 
\begin{equation}\label{eq_thm_weight_01}
    \frac{1}{n} \sum_{i=0}^{n-1} \norms{ w_{i}^{(t)} - w_{0}^{(t)} }^2  \leq \eta_t^2 \cdot \frac{8 L^2}{3} \norms{ w_{0}^{(t)} - w_{*} }^2 \ + \ \frac{16 L^2 \sigma_{*}^2}{3} \cdot \eta_t^4  \ + \ 2\sigma_{*}^2 \cdot \eta_t^2. 
\end{equation}
\end{lem}

%% Proof of Lemma 2.
\begin{proof}
Using the first line of \eqref{eq:fact2}, the optimality condition $\nabla{F}(w_{*}) = 0$ in (\textit{a}), and $(u+v)^2 \leq 2u^2 + 2v^2$ and the Cauchy-Schwarz inequality in (\textit{b}), for $i \in [n]$, we can derive
\begin{equation}\label{eq:scvx_est1a}
\hspace{-0ex}
\arraycolsep=0.1em
\begin{array}{lcl}
     \norms{ w_{i}^{(t)} - w_{0}^{(t)} }^2 & = &  \frac{\eta_t^2}{n^2} \big\Vert \sum_{j=0}^{i-1}  \nabla f ( w_{j}^{(t)} ; \pi^{(t)} (j + 1))  \big\Vert^2  \vspace{1.5ex}\\
    &\overset{\tiny(a)}{ = } &     \frac{\eta_t^2}{n^2}  \big\Vert \sum_{j=0}^{i-1} \big( \nabla f ( w_{j}^{(t)} ; \pi^{(t)} (j + 1)) -  \nabla f ( w_{*} ; \pi^{(t)} (j + 1)) \big) \vspace{1.5ex}\\
    && - {~} \sum_{j=i}^{n-1} \nabla f ( w_{*} ; \pi^{(t)} (j + 1))  \big\Vert^2 \vspace{1ex}\\
    & \overset{\tiny(b)}{\leq} &  \frac{2\eta_t^2 \cdot i}{n^2} \sum_{j=0}^{i-1}   \big\Vert  \nabla f ( w_{j}^{(t)} ; \pi^{(t)} (j + 1))  -  \nabla f ( w_{*} ; \pi^{(t)} (j + 1))   \big\Vert^2  \vspace{1.5ex}\\
    & & + {~}   \frac{2\eta_t^2 \cdot (n-i)}{n^2} \sum_{j=i}^{n-1} \big\Vert  \nabla f ( w_{*} ; \pi^{(t)} (j + 1)) \big\Vert^2.
\end{array}
\hspace{-4ex}
\end{equation}
Using \eqref{eq:Lsmooth_basic} and \eqref{defn_finite}, we can further estimate \eqref{eq:scvx_est1a} as
\begin{equation*}
\arraycolsep=0.1em
\begin{array}{lcl}
     \norms{ w_{i}^{(t)} - w_{0}^{(t)} }^2  & \overset{\tiny\eqref{eq:Lsmooth_basic}}{\leq} & \frac{2L^2\eta_t^2 \cdot i}{n^2}    \sum_{j=0}^{i-1} \norms{ w_{j}^{(t)} - w_{*} }^2 + \frac{2(n-i)\cdot \eta_t^2 }{n} \cdot \frac{1}{n} \sum_{j=0}^{n-1} \Vert \nabla f ( w_{*} ; \pi^{(t)} (j + 1)) \Vert^2 \vspace{1.5ex}\\
    & \overset{\tiny\eqref{defn_finite}}{\leq} & \frac{2 L^2 \eta_t^2 \cdot i }{n^2} \sum_{j=0}^{i-1} \norms{ w_{j}^{(t)} - w_{*} }^2 + \frac{2(n-i) \cdot \eta_t^2 }{n} \cdot \sigma_{*}^2.
\end{array}    
\end{equation*}    
This is exactly the first inequality of \eqref{eq_lem_bound_02}.

Next, by $\norms{ u + v }^2 \leq 2 \norms{ u }^2 + 2 \norm{ v }^2$ for any $u$ and $v$, for $i \in [n]$, using the last inequality we can easly show that
\begin{equation*}
\arraycolsep=0.2em
\begin{array}{lcl}
   \norms{w_{i}^{(t)} - w_{*}}^2 & \leq & 2 \norms{ w_{0}^{(t)} - w_{*} }^2 + 2 \norms{ w_{i}^{(t)} - w_{0}^{(t)} }^2  \vspace{1.5ex}\\
    & \leq &  2 \norms{ w_{0}^{(t)} - w_{*} }^2 + \eta_t^2 \cdot \frac{4 i L^2}{n^2} \sum_{j=0}^{i-1} \norms{w_{j}^{(t)} - w_{*} }^2 + \frac{4 \eta_t^2 (n-i)}{n} \cdot \sigma_{*}^2,
\end{array}
\end{equation*}
which proves the second estimate of \eqref{eq_lem_bound_02}.

Now, summing up the second estimate of \eqref{eq_lem_bound_02} from $j=0$ to $j=i-1$, we obtain 
\begin{equation*}
\arraycolsep=0.1em
\begin{array}{lcl}
\sum_{j=0}^{i-1}\norms{w_j^{(t)} - w_{*}}^2 &\leq & 2i \cdot \norms{ w_{0}^{(t)} - w_{*} }^2 + \frac{4 \eta_t^2 \sigma_{*}^2}{n} \sum_{j=0}^{i-1}(n - j)  + \frac{4L^2\eta_t^2}{n^2}\sum_{j=0}^{i-1} j \sum_{k=0}^{j-1}\norms{w_k^{(t)} - w_{*}}^2 \vspace{1ex}\\
& \leq & 2i \cdot \norms{ w_{0}^{(t)} - w_{*} }^2 + 4 \eta_t^2 \sigma_{*}^2 \cdot i + \frac{2L^2\eta_t^2 \cdot i(i-1)}{n^2}\sum_{j=0}^{i-1}\norms{w_j^{(t)} - w_{*}}^2 \vspace{1ex}\\
& \leq & 2i \big[  \norms{ w_{0}^{(t)} - w_{*} }^2 +   2\eta_t^2 \sigma_{*}^2 \big] + 2L^2\eta_t^2 \sum_{j=0}^{i-1}\norms{w_j^{(t)} - w_{*}}^2.
\end{array}
\end{equation*}
Here, we obtain the second inequality by first rearranging the double sum and then upper bound each term.
Since $0 < \eta_t \leq \frac{1}{2L}$, we have $1 - 2L^2\eta_t^2 \geq \frac{1}{2}$.
Rearranging the last inequality and using the last fact, we obtain \eqref{eq_lem_bound_02b}.

Finally, combining the first inequality of \eqref{eq_lem_bound_02} and \eqref{eq_lem_bound_02b}, we can derive that
\begin{equation*}
\arraycolsep=0.2em
\begin{array}{lcl}
\sum_{i=0}^{n-1}\norms{ w_{i}^{(t)} - w_{0}^{(t)} }^2 & \overset{\tiny\eqref{eq_lem_bound_02}}{\leq} &  \frac{2 L^2 \eta_t^2}{n^2}  \sum_{i=0}^{n-1} i \cdot \sum_{j=0}^{i-1} \norms{w_j^{(t)} - w_{*}}^2   +  \frac{2 \sigma_{*}^2 \eta_t^2}{n} \sum_{i=0}^{n-1}(n-i) \vspace{1ex}\\
    & \overset{\eqref{eq_lem_bound_02b}}{\leq} & \frac{2 L^2 \eta_t^2 }{n^2} \sum_{i=0}^{n-1} 4 i^2 \big[ \norms{w_0^{(t)} - w_{*}}^2 + 2 \eta_t^2\sigma_{*}^2 \big] +  \sigma_{*}^2 \eta_t^2(n+1) \vspace{1ex}\\
    &\leq & \frac{8 L^2 \eta_t^2 \cdot n}{3}  \norms{ w_{0}^{(t)} - w_{*} }^2 + \frac{16 L^2 \sigma_{*}^2 \cdot n }{3} \cdot \eta_t^4  + 2 \sigma_{*}^2 \eta_t^2 \cdot n,
\end{array}
\end{equation*}
which implies \eqref{eq_thm_weight_01} after multiplying both sides by $\frac{1}{n}$.
\end{proof}
%% End of proof.

%%% Lemma 2.1. 
\begin{lem}\label{le:ncvx_keybound1}
Suppose that Assumption \ref{ass_basic}$\mathrm{(ii)}$ and Assumption~\ref{ass_general_bounded_variance} hold for \eqref{ERM_problem_01}. 
Let $\sets{w_i^{(t)}}$ be generated by  Algorithm~\ref{sgd_replacement} with any shuffling strategy $\pi^{(t)}$ and a learning rate $\eta_i^{(t)} := \frac{\eta_t}{n} > 0$ for a given positive sequence $\sets{\eta_t}$ such that $ 0 < \eta_t \leq \frac{1}{L\sqrt{3}}$. 
Then, we have
\begin{equation}\label{eq:ncvx_keybound1}
\sum_{j=0}^{n-1} \Vert w_j^{(t)} - w_0^{(t)} \Vert^2 \leq  n \eta_t^2 \cdot \left[ \left( 3 \Theta + 2 \right) \Vert   \nabla F ( w_0^{(t)}) \Vert^2 + 3 \sigma^2 \right].
\end{equation}
If $\pi^{(t)}$ is uniformly sampled at random without replacement from $[n]$ and the learning rate $\eta_t$ satisfies $0 < \eta_t \leq \frac{1}{L\sqrt{3}}$ for all $t \geq 1$, then we have
\begin{equation}\label{eq:ncvx_keybound1_reshuffling}
\mathbb{E}\Big[ \sum_{j=0}^{n-1} \Vert w_j^{(t)} - w_0^{(t)} \Vert^2 \Big]  \leq 2 \eta_t^2 \cdot \left[ \left(\Theta + n \right)\mathbb{E} \big[ \Vert   \nabla F ( w_0^{(t)}) \Vert^2 \big] +  \sigma^2 \right]. 
\end{equation}
\end{lem}

\begin{proof}
First, from the first line of \eqref{eq:fact2}, by using $\norms{\sum_{i=1}^3 a_i}^2 \leq 3 \sum_{i=1}^3 \norms{a_i}^2$ in (\textit{a}) and the Cauchy-Schwarz inequality in (\textit{b}), we can derive that 
\begin{equation}\label{eq_bounded_var}
\hspace{-0.0ex}
\arraycolsep=0.1em
\begin{array}{lcl}
    \norms{ w_{i}^{(t)} - w_{0}^{(t)} }^2  & \overset{\tiny\eqref{eq:fact2}}{=}  &  \frac{i^2 \cdot \eta_t^2}{n^2} \big\Vert \frac{1}{i} \sum_{j=0}^{i-1} \nabla f ( w_{j}^{(t)} ; \pi^{(t)} (j + 1))  \big\Vert^2 \vspace{0.5ex}\\
    & \overset{\tiny(a)}{\leq} & \frac{3 i^2 \cdot \eta_t^2}{n^2} \bigg[ \Big\Vert \frac{1}{i} \sum_{j=0}^{i-1} \left( \nabla f ( w_{0}^{(t)} ; \pi^{(t)} (j + 1)) -   \nabla F ( w_0^{(t)}) \right)  \Big\Vert^2 + \Vert  \nabla F ( w_0^{(t)})\Vert^2 \bigg] \vspace{1ex}\\
    &&  + {~}  \frac{3 i^2 \cdot \eta_t^2}{n^2} \Big\Vert \frac{1}{i} \sum_{j=0}^{i-1} \left( \nabla f ( w_{j}^{(t)} ; \pi^{(t)} (j + 1))  - \nabla f ( w_{0}^{(t)} ; \pi^{(t)} (j + 1)) \right)  \Big\Vert^2 \vspace{0.5ex}\\
    & \overset{\tiny(b)}{\leq} &  \frac{3 i^2 \cdot \eta_t^2}{n^2} \frac{1}{i} \sum_{j=0}^{i-1}  \big\Vert  \nabla f ( w_{j}^{(t)} ; \pi^{(t)} (j + 1))  - \nabla f ( w_{0}^{(t)} ; \pi^{(t)} (j + 1))  \big\Vert^2 \vspace{1ex}\\
    && + {~} \frac{3 i^2 \cdot \eta_t^2}{n^2}  \Big[ \frac{1}{i} \sum_{j=0}^{i-1} \big\Vert  \nabla f ( w_{0}^{(t)} ; \pi^{(t)} (j + 1)) -   \nabla F ( w_0^{(t)})  \big\Vert^2 + \Vert  \nabla F ( w_0^{(t)})\Vert^2 \Big].
\end{array}
\hspace{-8ex}
\end{equation}
Let us introduce $\Delta := \sum_{j=0}^{n-1} \Vert w_j^{(t)} - w_0^{(t)} \Vert^2$.
Then, using \eqref{eq:Lsmooth_basic} from Assumption~\ref{ass_basic}(ii) and \eqref{eq:general_bounded_variance} from Assumption~\ref{ass_general_bounded_variance}, we can further derive from \eqref{eq_bounded_var} that
\begin{equation*}
\hspace{-0.0ex}
\arraycolsep=0.1em
\begin{array}{lcl}
\norms{ w_{i}^{(t)} - w_{0}^{(t)} }^2  
& \overset{\eqref{eq:Lsmooth_basic}}{\leq} &  \frac{3 i^2 \cdot \eta_t^2}{n^2}  \Big[ \frac{1}{i} \sum_{j=0}^{i-1} \big\Vert  \nabla f ( w_{0}^{(t)} ; \pi^{(t)} (j + 1)) -   \nabla F ( w_0^{(t)})  \big\Vert^2 + \Vert  \nabla F ( w_0^{(t)}) \Vert^2 \Big] \vspace{1ex}\\
     && + {~}  \frac{3 i^2 \cdot \eta_t^2 }{n^2} \frac{L^2}{i} \sum_{j=0}^{i-1} \Vert w_{j}^{(t)}  - w_{0}^{(t)}  \Vert^2 \vspace{1ex}\\
    & \leq &  \frac{3 i^2 \cdot \eta_t^2}{n^2}  \left[ \frac{n}{i} \cdot \frac{1}{n} \sum_{j=0}^{n-1} \big\Vert  \nabla f ( w_{0}^{(t)} ; \pi^{(t)} (j + 1)) -   \nabla F ( w_0^{(t)})  \big\Vert^2 + \Vert  \nabla F ( w_0^{(t)}) \Vert^2 \right] \vspace{1ex}\\
     && + {~} \frac{3 i L^2 \eta_t^2 }{n^2} \sum_{j=0}^{i-1} \Vert w_{j}^{(t)}  - w_{0}^{(t)} \Vert^2 \vspace{1ex}\\
    & \overset{\eqref{eq:general_bounded_variance}}{\leq} &    \frac{3 i L^2 \eta_t^2 }{n^2} \sum_{j=0}^{i-1}  \Vert w_{j}^{(t)}  - w_{0}^{(t)} \Vert^2 
    + \frac{3 i^2 \cdot \eta_t^2}{n^2}  \left[ \frac{n}{i} \big( \Theta \Vert   \nabla F ( w_0^{(t)}) \Vert^2 + \sigma^2 \big) + \Vert  \nabla F ( w_0^{(t)}) \Vert^2 \right]  \vspace{1ex}\\
    & \leq &  \frac{3 i L^2 \eta_t^2 }{n^2} \Delta + \frac{3 \eta_t^2}{n^2}  \left[ n\cdot i \big( \Theta \Vert   \nabla F ( w_0^{(t)}) \Vert^2 + \sigma^2 \big) + i^2 \Vert  \nabla F ( w_0^{(t)}) \Vert^2 \right].
\end{array}
\hspace{-0ex}
\end{equation*}
Using this estimate and the definition of $\Delta$, we have
 \begin{equation*}
\hspace{-0.5ex}
\arraycolsep=0.2em
\begin{array}{lcl}
    \Delta  & = & \sum_{i=0}^{n-1} \norms{ w_{i}^{(t)} - w_{0}^{(t)} }^2 \vspace{1ex} \\ 
    & \leq & \frac{3 L^2 \eta_t^2 }{n^2} \big( \sum_{i=0}^{n-1} i \big) \Delta + \frac{3 \eta_t^2}{n^2}  \left[ n \big( \Theta \Vert   \nabla F ( w_0^{(t)}) \Vert^2 + \sigma^2 \big) \sum_{i=0}^{n-1} i 
    + \Vert  \nabla F ( w_0^{(t)}) \Vert^2 \sum_{i=0}^{n-1} i^2 \right] \vspace{1ex}\\
    & \leq & \frac{3 L^2 \eta_t^2}{2} \Delta + \frac{3 n \eta_t^2}{2}   \left( \Theta \Vert   \nabla F ( w_0^{(t)}) \Vert^2 + \sigma^2 \right)  +  n\eta^2 \Vert  \nabla F ( w_0^{(t)}) \Vert^2  \vspace{1ex}\\
    & \leq & \frac{3L^2\eta_t^2}{2} \Delta + \frac{n \eta_t^2}{2}  \left[ \left( 3\Theta + 2 \right) \Vert   \nabla F ( w_0^{(t)})  \Vert^2 + 3 \sigma^2 \right], 
\end{array}    
\hspace{-6ex}
\end{equation*}
where the second inequality follows since $\sum_{i=0}^{n-1} i = \frac{n(n-1)}{2} \leq \frac{n^2}{2}$ and $\sum_{i=0}^{n-1} i^2 = \frac{n(n-1)(2n-1)}{6} \leq \frac{n^3}{3}$.
Rearranging the last inequality and noticing that $1 - \frac{3L^2\eta_t^2}{2} \geq \frac{1}{2}$ due to the condition $0 < \eta_t \leq \frac{1}{\sqrt{3}L}$ on $\eta_t$, we obtain \eqref{eq:ncvx_keybound1}.

%%%%%%%%%%%%%%%%%%%%%%
Now, let $\pi^{(t)} := (\pi^{(t)}(1), \cdots, \pi^{(t)}(n))$ be sampled uniformly at random without replacement from $[n]$.
For each epoch $t \geq 1$, we denote  $\mathcal{F}_t := \sigma(w_0^{(1)},\cdots,w_0^{(t)})$, the $\sigma$-algebra generated by the iterates of Algorithm~\ref{sgd_replacement}. 
Similar to the proof of \eqref{eq_bounded_var}, we can show that
\begin{equation*}
\hspace{-0.0ex}
\arraycolsep=0.1em
\begin{array}{lcl}
    \norms{ w_{i}^{(t)} - w_{0}^{(t)} }^2   & \leq &  \frac{3 i^2 \cdot \eta_t^2}{n^2}  \big\Vert \frac{1}{i}  \sum_{j=0}^{i-1} \big(  \nabla f ( w_{0}^{(t)} ; \pi^{(t)} (j + 1)) -   \nabla F ( w_0^{(t)})  \big) \big\Vert^2  
    +  \frac{3 i^2 \cdot \eta_t^2}{n^2}  \Vert  \nabla F ( w_0^{(t)}) \Vert^2  \vspace{1ex}\\
     && + {~} \frac{3 i L^2 \eta_t^2 }{n^2} \sum_{j=0}^{n-1} \Vert w_{j}^{(t)}  - w_{0}^{(t)} \Vert^2.
\end{array}
\end{equation*}
Taking expectation conditioned on $\mathcal{F}_t$ both sides of this estimate and using $\Delta$, we get
\begin{equation}\label{eq_bounded_var22}
\hspace{-0ex}
\arraycolsep=0.2em
\begin{array}{lcl}
    \mathbb{E} \left[\norms{ w_{i}^{(t)} - w_{0}^{(t)} }^2 \mid \mathcal{F}_t \right]     &\leq &  \frac{3 i^2 \cdot \eta_t^2}{n^2} \mathbb{E} \Big[ \big\Vert \frac{1}{i} \sum_{j=0}^{i-1} \big( \nabla f ( w_{0}^{(t)} ; \pi^{(t)} (j + 1)) -   \nabla F ( w_0^{(t)}) \big)  \big\Vert^2  \mid \mathcal{F}_t \Big] \vspace{1ex}\\
    &&  + {~} \frac{3 i^2 \cdot \eta_t^2}{n^2} \Vert  \nabla F ( w_0^{(t)})\Vert^2  +  \frac{3 i L^2 \eta_t^2 }{n^2} \mathbb{E} [\Delta \mid \mathcal{F}_t].
\end{array}
\hspace{-2ex}
\end{equation}
Applying  Lemma~\ref{le:lem_key_rr} and \eqref{eq:general_bounded_variance}, we can upper bound the first term of \eqref{eq_bounded_var22} as
\begin{equation*}
\hspace{-0.0ex}
\arraycolsep=0.2em
\begin{array}{lcl}
\mathcal{T}_{[2]} & := & \mathbb{E} \Big[ \big\Vert  \frac{1}{i} \sum_{j=0}^{i-1}  \nabla f ( w_{0}^{(t)} ; \pi^{(t)} ( j + 1 ) ) - \nabla F ( w_{0}^{(t)})  \big\Vert ^2 \mid  \mathcal{F}_t \Big] \vspace{1ex}\\
    &= & \frac{n-i}{i(n-1)} \frac{1}{n} \sum_{j=0}^{n-1} \big\Vert \nabla f ( w_{0}^{(t)} ;  j+1 ) - \nabla F ( w_{0}^{(t)})  \big\Vert ^2 \\
    &\overset{\eqref{eq:general_bounded_variance}}{\leq} &  \frac{i(n-i)}{i^2(n-1)} \Big[\Theta \big\Vert \nabla  F(w_{0}^{(t)})  \big\Vert ^2 + \sigma^2 \Big]. 
\end{array}
\hspace{-0ex}
\end{equation*}
Substituting this inequality into \eqref{eq_bounded_var22}, we get
\begin{equation*}
\hspace{-0.0ex}
\arraycolsep=0.1em
\begin{array}{lcl}
    \mathbb{E} \big[\norms{ w_{i}^{(t)} - w_{0}^{(t)} }^2 \mid \mathcal{F}_t \big]      &\leq &  \frac{3 i L^2 \eta_t^2 }{n^2} \mathbb{E} [\Delta \mid \mathcal{F}_t]  
    + \frac{3 \eta_t^2}{n^2} \frac{i(n-i)}{(n-1)} \big[\Theta \big\Vert \nabla  F(w_{0}^{(t)})  \big\Vert ^2 + \sigma^2 \big] +  \frac{3 i^2 \cdot \eta_t^2}{n^2} \Vert  \nabla F ( w_0^{(t)})\Vert^2. 
\end{array}
\hspace{-0ex}
\end{equation*}
Taking full expectation over $\mathcal{F}_t$ of both sides of the last estimate, we have
\begin{equation*}
\hspace{-0.0ex}
\arraycolsep=0.1em
\begin{array}{lcl}
    \mathbb{E} \big[\norms{ w_{i}^{(t)} - w_{0}^{(t)} }^2  \big]      &\leq &  \frac{3 i L^2 \eta_t^2 }{n^2} \mathbb{E} [\Delta]  
    + \frac{3 \eta_t^2}{n^2} \frac{i(n-i)}{(n-1)} \Big[\Theta \mathbb{E} \big[ \big \Vert \nabla  F(w_{0}^{(t)})  \big\Vert ^2 \big] + \sigma^2 \Big] + \frac{3 i^2 \cdot \eta_t^2}{n^2} \mathbb{E} \big[ \Vert  \nabla F ( w_0^{(t)})\Vert^2 \big]. 
\end{array}
\hspace{-0ex}
\end{equation*}
Using the last estimate and the definition of $\Delta$, we can derive that
 \begin{equation*}
\arraycolsep=0.2em
\begin{array}{lcl}
    \mathbb{E} [\Delta]  & = & \sum_{i=0}^{n-1} \mathbb{E} \big[ \norms{ w_{i}^{(t)} - w_{0}^{(t)} }^2 \big] \vspace{1ex} \\ 
    & \leq & \frac{3 L^2 \eta_t^2 }{n^2} \cdot \mathbb{E} [\Delta] \big( \sum_{i=0}^{n-1} i \big) + \frac{3 \eta_t^2}{n^2(n-1)}  \cdot \Big[ \Theta \mathbb{E} \big[  \Vert   \nabla F ( w_0^{(t)}) \Vert^2 \big] + \sigma^2 \Big] \cdot \big[ \sum_{i=0}^{n-1} i(n-i) \big] \vspace{1ex}\\
    && + {~}  \frac{3 \eta_t^2}{n^2}  \cdot \mathbb{E} \big[ \Vert  \nabla F ( w_0^{(t)}) \Vert^2 \big] \big( \sum_{i=0}^{n-1} i^2 \big)  \vspace{1ex}\\
    & \leq & \frac{3 L^2 \eta_t^2}{2} \cdot \mathbb{E} [\Delta] +  \eta_t^2 \cdot  \Big[ \big( \Theta + n \big) \mathbb{E} \big[ \Vert   \nabla F ( w_0^{(t)}) \Vert^2 \big] + \sigma^2 \Big],
\end{array}    
\end{equation*}
where the second inequality follows since $\sum_{i=0}^{n-1} i = \frac{n(n-1)}{2} \leq \frac{n^2}{2}$, $\sum_{i=0}^{n-1} i^2 = \frac{n(n-1)(2n-1)}{6} \leq \frac{n^3}{3}$, and $\sum_{i=0}^{n-1} i(n-i) = \frac{(n-1)n(n+1)}{6} \leq \frac{n^2(n-1)}{3}$.
Now, since $0 < \eta_t \leq \frac{1}{\sqrt{3}L}$, we have $1 - \frac{3L^2\eta_t^2}{2} \geq \frac{1}{2}$.
Rearranging the last inequality and using this fact, we finally get \eqref{eq:ncvx_keybound1_reshuffling}.
\end{proof}
%%% End of proof.

Let us improve Lemma~\ref{lem_bound_weights} above by using a randomized reshuffling strategy, the convexity of each $f(\cdot; i)$ for $i \in [n]$, and the strong convexity of $F$.
%%% Lemma 7.
\begin{lem}\label{le:scvx_key_bound2}
Suppose that Assumption~\ref{ass_basic}$\mathrm{(ii)}$ holds and each $f(\cdot; i)$ is convex for $i \in [n]$.
Let $\sets{w_i^{(t)}}$ be generated by  Algorithm~\ref{sgd_replacement} and $\sigma_{*}^2$ be defined by \eqref{defn_finite}.
Let $\pi^{(t)} := (\pi^{(t)}(1), \cdots, \pi^{(t)}(n))$ be sampled uniformly at random without replacement from $[n]$.
Then, if we choose $\eta_t$ such that $0 < \eta_t \leq \frac{\sqrt{5}-1}{2L}$, then, for any $t \geq 1$, we have
\begin{equation}\label{eq:scvx_key_bound2}
\mathbb{E}\big[ \norms{\tilde{w}_t - w_{*}}^2 \big]  \ \leq \ \mathbb{E} \big[ \norms{\tilde{w}_{t-1} - w_{*}}^2 \big] \ - \ 2\eta_t \cdot \left[ F(\tilde{w}_{t-1}) - F(w_{*}) \right]  \ + \  \frac{2L\eta_t^3\sigma_{*}^2}{3n}.
\end{equation}
\end{lem}

%%% Proof of Lemma 8.
\begin{proof}
Using the first line of \eqref{eq:fact2}, with the same proof as of \eqref{eq:scvx_est1a}, we have
\begin{equation*}
\arraycolsep=0.2em
\begin{array}{lcl}
\norms{w_i^{(t)} - \tilde{w}_{t-1}}^2 &\leq & \frac{2\eta_t^2 \cdot i}{n^2}\sum_{j=0}^{i-1}\Vert \nabla{f}(w_j^{(t)}; \pi^{(t)}(j+1)) - \nabla{f}(w_{*}; \pi^{(t)}(j+1))\Vert^2 + \frac{2\eta_t^2B_i^{*}}{n^2},
\end{array}
\end{equation*}
where $B_i^{*} := \Vert \sum_{j=i}^{n-1}\nabla{f}(w_{*}; \pi^{(t)}(j+1)) \Vert^2$.
Summing up this inequality from $i := 0$ to $i := n-1$, we can derive that
\begin{align}\label{eq:scvx_est1}
\hspace{-2ex}
\sum_{i=0}^{n-1}\norms{w_i^{(t)} - \tilde{w}_{t-1}}^2 & \leq  \frac{2\eta_t^2}{n^2} \sum_{i=0}^{n-1} i \cdot  \sum_{j=0}^{i-1} \Vert \nabla{f}(w_j^{(t)}; \pi^{(t)}(j+1)) - \nabla{f}(w_{*}; \pi^{(t)}(j+1))\Vert^2  +  \frac{2\eta_t^2}{n^2} {\displaystyle \sum_{i=0}^{n-1}} B_i^{*} \nonumber\\
&\leq  \eta_t^2\sum_{j=0}^{n-1} \Vert \nabla{f}(w_j^{(t)}; \pi^{(t)}(j+1)) - \nabla{f}(w_{*}; \pi^{(t)}(j+1))\Vert^2 \ + \  \frac{2\eta_t^2B^{*}}{n^2}.
\end{align}
where $B^{*} := \sum_{i=0}^{n-1}B_i^{*} =  \sum_{i=0}^{n-1}\Vert \sum_{j=i}^{n-1}\nabla{f}(w_{*}; \pi^{(t)}(j+1) ) \Vert^2$.

Next, using the second line of \eqref{eq:fact2} and the Cauchy-Schwarz inequality, we have
\begin{equation}\label{eq:scvx_est2} 
\arraycolsep=0.2em
\begin{array}{lcl}
\norms{\tilde{w}_t - w_{*}}^2 &= & \norms{\tilde{w}_{t-1} - w_{*}}^2 + \frac{2\eta_t}{n}\sum_{j=0}^{n-1}\iprods{\nabla{f}(w_j^{(t)}; \pi^{(t)}(j+1)), w_{*} - \tilde{w}_{t-1}} \vspace{1.5ex}\\
&& + {~} \frac{\eta_t^2}{n^2}\Vert \sum_{j=0}^{n-1} (\nabla{f}(w_j^{(t)}; \pi^{(t)}(j+1)) - \nabla{f}(w_{*}; \pi^{(t)}(j+1)) ) \Vert^2 \vspace{1.5ex}\\
&\leq & \norms{\tilde{w}_{t-1} - w_{*}}^2 + \frac{2\eta_t}{n}\sum_{j=0}^{n-1}\iprods{\nabla{f}(w_j^{(t)}; \pi^{(t)}(j+1)), w_{*} - \tilde{w}_{t-1} } \vspace{1.5ex}\\
&& + {~} \frac{\eta_t^2}{n} \sum_{j=0}^{n-1}\Vert \nabla{f}(w_j^{(t)}; \pi^{(t)}(j+1)) - \nabla{f}(w_{*}; \pi^{(t)}(j+1)) \Vert^2.
\end{array}
\end{equation}
We can upper bound the second term on the right-hand side of \eqref{eq:scvx_est2} as follows:
\begin{equation*}
\arraycolsep=0.2em
\begin{array}{lcl}
\mathcal{T}_{[1]} & := & \sum_{j=0}^{n-1}\iprods{\nabla{f}(w_j^{(t)}; \pi^{(t)}(j+1)), w_{*} - \tilde{w}_{t-1}}  \vspace{1.5ex}\\
& = & \sum_{j=0}^{n-1}\iprods{\nabla{f}(w_j^{(t)}; \pi^{(t)}(j+1)), w_{*} -  w_j^{(t)}} + \sum_{j=0}^{n-1}\iprods{\nabla{f}(w_j^{(t)}; \pi^{(t)}(j+1)),  w_j^{(t)} - \tilde{w}_{t-1}} \vspace{1ex}\\
& \overset{(a)}{\leq} & \sum_{j=0}^{n-1}\iprods{\nabla{f}(w_j^{(t)}; \pi^{(t)}(j+1)), w_{*} -  w_j^{(t)}} + \frac{L}{2}\sum_{j=0}^{n-1}\norms{w_j^{(t)} - \tilde{w}_{t-1}}^2 \vspace{1.5ex}\\
&& + {~} \sum_{j=0}^{n-1} \big[f(w_j^{(t)}; \pi^{(t)}(j+1)) - f(\tilde{w}_{t-1}; \pi^{(t)}(j+1)) \big] \vspace{1.5ex}\\
& = & -\sum_{j=0}^{n-1}\big[f(w_{*}; \pi^{(t)}(j+1)) - f(w_j^{(t)}; \pi^{(t)}(j+1)) -  \iprods{\nabla{f}(w_j^{(t)}; \pi^{(t)}(j+1)), w_{*} -  w_j^{(t)}} \big] \vspace{1.5ex}\\
&& + {~} \frac{L}{2}\sum_{j=0}^{n-1}\norms{w_j^{(t)} - \tilde{w}_{t-1}}^2 + \sum_{j=0}^{n-1} \big[f(w_{*}; \pi^{(t)}(j+1)) - f(\tilde{w}_{t-1}; \pi^{(t)}(j+1)) \big]  \vspace{1ex}\\
& \overset{(b)}{\leq} & \frac{L}{2}\sum_{j=0}^{n-1} \norms{w_j^{(t)} - \tilde{w}_{t-1}}^2 - \frac{1}{2L}\sum_{j=0}^{n-1}\Vert \nabla{f}(w_j^{(t)}; \pi^{(t)}(j+1)) - \nabla{f}(w_{*}; \pi^{(t)}(j+1)) \Vert^2 \vspace{1.5ex}\\
&& - {~} n\big[ F(\tilde{w}_{t-1}) - F(w_{*}) \big].  
\end{array}
\end{equation*}
Here, we have used  in (a) the following estimate
\begin{equation*}
\iprods{\nabla{f}(w_j^{(t)};\cdot), w_j^{(t)} - \tilde{w}_{t-1}} \leq f(w_j^{(t)};\cdot) - f(\tilde{w}_{t-1}; \cdot) + \frac{L}{2}\norms{w_j^{(t)} - \tilde{w}_{t-1}}^2
\end{equation*}
and  in (b) the following two estimates:
 \begin{equation*}
 \begin{array}{ll}
&f(w_{*};\cdot) - f(w_j^{(t)};\cdot)  - \iprods{\nabla{f}(w_j^{(t)};\cdot), w_{*} - w_j^{(t)}} \geq \frac{1}{2L}\norms{\nabla{f}(w_j^{(t)};\cdot) - \nabla{f}(w_{*};\cdot)}^2, \vspace{1ex}\\ 
\text{and} &\sum_{j=0}^{n-1}\big[f(w_{*}; \pi^{(t)}(j+1)) - f(\tilde{w}_{t-1}; \pi^{(t)}(j+1)) \big] = n\left[ F(w_{*}) - F(\tilde{w}_{t-1}) \right].
\end{array} 
 \end{equation*}

Substituting \eqref{eq:scvx_est1} into $\mathcal{T}_{[1]}$, we can further upper bound it as 
\begin{equation*}
\arraycolsep=0.2em
\begin{array}{lcl}
\mathcal{T}_{[1]} & \leq &  \left( \frac{L\eta_t^2}{2}  - \frac{1}{2L}\right)\sum_{j=0}^{n-1}\Vert \nabla{f}(w_j^{(t)}; \pi^{(t)}(j+1) - \nabla{f}(w_{*}; \pi^{(t)}(j+1) \Vert^2  \vspace{1ex}\\
&& - {~}  n\left[ F(\tilde{w}_{t-1}) - F(w_{*}) \right] + \frac{L\eta_t^2B^{*}}{n^2}.
\end{array}
\end{equation*}
Using this upper bound of $\mathcal{T}_{[1]}$ into \eqref{eq:scvx_est2}, we get
\begin{equation*} 
\arraycolsep=0.1em
\begin{array}{lcl}
\norms{\tilde{w}_t - w_{*}}^2 &\leq & \norms{\tilde{w}_{t-1} - w_{*}}^2 - 2\eta_t \left[ F(\tilde{w}_{t-1}) - F(w_{*}) \right]  + \frac{2L\eta_t^3B^{*}}{n^3} \vspace{1ex}\\
&& + {~} \left( \frac{L\eta_t^3}{n}  - \frac{\eta_t}{Ln} + \frac{\eta_t^2}{n} \right) \sum_{j=0}^{n-1}\Vert \nabla{f}(w_j^{(t)}; \pi^{(t)}(j+1)) - \nabla{f}(w_{*}; \pi^{(t)}(j+1)) \Vert^2.
\end{array}
\end{equation*}
Let us choose $\eta_t > 0$ such that $L^2\eta_t^3 + L\eta_t^2 - \eta_t  \leq 0$ (or equivalently, $0 < \eta_t \leq \frac{\sqrt{5}-1}{2L}$).
Then, this estimate reduces to 
\begin{equation}\label{eq:scvx_est4}
\begin{array}{lcl}
\norms{\tilde{w}_t - w_{*}}^2  \leq  \norms{\tilde{w}_{t-1} - w_{*}}^2 - 2\eta_t \left[ F(\tilde{w}_{t-1}) - F(w_{*}) \right]   + \frac{2L\eta_t^3B^{*}}{n^3}.
\end{array}
\end{equation}
If $\pi^{(t)} := (\pi^{(t)}(1),\cdots,\pi^{(t)}(n))$ is uniformly  sampled  at  random  without  replacement  from $[n]$, 
then by Lemma~\ref{le:lem_key_rr}, we have $\mathbb{E}\left[ \frac{1}{n-i} \sum_{j=i}^{n-1} \nabla f ( w_{*} ; \pi^{(t)} (j + 1)) \right] = \nabla{F} ( w_{*} )$ and
\begin{equation*}
\arraycolsep=0.1em
\begin{array}{lcl}
    \mathbb{E}[B^{*}] &= & \mathbb{E} \Big[ \sum_{i=0}^{n-1} \big\Vert \sum_{j=i}^{n-1} \nabla f ( w_{*} ; \pi^{(t)} (j + 1))  \big\Vert^2 \Big] \vspace{1.25ex}\\
    & = & \sum_{i=0}^{n-1} (n-i)^2 \mathbb{E} \Big[ \big\Vert \frac{1}{n-i} \sum_{j=i}^{n-1} \nabla{f} ( w_{*} ; \pi^{(t)} (j + 1))  \big\Vert^2 \Big] \vspace{1.25ex}\\
    &=&  \sum_{i=0}^{n-1}(n - i)^2 \mathbb{E} \Big[ \big\Vert \frac{1}{n-i} \sum_{j=i}^{n-1} \nabla{f} ( w_{*} ; \pi^{(t)} (j + 1))  - \nabla{F} ( w_* )  \big\|^2 \Big] \vspace{1.25ex}\\
    &= & \sum_{i=0}^{n-1}\frac{(n-i)^2 i}{(n-i)(n-1)} \frac{1}{n} \sum_{j=0}^{n-1} \big\Vert \nabla{f} ( w_{*} ; \pi^{(t)} (j + 1))  \big\Vert^2 \vspace{1.25ex}\\
    &\overset{\eqref{defn_finite}}{=} & \frac{\sigma_{*}^2}{n-1} \sum_{i=0}^{n-1} i(n-i) \vspace{1.25ex}\\
    & = &    \frac{n(n+1)\sigma_{*}^2}{6}. 
\end{array}
\end{equation*}
Taking expectation both sides of \eqref{eq:scvx_est4} and using this upper bound of $B^{*}$, we obtain \eqref{eq:scvx_key_bound2}.
\end{proof}
%%% End of the proof.

Finally, we will need the following bound on $F$ in the sequel.
%%% Lemma 8.
\begin{lem}\label{le:F_bound}
Suppose that Assumption \ref{ass_basic}$\mathrm{(ii)}$  holds for \eqref{ERM_problem_01}. 
Let $\sets{w_i^{(t)}}$ be generated by  Algorithm~\ref{sgd_replacement} with any shuffling strategy $\pi^{(t)}$ and a learning rate $\eta_i^{(t)} := \frac{\eta_t}{n} > 0$ for a given positive sequence $\sets{\eta_t}$ such that $ 0 < \eta_t \leq \frac{1}{L}$. 
Then, for any $t \geq 1$, we have
\begin{equation}\label{eq_key_thm_01_est}
    F( w_0^{(t+1)} ) \leq F( w_0^{(t)} ) - \frac{\eta_t}{2} \norms{ \nabla F( w_0^{(t)} ) }^2 + \frac{L^2 \eta_t}{2n}  \sum_{i=0}^{n-1} \norms{ w_i^{(t)} - w_0^{(t)} }^2.
\end{equation}
\end{lem}

%%% Proof of Lemma 8.
\begin{proof}
Since $F$ is $L$-smooth by Assumption \ref{ass_basic}$\mathrm{(ii)}$, we can derive 
\begin{equation*} 
\arraycolsep=0.1em
\begin{array}{lcl}
    F( w_0^{(t+1)} )  & \overset{\eqref{eq:Lsmooth}}{\leq}  &   F( w_0^{(t)} ) + \nabla F( w_0^{(t)} )^{\top}(w_0^{(t+1)} - w_0^{(t)}) + \frac{L}{2}\norms{w_0^{(t+1)} - w_0^{(t)}}^2 \vspace{1ex}\\
    &\overset{\tiny\eqref{eq:fact2}}{=}  & F( w_0^{(t)} ) - \eta_t \nabla F( w_0^{(t)} )^\top \left( \frac{1}{n} \sum_{i=0}^{n-1} \nabla f (w_i^{(t)} ; \pi^{(t)} (i + 1) ) \right)  \vspace{1ex}\\
    &&  + {~} \frac{L \eta_t^{2}}{2} \big\Vert \frac{1}{n}\sum_{i=0}^{n-1} \nabla f (w_i^{(t)} ; \pi^{(t)} (i + 1) ) \big\Vert^2  \vspace{1ex}\\
    & \overset{\tiny(a)}{=} &   F( w_0^{(t)} ) - \frac{\eta_t}{2} \norms{ \nabla F( w_0^{(t)} )}^2 + \frac{\eta_t}{2} \big\Vert \nabla F( w_0^{(t)} )  - \frac{1}{n}\sum_{i=0}^{n-1} \nabla f (w_i^{(t)} ; \pi^{(t)} (i + 1) ) \big\Vert^2 \vspace{1ex}\\
    &&  - {~} \frac{\eta_t}{2} \left( 1 - L\eta_t  \right) \big\Vert \frac{1}{n}\sum_{i=0}^{n-1} \nabla f (w_i^{(t)} ; \pi^{(t)} (i + 1) ) \big\Vert^2 \vspace{1ex}\\
    & \overset{\tiny(b)}{\leq}  & F( w_0^{(t)} )  +  \frac{\eta_t}{2} \big\Vert \frac{1}{n}\sum_{i=0}^{n-1} \nabla f (w_0^{(t)} ; \pi^{(t)} (i + 1) )  - \frac{1}{n}\sum_{i=0}^{n-1} \nabla f (w_i^{(t)} ; \pi^{(t)} (i + 1) ) \big\Vert^2 \vspace{1ex}\\
    && - {~}  \frac{\eta_t}{2} \norms{ \nabla F( w_0^{(t)} )}^2 \vspace{1ex}\\
    & \overset{\tiny(c)}{\leq} &  F( w_0^{(t)} ) +  \frac{\eta_t}{2n} \sum_{i=0}^{n-1} \big\Vert \nabla f (w_0^{(t)} ; \pi^{(t)} (i + 1) )  - \nabla f (w_i^{(t)} ; \pi^{(t)} (i + 1) ) \big\Vert^2 \vspace{1ex}\\
     && - {~} \frac{\eta_t}{2} \norms{ \nabla F( w_0^{(t)} ) }^2  \vspace{1ex}\\
    & \overset{\tiny\eqref{eq:Lsmooth_basic}}{\leq} & F( w_0^{(t)} ) - \frac{\eta_t}{2} \norms{ \nabla F( w_0^{(t)} ) }^2 + \frac{L^2 \eta_t}{2n}  \sum_{i=0}^{n-1} \norms{ w_i^{(t)} - w_0^{(t)} }^2,
\end{array}
\end{equation*}
where (\textit{a}) follows from $u^{\top}v = \frac{1}{2}(\norms{u}^2 + \norms{v}^2 - \norms{u - v}^2)$, (\textit{b}) follows from the fact that $\eta_t \leq \frac{1}{L}$, and (\textit{c}) is from the Cauchy-Schwarz inequality. 
\end{proof}
%%% End of the proof.

%%%%%%%%%%%%%%%%%%%%%%%%%%%%%%%%%%%%%%%%%%%%%
%%% C. Proof of Strongly convex case.
%%%%%%%%%%%%%%%%%%%%%%%%%%%%%%%%%%%%%%%%%%%%%
\section{Convergence Analysis for Strongly Convex Case}\label{sec_appendix_strongly_convex}
In this section, we present the full proof of the results in the main text of Section~\ref{sec_analysis_01}. 

%%% Proofs of Theorem~ 6 and Corollary 4.
\subsection{Proofs of  Theorem~\ref{thm_main_result_02_const} and Theorem~\ref{thm_main_result_02}: The strongly convex case}
\label{subsec:appendix_Th12_proof}

%%% The proof of Theorem 1.b
\begin{proof}[\mytxtbi{The proof of Theorem \ref{thm_main_result_02_const}}]
Using \eqref{eq_thm_weight_01}, we can further estimate \eqref{eq_key_thm_01_est} as follows:
\begin{equation*}
\arraycolsep=0.2em
\begin{array}{lcl}
    F( w_0^{(t+1)} ) & \leq &  F( w_0^{(t)} ) - \frac{\eta_t}{2} \norms{ \nabla F( w_0^{(t)} )  }^2 + \frac{L^2 \eta_t}{2} \frac{1}{n}\sum_{i=0}^{n-1} \norms{ w_i^{(t)} - w_0^{(t)}}^2 \vspace{1ex}\\
    & \overset{\tiny\eqref{eq_thm_weight_01}}{\leq} &  F( w_0^{(t)} ) - \frac{\eta_t}{2} \norms{ \nabla F( w_0^{(t)} ) }^2 \vspace{1ex}\\
    && + {~} \frac{L^2 \eta_t}{2} \left( \eta_t^2 \cdot \frac{8 L^2}{3} \norms{ w_{0}^{(t)} - w_{*} }^2 \ + \ \eta_t^4 \cdot \frac{16 L^2 \sigma_{*}^2 }{3} \ + \ \eta_t^2 \cdot 2 \sigma_{*}^2 \right) \vspace{1ex}\\
    & \overset{\tiny \eqref{eq:stronglyconvex}}{\leq} & F( w_0^{(t)} ) - \mu \eta_t \big[ F( w_0^{(t)} ) - F(w_{*}) \big] +  \frac{8 L^4 \eta_t^3}{3 \mu} \big[ F( w_0^{(t)} ) - F(w_{*}) \big] \vspace{1ex}\\
    &&   + {~} \frac{8 L^4 \sigma_{*}^2}{3}\cdot \eta_t^5  + L^2 \sigma_{*}^2 \cdot \eta_t^3, 
\end{array}
\end{equation*}
Subtracting $F(w_{*})$ from both sides of  the last inequality, we can further derive
\begin{equation}\label{eq_rec_01}
    F( w_0^{(t+1)} ) - F(w_{*}) \leq \left[ 1 - \eta_t \left( \mu - \tfrac{8 L^4}{3 \mu} \eta_t^2 \right) \right] \big[ F( w_0^{(t)} ) - F(w_{*}) \big] + L^2 \sigma_{*}^2 \eta_t^3 \left( 1 +  \tfrac{8 L^2  \eta_t^2  }{3} \right).
\end{equation}
Now, assume that $0 < \eta_t < \sqrt{\frac{3}{8}} \frac{\mu}{L^2}$. 
Then, one can show that
\begin{equation*}
\mu - \frac{4 L^4}{3 \mu} \eta_t^2 \geq \mu - \frac{2}{3} \mu = \frac{\mu}{3} > 0~~~\text{and}~~~
\frac{8 L^4  \sigma_{*}^2}{3}\eta_t^2   < \frac{8 L^4  \sigma_{*}^2}{3}\cdot  \frac{3}{8} \frac{\mu^2}{L^4} = \mu^2  \sigma_{*}^2.
\end{equation*}
Using these bounds into \eqref{eq_rec_01}, we can further upper bound it as
\begin{equation}\label{eq_rec_01b}
    F( w_0^{(t+1)} ) - F(w_{*}) \leq \left( 1 - \frac{\mu}{3}\eta_t \right) \big[F( w_0^{(t)} ) - F(w_{*})\big] + \eta_t^3 \left( \mu^2  + L^2 \right)  \sigma_{*}^2.
\end{equation}
Note that we have imposed $\eta_t \leq \min\set{\frac{1}{2L}, \sqrt{\frac{3}{8}} \frac{\mu}{L^2}}$ due to $\eta_t \leq \frac{1}{2L}$ in Lemma~\ref{lem_bound_weights}.

Now, let us define $Y_t := F( w_0^{(t)} ) - F(w_{*}) = F(\tilde{w}_{t-1}) - F(w_{*})\geq 0$, $\rho := \frac{\mu}{3}$, and $D := (\mu^2  + L^2)\sigma_{*}^2$.
The estimate \eqref{eq_rec_01b} becomes
\begin{equation*}
Y_{t+1} \leq (1 - \rho \cdot\eta_t) Y_t + D \eta_t^3. 
\end{equation*} 
Applying \eqref{eq_rate_t3} of Lemma~\ref{lem_general_framework} with $q=2$ and $\eta := \frac{2\log(T)}{\rho T} = \frac{6\log(T)}{\mu T}$, we obtain
\begin{equation*} 
\arraycolsep=0.2em
\begin{array}{lcl}
 Y_{T+1} & \leq & (1 - \rho \eta)^{T} Y_1 + \frac{D\eta^{2}[ 1 - (1 - \rho \eta)^{T}]}{\rho} \leq  Y_1\exp( -\rho\eta T) + \frac{D\eta^2}{\rho} \vspace{1ex}\\
 &= & \big[F(\tilde{w}_0) - F_{*} \big]\exp\big( -2\log(T) \big) + \frac{54(\mu^2 + L^2)\sigma_{*}^2 \log(T)^2}{\mu^3 T^2}.
 \end{array}
\end{equation*}
This estimate leads to $F(\tilde{w}_T) - F(w_{*}) \leq \frac{\big[F(\tilde{w}_0) - F(w_{*}) \big]}{T^2} + \frac{54(\mu^2 + L^2)\sigma_{*}^2 \log(T)^2}{\mu^3 T^2}$, which is exactly \eqref{eq:scvx_bound1_const} after taking expectation.
To guarantee $\eta_t = \frac{6\log(T)}{\mu T} \leq \min\set{\frac{1}{2L}, \sqrt{\frac{3}{8}} \frac{\mu}{L^2}}$, we need to choose $T$ such that $\frac{\log(T)}{T} \leq \min\set{\frac{\mu}{12L},  \frac{\mu^2}{6L^2}\sqrt{\frac{3}{8}}}$.
This condition holds if $T \geq 12\kappa^2\log(T)$.

Next, under Assumption~\ref{ass_general_bounded_variance} and the randomized reshuffling strategy of $\pi^{(t)}$, using \eqref{eq:ncvx_keybound1_reshuffling}, we can further estimate \eqref{eq_key_thm_01_est} as follows:
\begin{equation*}
\arraycolsep=0.1em
\begin{array}{lcl}
   \mathbb{E}\big[ F( w_0^{(t+1)} ) \big] & \leq &   \mathbb{E}\big[ F( w_0^{(t)} ) \big]  - \frac{\eta_t}{2}  \mathbb{E}\big[ \norms{ \nabla F( w_0^{(t)} )  }^2 \big] + \frac{L^2 \eta_t}{2n} \sum_{i=0}^{n-1}  \mathbb{E}\big[ \norms{ w_i^{(t)} - w_0^{(t)}}^2 \big] \vspace{1ex}\\
    & \overset{\tiny \eqref{eq:ncvx_keybound1_reshuffling}}{\leq} &  \mathbb{E}\big[  F( w_0^{(t)} ) \big]  - \frac{\eta_t}{2} \mathbb{E}\big[ \norms{ \nabla F( w_0^{(t)} ) }^2 \big] 
     +  \frac{L^2\eta_t^3}{n}  \left[ \left(\Theta + n \right) \mathbb{E} \big[ \Vert   \nabla F ( w_0^{(t)}) \Vert^2 \big] +  \sigma^2 \right] \vspace{1ex}\\
    & \overset{\tiny \eqref{eq:stronglyconvex}}{\leq} & \mathbb{E} \big[ F( w_0^{(t)} ) \big] - \mu \eta_t \mathbb{E} \big[ F( w_0^{(t)} ) - F(w_{*}) \big] +  \frac{2 L^2(\Theta + n) \eta_t^3}{n \mu} \mathbb{E} \big[ F( w_0^{(t)} ) - F(w_{*}) \big] \vspace{1ex}\\
    &&   + {~} \frac{L^2 \sigma_{*}^2 \eta_t^3}{n}.
\end{array}
\end{equation*}
Subtracting $F(w_{*})$ from both sides of  the last inequality, we can further derive
\begin{equation}\label{eq_rec_01_case2}
    \mathbb{E} \big[ F( w_0^{(t+1)} ) - F(w_{*}) \big] \leq \left[ 1 - \eta_t \left( \mu - \tfrac{2 L^2(\Theta+n)}{n\mu} \eta_t^2 \right) \right] \mathbb{E} \big[ F( w_0^{(t)} ) - F(w_{*}) \big] + \frac{L^2 \sigma^2 \eta_t^3}{n}.
\end{equation}
Now, assume that $0 < \eta_t \leq \frac{\mu}{2L\sqrt{\Theta/n + 1}} = \frac{1}{2\kappa\sqrt{1 + \Theta/n}}$. 
Then, one can show that $\mu - \frac{2 L^2(\Theta + n)}{n \mu} \eta_t^2 \geq \mu - \frac{1}{2} \mu = \frac{\mu}{2} > 0$.
Using this bound into \eqref{eq_rec_01_case2} and noticing that $\tilde{w}_t = w_0^{(t+1)}$, we can further upper bound it as
\begin{equation}\label{eq_rec_01b_case2}
   \mathbb{E} \big[ F( \tilde{w}_t ) - F(w_{*}) \big] \leq \left( 1 - \frac{\mu}{2}\eta_t \right) \mathbb{E} \big[F( \tilde{w}_{t-1} ) - F(w_{*})\big] + \frac{L^2\sigma^2}{n}\cdot \eta_t^3.
\end{equation}
Note that we have imposed $\eta_t \leq \min\set{\frac{1}{\sqrt{3}L},  \frac{1}{2\kappa\sqrt{1 + \Theta/n}} }$ due to $\eta_t \leq \frac{1}{\sqrt{3}L}$ in Lemma~\ref{le:ncvx_keybound1}.

Let $Y_t :=  \mathbb{E} \big[ F(\tilde{w}_{t-1} ) - F(w_{*}) \big]$, $\eta_t := \eta > 0$ be fixed for all $t \geq 1$, and $\rho := \frac{\mu}{2}$.
Then, from \eqref{eq_rec_01b_case2},  we have $Y_{t+1} \leq (1 - \rho\eta)Y_t + \frac{L^2\sigma^2\eta^3}{n}$.
Applying \eqref{eq_rate_t3} of Lemma~\ref{lem_general_framework} with $q=2$ and $\eta := \frac{4\log(\sqrt{n}T)}{\mu T}$, we obtain
\begin{equation*} 
\arraycolsep=0.2em
\begin{array}{lcl}
 Y_{T+1} & = &  \mathbb{E}\big[ F(\tilde{w}_{T}) - F(w_{*}) \big] \leq (1 - \rho \eta)^{T} Y_1 + \frac{2L^2\sigma^2\eta^{2}[ 1 - (1 - \rho \eta)^{T}]}{n\mu} \leq  Y_1\exp( -\mu \eta T) + \frac{2L^2\sigma^2\eta^2}{n\mu} \vspace{1ex}\\
 &= & \mathbb{E}\big[ F(\tilde{w}_{0}) - F(w_{*}) \big] \exp\big( -2\log(\sqrt{n} T) \big) + \frac{2L^2\sigma^2 \log(\sqrt{n}T)^2}{\mu^3 n T^2},
 \end{array}
\end{equation*}
which is exactly  \eqref{eq:scvx_bound2_const0}.
Note that to guarantee $\rho\eta = \frac{2\log(\sqrt{n}T)}{T} \leq 1$ and $\eta =  \frac{4\log(\sqrt{n}T)}{\mu T} \leq \frac{1}{2\kappa\sqrt{1 + \Theta/n}}$, we need to choose $T \geq \frac{8L\sqrt{\Theta/n+1}}{\mu^2}\log(\sqrt{n}T)$.

Finally, if $\pi^{(t)}$ is sampled at random without replacement and each $f(\cdot; i)$ is convex for $i \in [n]$, from \eqref{eq:scvx_key_bound2} and $F(\tilde{w}_{t-1}) - F(w_{*}) \geq \frac{\mu}{2}\norms{\tilde{w}_{t-1} - w_{*}}^2$, we have
\begin{equation*}
\mathbb{E}\big[ \norms{\tilde{w}_t - w_{*}}^2 \big]  \leq   (1 - \mu\eta_t) \mathbb{E} \big[ \norms{\tilde{w}_{t-1} - w_{*}}^2 \big]  +  \frac{2L\eta_t^3\sigma_{*}^2}{3n}.
\end{equation*}
Let us denote $Y_t := \mathbb{E}\big[ \norms{\tilde{w}_{t-1} - w_{*}}^2 \big]$. Then, the last inequality can be written as $Y_{t+1} \leq (1 - \mu\eta)Y_t + \frac{2L\eta^3\sigma_{*}^2}{3n}$.
Applying \eqref{eq_rate_t3} of Lemma~\ref{lem_general_framework} with $q=2$ and $\eta := \frac{2\log(\sqrt{n}T)}{\mu T}$, we obtain
\begin{equation*} 
\arraycolsep=0.2em
\begin{array}{lcl}
 Y_{T+1} & = &  \mathbb{E}\big[ \norms{\tilde{w}_{T} - w_{*}}^2 \big] \leq (1 - \mu \eta)^{T} Y_1 + \frac{2L\sigma_{*}^2\eta^{2}[ 1 - (1 - \rho \eta)^{T}]}{3n\mu} \leq  Y_1\exp( -\mu \eta T) + \frac{2L\sigma_{*}^2\eta^2}{3n\mu} \vspace{1ex}\\
 &= & \mathbb{E}\big[ \norms{\tilde{w}_{0} - w_{*}}^2 \big] \exp\big( -2\log(\sqrt{n} T) \big) + \frac{8L\sigma_{*}^2 \log(\sqrt{n}T)^2}{3 \mu^3 n T^2},
 \end{array}
\end{equation*}
which is exactly  \eqref{eq:scvx_bound2_const}.
Note that to guarantee $\mu \eta = \frac{2\log(\sqrt{n}T)}{T} \leq 1$ and $\eta =  \frac{2\log(\sqrt{n}T)}{\mu T} \leq \frac{\sqrt{5}-1}{2L}$, we need to choose $T$ such that $\frac{\log(T\sqrt{n})}{T} \leq \min\big\{ \frac{1}{2}, \frac{(\sqrt{5}-1)\mu}{2L}\big\}$.
\end{proof}
%%% End of the proof.

%%% The proof of Theorem 1.
\begin{proof}[\mytxtbi{The proof of Theorem \ref{thm_main_result_02}}]
Similar to the proof of Theorem~\ref{thm_main_result_02_const}, we define $Y_t := \mathbb{E}\big[ F(\tilde{w}_{t-1}) - F(w_{*}) \big] \geq 0$, $\rho := \frac{\mu}{3}$, and $D := (\mu^2  + L^2)\sigma_{*}^2$.
The estimate \eqref{eq_rec_01b} implies  $Y_{t+1} \leq (1 - \rho \cdot\eta) Y_t + D \eta^3$.
Moreover, since $\eta_t = \frac{6}{\mu(t + \beta)} = \frac{2}{\rho(t + \beta)}$ for $\beta \geq 1$, apply Lemma~\ref{lem_general_framework} with $q=2$, we obtain 
\begin{equation*}
    Y_{t+1} \leq \frac{\beta(\beta - 1)}{(t+\beta - 1)(t+\beta)}Y_1 + \frac{8 D\log(t+\beta+1)}{\rho^3(t+\beta-1)(t+\beta)}, 
\end{equation*}
which leads to \eqref{eq:scvx_bound1} after substituting $Y_{t+1} := \mathbb{E}\big[ F(\tilde{w}_{t}) - F(w_{*}) \big] $, $Y_1 :=  F(\tilde{w}_{0}) - F(w_{*})$, $D := (\mu^2  + L^2)\sigma_{*}^2$, and $\rho := \frac{\mu}{3}$ into the last estimate.
However, to guarantee $\eta_t = \frac{6}{\mu(t + \beta)} \leq \min\set{\frac{1}{2L}, \sqrt{\frac{3}{4}} \frac{\mu}{L^2}}$, we need to impose $\frac{6}{\beta + 1} \leq \min\set{\frac{\mu}{2L}, \sqrt{\frac{3}{4}} \frac{\mu^2}{L^2}}$, which holds if $\beta \geq 12\kappa^2 - 1$.

To prove  \eqref{eq:scvx_bound1}, we use \eqref{eq:scvx_key_bound2} from Lemma~\ref{le:scvx_key_bound2} and $F(\tilde{w}_{t-1}) - F(w_{*}) \geq \frac{\mu}{2}\norms{\tilde{w}_{t-1} - w_{*}}^2$ to get
\begin{equation*} 
\mathbb{E}\big[ \norms{\tilde{w}_t - w_{*}}^2 \big]  \leq  \big(1 - \mu\eta_t \big) \mathbb{E}\big[ \norms{\tilde{w}_{t-1} - w_{*}}^2 \big] +  \frac{2L\eta_t^3\sigma_{*}^2}{3n}.
\end{equation*}
By letting $\eta_t := \frac{2}{t + \beta}$, $Y_{t+1} := \mathbb{E}\big[ \norms{\tilde{w}_t - w_{*}}^2 \big]$, $\rho := \mu$, and $D :=  \frac{2L\sigma_{*}^2}{3n}$ for all $t \geq 1$.
The last estimate becomes 
\begin{equation*}
Y_{t+1} \leq (1 - \rho\eta_t) Y_t + D\eta_t^3, \quad \forall t \geq 1.
\end{equation*}
By applying Lemma~\ref{lem_general_framework} with $q=2$ we have $Y_{t+1} \leq \frac{\beta(\beta - 1)}{(t + \beta - 1)(t + \beta)}Y_1 + \frac{8D}{\rho^3(t + \beta - 1)(t + \beta)}$.
In order to guarantee that $\eta_t  = \frac{2}{t + \beta} \leq \frac{\sqrt{5} - 1}{2L}$ for all $t \geq 1$, we need to choose $\beta \geq  \frac{4L}{\sqrt{5}-1} - 1$.
However, since $L \leq \frac{\sqrt{5}-1}{2}$ and $\beta \geq 1$, this condition automatically holds.
\end{proof}
%%% End of the proof.

%%%%%%%%%%%%%%%%%%%%%%%%%%%%%%%%%%%%%%%%%%%%
%%% B. The nonconvex Case
%%%%%%%%%%%%%%%%%%%%%%%%%%%%%%%%%%%%%%%%%%%%
\section{Convergence Analysis for Nonconvex Case}\label{sec_appendix_nonconvex}
In this appendix, we provide the full proofs of the results in Section~\ref{sec_analysis_02}.

%% B.1. The proof of Theorem 2, Corollaries 3 and 4.
\subsection{Proofs of Theorem \ref{thm_nonconvex_01}, Corollary \ref{cor_nonconvex_01}, and Corollary~\ref{cor_nonconvex_01_01}}\label{apdx:proof_Th1_Corr_12}
\begin{proof}[\mytxtbi{The proof of Theorem~\ref{thm_nonconvex_01}}]
First, using \eqref{eq:ncvx_keybound1} into \eqref{eq_key_thm_01_est}, we can derive that
\begin{equation*}
\arraycolsep=0.2em
\begin{array}{lcl}
    F( w_0^{(t+1)} ) & \overset{\eqref{eq:ncvx_keybound1}}{\leq} &  F( w_0^{(t)} ) - \frac{\eta_t}{2} \norms{ \nabla F( w_0^{(t)} )  }^2 + \frac{L^2 \eta_t^3}{2} \left[ \left( 3 \Theta + 2 \right) \Vert   \nabla F ( w_0^{(t)}) \Vert^2 + 3 \sigma^2 \right]  \vspace{1.5ex}\\
    & = & F( w_0^{(t)} ) - \frac{\eta_t}{2}\big(1 -  L^2\eta_t^2(3\Theta + 2)\big) \norms{ \nabla F( w_0^{(t)} )  }^2 + \frac{3 L^2 \sigma^2  \eta_t^3}{2} \vspace{1ex}\\
    & \leq & F( w_0^{(t)} ) - \frac{\eta_t}{4} \norms{ \nabla F( w_0^{(t)} )  }^2 + \frac{3 L^2 \sigma^2  \eta_t^3}{2}.
\end{array}
\end{equation*}
where the last inequality follows since $\eta_t^2 \leq \frac{1}{2(3 \Theta + 2) L^2}$. Note that $\tilde{w}_t = w_0^{(t+1)}$ and $\tilde{w}_{t-1} = w_0^{(t)}$ in Algorithm~\ref{sgd_replacement}, the last estimate becomes
\begin{equation}\label{eq_001}
    F( \tilde{w}_t )  \leq F( \tilde{w}_{t-1} ) - \frac{\eta_t}{4} \norms{ \nabla F( \tilde{w}_{t-1} )  }^2 +  \frac{3 L^2 \sigma^2  \eta_t^3}{2}.
\end{equation}
Using $\eta_t := \eta$ into \eqref{eq_001} and  rearranging  its result, then taking expectation we end up with
\begin{equation*}
  \mathbb{E}\big[ \norms{\nabla F ( \tilde{w}_{t-1} ) }^2 \big]    \leq  \frac{4}{\eta} \mathbb{E} \big[F( \tilde{w}_{t-1} ) - F( \tilde{w}_t ) \big] \ + \  6 L^2 \sigma^2  \eta^2. 
\end{equation*}
Taking average the last inequality from $t :=1$ to $t := T$ and using the fact that $\mathbb{E}\big[ F(\tilde{w}_t) \big] \geq F_{*}$ from Assumption~\ref{ass_basic}(i), we finally obtain
\begin{equation*}
    \frac{1}{T} \sum_{t=1}^T \mathbb{E}\big[ \norms{ \nabla F ( \tilde{w}_{t-1} ) }^2 \big]  \leq \frac{4}{T \eta} \big[ F ( \tilde{w}_{0} ) - F_{*} \big] + 6 L^2 \sigma^2 \eta^2,
\end{equation*}
which is exactly \eqref{eq_thm_nonconvex_01}.
Note that, since $\tilde{w}_0$ is deterministic, we drop the expectation on the right-hand side of this estimate  \eqref{eq_thm_nonconvex_01}.

If $\pi^{(t)}$ is sampled uniformly at random without replacement from $[n]$, then taking expectation both sides of \eqref{eq_key_thm_01_est}, and then using \eqref{eq:ncvx_keybound1_reshuffling}, we obtain 
\begin{equation*} 
\hspace{-0ex}
\arraycolsep=0.1em
\begin{array}{lcl}
    \mathbb{E} \big[ F( w_0^{(t+1)} ) \big] & \leq &  \mathbb{E} \big[ F( w_0^{(t)} ) \big] - \frac{\eta_t}{2} \mathbb{E} \big[ \norms{ \nabla F( w_0^{(t)} ) }^2 \big] + \frac{L^2 \eta_t}{2n} \mathbb{E} \big[  \sum_{i=0}^{n-1}\norms{w_i^{(t)} - w_0^{(t)}}^2 \big] \vspace{1ex} \\
    & \overset{\eqref{eq:ncvx_keybound1_reshuffling}}{\leq} & \mathbb{E} \big[ F( w_0^{(t)} ) \big] - \frac{\eta_t}{2} \mathbb{E} \big[ \Vert \nabla{ F} ( w_0^{(t)} ) \Vert^2 \big] + \frac{L^2 \eta_t^3}{n} \left[ \left( \Theta + n \right)\mathbb{E} \big[ \Vert   \nabla F ( w_0^{(t)})  \Vert^2 \big] +  \sigma^2 \right] \vspace{1ex}\\
    & \leq & \mathbb{E} \big[ F( w_0^{(t)} ) \big] - \frac{\eta_t}{4} \mathbb{E} \big[ \norms{ \nabla F( w_0^{(t)} ) }^2 \big] + \frac{L^2 \sigma^2 \eta_t^3}{n}, 
\end{array}
\hspace{-0ex}
\end{equation*}
where the last inequality follows since $\eta_t^2 \leq \frac{n}{2(\Theta + n)L^2}$. Note that $\tilde{w}_t = w_0^{(t+1)}$ and $\tilde{w}_{t-1} = w_0^{(t)}$ in Algorithm~\ref{sgd_replacement}, the last estimate becomes
\begin{equation}\label{eq:eq_001_RR_est}
\mathbb{E} \big[ F( \tilde{w}_t ) - F_{*} \big] \leq \mathbb{E} \big[ F( \tilde{w}_{t-1} ) - F_{*} \big] - \frac{\eta_t}{4} \mathbb{E} \big[ \norms{ \nabla F( \tilde{w}_{t-1} )  }^2 \big] +  \frac{L^2 \sigma^2 \eta_t^3}{n}.
\end{equation}
Using this estimate and with a similar proof as of \eqref{eq_thm_nonconvex_01}, we obtain \eqref{eq_thm_nonconvex_01_RR}.
\end{proof}
%% End of the proof.

%%% The proof of Corollary 3.
\begin{proof}[\mytxtbi{The proof of Corollary \ref{cor_nonconvex_01}}]
Given $2\sigma^2 \geq \epsilon$ for a given accuracy $\epsilon > 0$, to guarantee $\frac{1}{T} \sum_{t=1}^T \norms{ \nabla F ( \tilde{w}_{t-1} ) }^2 \leq \epsilon$, by using \eqref{eq_thm_nonconvex_01} in Theorem~\ref{thm_nonconvex_01}, we impose
\begin{equation*}
    \frac{4}{T \eta} \big[ F ( \tilde{w}_{0} ) - F_{*} \big] +  6 L^2 \sigma^2 \eta^2 \leq \epsilon. 
\end{equation*}
Using $\eta = \frac{\sqrt{\epsilon}}{2L \sigma \sqrt{3\Theta + 2}} \leq \frac{1}{L\sqrt{2(3\Theta+2)}}$ into this inequation, we can easily get
\begin{equation*}
\frac{8 L \sigma \sqrt{3\Theta+2}}{T \sqrt{\epsilon}} \big[ F ( \tilde{w}_{0} ) - F_{*} \big] \leq \epsilon \left[ \frac{6\Theta + 1}{6 \Theta + 4)} \right] ~~~~\Rightarrow~~~~ T \geq \frac{16 L \sigma (3\Theta+2)^{3/2} \big[ F ( \tilde{w}_{0} ) - F_{*} \big]}{(6\Theta + 1)} \cdot \frac{1}{\epsilon^{3/2}}
\end{equation*}
Rounding this expression we get $T := \left\lfloor \frac{16 L \sigma (3\Theta+2)^{3/2} \left[ F ( \tilde{w}_{0} ) - F_{*} \right]}{(6\Theta + 1)} \cdot \frac{1}{\epsilon^{3/2}} \right\rfloor$.
As a result, the total number of gradient evaluations is $\mathcal{T}_{\nabla{f}} := n T = \left\lfloor \frac{16 L \sigma (3\Theta+2)^{3/2} \left[ F ( \tilde{w}_{0} ) - F_{*} \right]}{(6\Theta + 1)} \cdot \frac{n}{\epsilon^{3/2}} \right\rfloor$.

Alternatively, let us choose $\eta := \frac{\sqrt{n\epsilon}}{2L\sigma\sqrt{2(\Theta/n + 1)}}$, where $0 < \epsilon \leq \frac{4\sigma^2}{n}$.
Then, we have $0 < \eta \leq \frac{1}{L\sqrt{2(\Theta/n + 1)}}$.
Similar to the above proof, but using \eqref{eq_thm_nonconvex_01_RR}, we have 
\begin{equation*}
\frac{4}{T \eta} \big[ F ( \tilde{w}_{0} ) - F_{*} \big] +  \frac{4 L^2 \sigma^2 \eta^2}{n} = \frac{8 L\sigma\sqrt{2(\Theta/n + 1)}}{T\sqrt{n\epsilon}} \big[ F ( \tilde{w}_{0} ) - F_{*} \big] + \frac{\epsilon}{2(\Theta/n + 1)} \leq \epsilon.
\end{equation*}
This condition leads to $T \geq \frac{8 L\sigma(2\Theta/n + 2)^{3/2} \left[ F ( \tilde{w}_{0} ) - F_{*} \right] }{(2\Theta/n +1)} \cdot \frac{1}{\sqrt{n} \epsilon^{3/2}}$.
Hence, the total number of gradient evaluations is $\mathcal{T}_{\nabla{f}} = nT = \left\lfloor  \frac{8L\sigma(2\Theta/n + 2)^{3/2} \left[ F ( \tilde{w}_{0} ) - F_{*} \right]  }{(2\Theta/n + 1)\epsilon^{3/2}} \cdot \frac{\sqrt{n}}{\epsilon^{3/2}} \right\rfloor$.
\end{proof}
%%% End of proof.

%%% The proof of Corollary 4.
\begin{proof}[\mytxtbi{The proof of Corollary \ref{cor_nonconvex_01_01}}]
Substituting  $\eta = \frac{\gamma}{T^{1/3}} \leq \frac{1}{2L\sqrt{3\Theta+2}} \leq \frac{1}{L}$ into  \eqref{eq_thm_nonconvex_01} of Theorem~\ref{thm_nonconvex_01}, we obtain
\begin{equation*}
    \frac{1}{T} \sum_{t=1}^T \mathbb{E}\big[ \norms{ \nabla F ( \tilde{w}_{t-1} ) }^2 \big] \leq \frac{4}{T \eta} \big( F ( \tilde{w}_{0} ) - F_{*} \big) + 6 L^2 \sigma^2 \eta^2 = \frac{1}{T^{2/3}} \left[  \frac{4 (F ( \tilde{w}_{0} ) - F_{*} )}{\gamma} + 6 L^2 \sigma^2 \gamma^2 \right], 
\end{equation*}
which is exactly our desired estimate \eqref{eq_nonconvex_01}.

If, in addition, $\pi^{(t)}$ is sampled uniformly at random from $[n]$, then by choosing $\eta := \frac{\gamma n^{1/3}}{T^{1/3}}$ such that $\eta_t \leq \frac{1}{L\sqrt{2(\Theta/n + 1)}} \leq \frac{1}{L}$, we obtain from \eqref{eq_thm_nonconvex_01_RR} that
\begin{equation*} 
\begin{array}{l}
    \frac{1}{T} \sum_{t=1}^T \mathbb{E}\big[ \norms{\nabla F ( \tilde{w}_{t-1} )}^2 \big] \leq \frac{4}{T \eta} \big[ F ( \tilde{w}_{0} ) - F_{*} \big] +  \frac{4 L^2 \sigma^2 \eta^2}{n} = \frac{1}{n^{1/3}T^{2/3}} \left[  \frac{4 [F ( \tilde{w}_{0} ) - F_{*} ]}{\gamma} + 4 L^2 \sigma^2 \gamma^2 \right].
\end{array}
\end{equation*}
This proves \eqref{eq_nonconvex_01b}.
Here, we need to choose $T \geq 1$ such that $\frac{\gamma n^{1/3}}{T^{1/3}} \leq \frac{1}{L\sqrt{2(\Theta/n + 1)}}$.
\end{proof}
%% End of the proof.

%%%%% B.2. Convergence Analysis with Diminishing Stepsize.
\subsection{Proof of Theorem \ref{thm_nonconvex_02}: Asymptotic convergence with diminishing stepsize}
\label{apdx:proof_Th2_and_3}
To establish Theorem \ref{thm_nonconvex_02}, we will use the following lemma from \citep{BertsekasSurvey}. 

%%% Lemma 5.
\begin{lem}[\citep{BertsekasSurvey}]\label{prop_supermartingale}
Let $\sets{Y_t}_{t\geq 0}$, $\sets{Z_t}_{t\geq 0}$, and $\set{W_t}_{t\geq 0}$ be three sequences of random variables.
Let $\sets{\mathcal{F}_t}_{t\geq 0}$ be a filtration, i.e. a $\sigma$-algebras such that $\mathcal{F}_t \subseteq \mathcal{F}_{t+1}$ for all $t \geq 0$. 
Suppose that the following conditions hold:
\begin{compactitem}
\item[$\mathrm{(a)}$] $Y_t$, $Z_t$, and $W_t$ are nonnegative and $\mathcal{F}_t$-measurable for all $t \geq 0$;
\item[$\mathrm{(b)}$] for each $t \geq 0$, we have $\Exp{Y_{t+1} \mid \mathcal{F}_t} \leq Y_t - Z_t + W_t$;
\item[$\mathrm{(c)}$] with probability $1$ $($w.p.1$)$, it holds that  $\sum_{t=0}^{\infty} W_t < +\infty$.
\end{compactitem}
Then, w.p.1, we have
\begin{gather*}
\sum_{t=0}^{\infty} Z_t < +\infty \quad \text{and} \quad Y_t \to Y \geq 0 \ \text{as} \ t \to +\infty. 
\end{gather*}
\end{lem}
%%% End of Lemma 5.

Using Lemma~\ref{prop_supermartingale} we can now prove Theorem~\ref{thm_nonconvex_02} in the main text as follows.

%%% Proof of Theorem 3.
\begin{proof}[\mytxtbi{The proof of Theorem \ref{thm_nonconvex_02}}]
First, following the same argument as in the proof of \eqref{eq_001} of  Theorem~\ref{thm_nonconvex_01}, we have
\begin{equation*}
    F( \tilde{w}_{t+1} )  \leq F( \tilde{w}_{t} ) - \frac{\eta_{t+1}}{4} \norms{ \nabla F( \tilde{w}_{t} ) }^2 + \frac{3 L^2 \sigma^2 \eta_{t+1}^3}{2}. 
\end{equation*}
Let us define $\mathcal{F}_t = \sigma(\tilde{w}_{0},\cdots,\tilde{w}_{t})$ the $\sigma$-algebra generated by $\sets{ \tilde{w}_{0},\cdots,\tilde{w}_{t} }$. 
Then, for $t \geq 0$, the last inequality implies
\begin{equation*}
   \mathbb{E}\big[ F( \tilde{w}_{t+1} ) - F_{*} \mid \mathcal{F}_{t} \big]  \leq \left[ F( \tilde{w}_{t} ) - F_{*} \right] - \frac{\eta_{t+1}}{4} \norms{ \nabla F( \tilde{w}_{t} )  }^2 + \tfrac{3 L^2 \sigma^2 \eta_{t+1}^3}{2}. 
\end{equation*}
Let us define $Y_t := [F( \tilde{w}_{t} ) - F_{*}] \geq 0$, $Z_t := \frac{\eta_{t+1}}{4} \norms{ \nabla F( \tilde{w}_{t} )  }^2 \geq 0$ and $W_t := \frac{3}{2} L^2 \sigma^2 \eta_{t+1}^3$.
Then, the first condition (a) of Lemma~\ref{prop_supermartingale} holds.
Moreover, the last inequality shows that $\Exp{Y_{t+1} \mid \mathcal{F}_t} \leq Y_t - Z_t + W_t$, which means that the condition (b) of Lemma~\ref{prop_supermartingale} holds.
Since $\sum_{t=1}^\infty \eta_t^3  < +\infty$, we have $\sum_{t=0}^{\infty}W_t < +\infty$, which fulfills the condition (c) of Lemma~\ref{prop_supermartingale}.
Then, by applying Lemma \ref{prop_supermartingale}, we obtain  w.p.1 that
\begin{equation*}
F( \tilde{w}_t ) - F_{*} \to Y \geq 0 \ \text{as} \ t \to +\infty, \quad \text{and} \quad \sum_{t=0}^{\infty} \frac{\eta_{t+1}}{4} \norms{ \nabla F( \tilde{w}_t ) }^2 < +\infty. 
\end{equation*}
We prove $\liminf\limits_{t\to\infty} \norms{ \nabla F ( \tilde{w}_{t-1} ) } = 0$ w.p.1. by contradiction.
Indeed, we assume that there exist $\epsilon > 0$ and $t_0 \geq 0$ such that $\norms{ \nabla F ( \tilde{w}_t ) }^2 \geq \epsilon$ for all $t \geq t_0$. 
In this case, since $\sum_{t = 0}^{\infty}\eta_t = \infty$, we have
\begin{equation*}
  \infty >   \sum_{t=t_0}^{\infty} \frac{\eta_{t+1}}{4} \norms{ \nabla F( \tilde{w}_t ) }^2  \geq \frac{\epsilon}{4} \sum_{t = t_0}^{\infty} \eta_{t+1}  = \infty. 
\end{equation*}
This is a contradiction. 
As a result, w.p.1., we have $\liminf\limits_{k\to\infty} \norms{ \nabla F ( \tilde{w}_{k} ) }^2 = 0$, or equivalently, it holds that $\liminf\limits_{k\to\infty} \norms{ \nabla F ( \tilde{w}_{k} ) } = 0$.
The proof still holds if we use \eqref{eq:eq_001_RR_est} of Theorem~\ref{thm_nonconvex_01}.
\end{proof}
%%% End of proof.

%%% B.3. Convergence analysis with different learning rates
\subsection{Convergence analysis for different learning rates}\label{subsec:new_results}
This appendix provides convergence analysis for general choices of learning rate.

\begin{proof}[\mytxtbi{The proof of Theorem~\ref{thm_nonconvex_04}}]
By \eqref{eq_001}, we have
\begin{align*}
    F( \tilde{w}_t ) & \leq F( \tilde{w}_{t-1} ) - \frac{\eta_t}{4} \| \nabla F( \tilde{w}_{t-1} )   \|^2 + \frac{3 L^2 \sigma^2 \eta_t^3}{2} \leq F( \tilde{w}_{t-1} ) + \frac{3 L^2 \sigma^2 \eta_t^3}{2}. 
\end{align*}
Alternatively, if a randomized reshuffling scheme is used, then by using \eqref{eq:eq_001_RR_est}, we have
\begin{equation*} 
\mathbb{E} \big[ F( \tilde{w}_t ) \big] \leq \mathbb{E} \big[ F( \tilde{w}_{t-1} ) \big] - \frac{\eta_t}{4} \mathbb{E} \big[ \norms{ \nabla F( \tilde{w}_{t-1} )  }^2 \big] +  \frac{L^2 \sigma^2 \eta_t^3}{n} \leq  \mathbb{E} \big[ F( \tilde{w}_{t-1} ) \big] + \frac{L^2 \sigma^2 \eta_t^3}{n}.
\end{equation*}
Combining both cases, we can write them in a single inequality as
\begin{equation}\label{eq:th5_est100} 
\mathbb{E} \big[ F( \tilde{w}_t ) \big] \leq \mathbb{E} \big[ F( \tilde{w}_{t-1} ) \big] - \frac{\eta_t}{4} \mathbb{E} \big[ \norms{ \nabla F( \tilde{w}_{t-1} )  }^2 \big] +  D \cdot \eta_t^3 \leq  \mathbb{E} \big[ F( \tilde{w}_{t-1} ) \big] + D \cdot \eta_t^3,
\end{equation}
where $D := \frac{3}{2}L^2\sigma^2$ for the general shuffling strategy and $D := \frac{L^2\sigma^2}{n}$ for the randomized reshuffling strategy.

Since $\eta_t = \frac{\gamma}{(t+\beta)^{\alpha}} $, summing up the last inequality \eqref{eq:th5_est100} from $t = 1$ to $t = k \geq 1$, we have
\begin{equation*}
\arraycolsep=0.2em
\begin{array}{lclcl}
    F( \tilde{w}_k ) & \leq & F( \tilde{w}_{0} ) + D \cdot \sum_{t=1}^k \eta_t^3   & = &   F( \tilde{w}_{0} ) + D \cdot \sum_{t=1}^k \frac{\gamma^3}{(t+\beta)^{3\alpha}}  \vspace{1ex}\\
    & \overset{\eqref{integral_03}}{\leq} &  F( \tilde{w}_{0} ) + D \cdot \gamma^3 \int_{t=0}^{k} \frac{dt}{(t+\beta)^{3\alpha}}   & = &   F( \tilde{w}_{0} ) + D \cdot \gamma^3 \left[ - \frac{(t+\beta)^{-(3\alpha - 1)}}{3 \alpha - 1} \Big |_{t = 0}^{k} \right]  \vspace{1ex}\\
    & \leq &  F( \tilde{w}_{0} ) + \frac{D \gamma^3}{ (3\alpha - 1)\beta^{3\alpha - 1}}. 
\end{array}
\end{equation*}
Here, we use the fact that $\frac{1}{(t+\beta)^{3\alpha}}$ is nonnegative and monotonically decreasing on $[0, +\infty)$ and  $\frac{1}{3} < \alpha < 1$. 
Subtracting $F_{*}$ from both sides of the last estimate, for $t \geq 1$, we have
\begin{equation}\label{bounded_function_01}
     \mathbb{E}\big[ F( \tilde{w}_t ) -F_{*}\big]   \leq  \mathbb{E}\big[ F( \tilde{w}_{0} ) -F_{*} \big] + \frac{D \gamma^3}{ (3\alpha - 1)\beta^{3\alpha - 1}}.
\end{equation}
On the other hand, subtracting $F_{*}$ from both sides of \eqref{eq:th5_est100}, we have
\begin{equation}\label{eq_rec_01f_01}
     \mathbb{E}\big[ F( \tilde{w}_t ) - F_{*} \big]  \leq  \mathbb{E}\big[  F( \tilde{w}_{t-1} ) - F_{*}  \big] - \frac{\eta_t}{4}  \mathbb{E}\big[ \norms{ \nabla F( \tilde{w}_{t-1} )  }^2 \big] +  D \cdot \eta_t^3. 
\end{equation}
Now, let us define $Y_t :=  \mathbb{E}\big[ F(\tilde{w}_{t-1}) - F_{*} \big] \geq 0$, $Z_t :=  \mathbb{E}\big[ \norms{ \nabla F( \tilde{w}_{t-1} )  }^2 \big] \geq 0$, for $t \geq 1$, and $\rho := \frac{1}{4}$.
The estimate \eqref{eq_rec_01f_01} becomes
\begin{equation*}
Y_{t+1} \leq Y_t - \rho \eta_t Z_t + D \eta_t^3. 
\end{equation*} 
Let us define $C := [F( \tilde{w}_{0} ) -F_{*}] + \frac{D \gamma^3}{(3\alpha - 1)\beta^{3\alpha - 1}} > 0$. 
By \eqref{bounded_function_01}, we have $Y_t \leq C$ (note that $H = 0$ in Lemma~\ref{lem_general_framework_02}), $t \geq 1$. 
Applying Lemma~\ref{lem_general_framework_02} with $q = 3$ and $m = 1$, we conclude that
\begin{compactitem}
    \item If $\alpha = \frac{1}{2}$, we have
    \begin{equation*}
    \begin{array}{lcl}
        \frac{1}{T} \sum_{t = 1}^{T}  \mathbb{E}\big[ \norms{ \nabla F( \tilde{w}_{t-1} )  }^2 \big]
        & \leq & \frac{4 (1+\beta)^{1/2} [F( \tilde{w}_{0} ) -F_{*}]}{ \gamma} \cdot \frac{1}{T} \ + \ \frac{4 C}{ \gamma} \cdot \frac{(T - 1 + \beta)^{1/2}}{T} \vspace{1ex}\\
         &&  + {~} 4 D \gamma^{2} \cdot \frac{\log(T+\beta) - \log(\beta)}{T}. 
    \end{array}
    \end{equation*}
    \item If $\alpha \neq \frac{1}{2}$, we have
    \begin{equation*}
    \begin{array}{lcl}
        \frac{1}{T} \sum_{t = 1}^{T}  \mathbb{E}\big[ \norms{ \nabla F( \tilde{w}_{t-1} )  }^2 \big]  & \leq & \frac{4 (1+\beta)^{\alpha} [F( \tilde{w}_{0} ) -F_{*}]}{ \gamma} \cdot \frac{1}{T} \ + \ \frac{2 C}{\alpha \gamma} \cdot \frac{(T - 1 + \beta)^{\alpha}}{T}  \vspace{1ex}\\
        && + {~} \frac{4 D \gamma^{2}}{ (1 - 2\alpha)} \cdot \frac{(T + \beta)^{1 - 2\alpha}}{T}. 
    \end{array}
    \end{equation*}
\end{compactitem}
For the general shuffling strategy, since $D := \frac{3}{2}L^2\sigma^2$, we need to impose $\eta_t = \frac{\gamma}{(t + \beta)^{\alpha}} \leq \frac{1}{L\sqrt{2(3\Theta + 2)}}$ for all $t \geq 1$.
This condition holds if $\gamma L\sqrt{2(3\Theta + 2)} \leq (\beta + 1)^{\alpha}$.
For the randomized reshuffling strategy, since $D := \frac{L^2\sigma^2}{n}$, we need to impose $\eta_t = \frac{\gamma}{(t + \beta)^{\alpha}} \leq \frac{1}{L\sqrt{2(\Theta/n + 1)}}$ for all $t \geq 1$.
This condition holds if $\gamma L\sqrt{2(\Theta/n + 1)} \leq (\beta + 1)^{\alpha}$.

Finally, if we replace $\gamma$ by $n^{1/3}\gamma$, then the last condition becomes $\gamma n^{1/3} L\sqrt{2(\Theta/n + 1)} \leq (\beta + 1)^{\alpha}$.
Substituting this new quantity $n^{1/3}\gamma$ into $\gamma$ of the right-hand side of our bounds, we obtain the remaining conclusions of this theorem.
\end{proof}
%%% End of the proof.

% % =======================================
\begin{proof}[\mytxtbi{The proof of Theorem~\ref{thm_nonconvex_04_02}}]
By \eqref{eq:th5_est100}, we have
\begin{equation*} 
\mathbb{E} \big[ F( \tilde{w}_t ) \big] \leq \mathbb{E} \big[ F( \tilde{w}_{t-1} ) \big] - \frac{\eta_t}{4} \mathbb{E} \big[ \norms{ \nabla F( \tilde{w}_{t-1} )  }^2 \big] +  D \cdot \eta_t^3 \leq  \mathbb{E} \big[ F( \tilde{w}_{t-1} ) \big] + D \cdot \eta_t^3,
\end{equation*}
where $D := \frac{3}{2}L^2\sigma^2$ for the general shuffling strategy, and $D := \frac{L^2\sigma^2}{n}$ for the randomized reshuffling strategy.
Since $\eta_t = \frac{\gamma}{(t+\beta)^{1/3}} $, summing up this inequality from $t = 1$ to $t = k \geq 1$, we obtain
\begin{equation*}
\arraycolsep=0.2em
\begin{array}{lclcl}
    \mathbb{E} \big[ F( \tilde{w}_k ) \big] & \leq & F( \tilde{w}_{0} ) + D \cdot \sum_{t=1}^k \eta_t^3 & = & F( \tilde{w}_{0} ) + D \cdot \sum_{t=1}^k \frac{\gamma^3}{(t+\beta)} \vspace{1ex}\\
    & \overset{\eqref{integral_03}}{\leq} & F( \tilde{w}_{0} ) + \frac{D \gamma^3}{(1 + \beta)} + D \gamma^3 \int_{t=1}^{k} \frac{dt}{(t+\beta)} 
 & \leq & F( \tilde{w}_{0} ) + \frac{D \gamma^3}{(1 + \beta)} + D\gamma^3 \log(k + \beta). 
 \end{array}
\end{equation*}
Here, we use the fact that $\frac{1}{t+\beta}$ is nonnegative and monotonically decreasing on $[0, +\infty)$.
Subtracting $F_{*}$ from both sides of the last estimate, for $t \geq 1$, we have
\begin{equation}\label{bounded_function_02}
\mathbb{E} \big[ F( \tilde{w}_t ) -F_{*}  \big]  \leq \left[ F( \tilde{w}_{0} ) -F_{*} \right] + \frac{D \gamma^3}{(1 + \beta)} + D \gamma^3 \log(t  + \beta). 
\end{equation}
Define $Y_t := \mathbb{E} \big[ F(\tilde{w}_{t-1}) - F_{*} \big] \geq 0$ and $Z_t := \mathbb{E} \big[ \norms{ \nabla F( \tilde{w}_{t-1} )  }^2 \big] \geq 0$ for $t \geq 1$, and $\rho := \frac{1}{4}$.
Then, the estimate \eqref{bounded_function_02} becomes
\begin{equation*}
Y_{t+1} \leq Y_t - \rho \eta_t Z_t + D \eta_t^3. 
\end{equation*} 
Let us define $C := [F( \tilde{w}_{0} ) -F_{*}] + \frac{D \gamma^3}{(1 + \beta)} > 0$, $H: = D \gamma^3 > 0$, and $\theta:= 1 + \beta > 0$.
Clearly, we have $1 + \theta - \beta = 2 > \frac{2}{3}e^{1/2}$. 
By \eqref{bounded_function_02}, we have $Y_t \leq C + H \log(t + \theta)$ for $t \geq 1$. Applying Lemma~\ref{lem_general_framework_02} with $q = 3$, $m = 1$, and $\alpha = \frac{1}{3}$, we conclude that
\begin{equation*}
\begin{array}{lcl}
    \frac{1}{T} \sum_{t = 1}^{T} \mathbb{E}\big[ \norms{ \nabla F( \tilde{w}_{t-1} )  }^2 \big] & \leq & \frac{4 (1+\beta)^{1/3} [F( \tilde{w}_{0} ) -F_{*}]}{ \gamma} \cdot \frac{1}{T} \ + \ \frac{6 C}{ \gamma} \cdot \frac{(T - 1 + \beta)^{1/3}}{T} \vspace{1ex}\\
     && + {~} 6D \gamma^2 \cdot \frac{(T - 1 + \beta)^{1/3} \log(T + \beta)}{T}  \ + \ 12D \gamma^{2} \cdot \frac{(T + \beta)^{1/3}}{T}, 
\end{array}
\end{equation*}
which is our main bound.
The remaining conclusion is proved similarly as in Theorem~\ref{thm_nonconvex_04}.
\end{proof}
%% End of the proof.

%%% Proofs of Theorem 4 and Corollary 5: Convergence Analysis with Diminishing Stepsize.
\subsection{Proofs of Theorem~\ref{thm_nonconvex_03}: The gradient dominance case}\label{apdx:proof_of_Th5_Cor34}

% % ======================================
\begin{proof}[\mytxtbi{The proof Theorem~\ref{thm_nonconvex_03}}]
Using \eqref{eq:th5_est100} from the proof of  Theorem~\ref{thm_nonconvex_04} and Assumption~\ref{ass_PLcond}, we can derive that
\begin{equation*} 
\arraycolsep=0.2em
\begin{array}{lcl}
\mathbb{E} \big[ F( \tilde{w}_t ) - F_{*} \big] & \leq & \mathbb{E} \big[ F( \tilde{w}_{t-1} ) - F_{*} \big] - \frac{\eta_t}{4} \mathbb{E} \big[ \norms{ \nabla F( \tilde{w}_{t-1} )  }^2 \big] +  D \cdot \eta_t^3 \vspace{1ex}\\
&\overset{\eqref{eq_grad_dominant}}{\leq} & \left( 1 - \frac{\eta_t}{4 \tau} \right) \mathbb{E}\big[ F( \tilde{w}_{t-1} ) - F_{*} \big] + D \cdot \eta_t^3. 
\end{array}
\end{equation*}
where $D := \frac{3}{2}L^2\sigma^2$ for the general shuffling strategy, and $D := \frac{L^2\sigma^2}{n}$ for the randomized reshuffling strategy.
Let $Y_t := \mathbb{E} \big[ F(\tilde{w}_{t-1}) - F_{*} \big] \geq 0$, and $\rho := \frac{1}{4 \tau}$.
We now verify that these quantities satisfy the conditions of Lemma~\ref{lem_general_framework} with $q = 2$, i.e. $Y_{t+1} \leq (1 - \rho \eta_t)Y_{t-1} + D \cdot \eta_t^3$.
Applying Lemma~\ref{lem_general_framework} with $\eta_t := \frac{2}{t + \beta}$ for some $\beta \geq 1$, then we obtain 
\begin{equation*}
\mathbb{E} \big[ F(\tilde{w}_{t}) - F_{*} \big] = Y_{t+1} \leq \frac{1}{(t+\beta-1)(t + \beta)}\Big[ \beta(\beta - 1)[ F(\tilde{w}_0) - F_{*} ] + \frac{8D\log(t + \beta)}{\rho^3} \Big].
\end{equation*}
For the general shuffling scheme, we need to choose $\beta \geq 1$ such that $\eta_t = \frac{2}{t + \beta} \leq \frac{1}{L\sqrt{2(3\Theta + 2)}}$ for all $t \geq 1$.
Hence, we can choose $\beta \geq 2L\sqrt{2(3\Theta + 2)} - 1$, and in this case, we obtain \eqref{eq:nonconvex_03_general}.

For the randomized reshuffling strategy, if we choose $\eta_t := \frac{2}{t + 1 + 1/n}$, then to guarantee $\eta_t \leq \frac{1}{L\sqrt{2(\Theta/n + 1)}}$ for all $t \geq 1$, we require $L\sqrt{2(\Theta/n + 1)} \leq 1$.
In this case, by substituting $\beta := 1 + 1/n$ and $D := \frac{L^2\sigma^2}{n}$ into the last estimate, we obtain \eqref{eq:nonconvex_03_RR}.
\end{proof}
%%% End of the proof.

%%%%%%%%%%%%%%%%%%%%%%%%%%%%%%%%%%%%%
%%%%% D. New Convergence Analysis for Classical SGD.
%%%%%%%%%%%%%%%%%%%%%%%%%%%%%%%%%%%%%
\section{New Convergence Analysis for Standard SGD}\label{sec_standard_sgd}
As a side result of our stand-alone technical lemma, Lemma~\ref{lem_general_framework_02}, we show in this appendix that we can apply the analysis framework of Lemma~\ref{lem_general_framework_02} to obtain convergence rate results for the standard stochastic gradient algorithm, abbreviated by SGD.

To keep it more general, we consider the stochastic optimization problem  with respect to some distribution $\mathcal{D}$ as in \eqref{Expectation_problem_01}, i.e.:
\begin{equation}\label{expected_risk}
    \min_{w \in \mathbb{R}^d} \Big\{ F(w) := \mathbb{E}_{\xi \sim \mathcal{D}} \big[ f(w; \xi) \big] \Big\},
\end{equation}
where $\nabla{f}$ is an unbiased gradient estimator of the gradient $\nabla{F}$ of $F$, i.e.:
\begin{equation*}
    \mathbb{E}_{\xi \sim \mathcal{D}} \big[ \nabla f(w ; \xi) \big] = \nabla F(w), \quad  \forall w \in \dom{F}.
\end{equation*}
The standard SGD method without mini-batch for solving \eqref{expected_risk} can be described as in Algorithm~\ref{sgd_algorithm}. 

%%% SGD Algorithm.
\begin{algorithm}[hpt!]
\caption{Stochastic Gradient Descent (SGD) Method (without mini-batch)}\label{sgd_algorithm}
\begin{algorithmic}[1]
   \STATE {\bfseries Initialize:} Choose an initial point $w_1\in\dom{F}$.
   \FOR{$t=1,2,\cdots$}
  \STATE Generate a realization of $\xi_t$ and evaluate a stochastic gradient $\nabla f(w_{t}; \xi_t)$;
  \STATE Choose a step size (i.e. a learning rate) $\eta_t > 0$ (specified later); 
   \STATE Update  $w_{t+1} := w_{t} - \eta_t \nabla f(w_{t}; \xi_t)$;
   \ENDFOR
\end{algorithmic}
\end{algorithm}

To analyze convergence rate of Algorithm~\ref{sgd_algorithm}, we assume that problem \eqref{expected_risk} satisfies Assumptions~\ref{ass_basic}(i) and Assumptions \ref{ass_weaker_smooth_F} and \ref{ass_bounded_variance} below. 

%%% Assumption A.6.
\begin{ass}[One-side $L$-smoothness]\label{ass_weaker_smooth_F}
The objective function  $F$ of \eqref{expected_risk}  satisfies
\begin{equation}\label{eq:Lsmooth_2}
F(w) \leq F(w') + \iprods{\nabla F(w'), w-w'} + \frac{L}{2}\norms{w-w'}^2, \quad \forall w, w' \in \dom{F}. 
\end{equation}
\end{ass}

%%% Assumption A.7.
\begin{ass}[Bounded variance]\label{ass_bounded_variance}
For \eqref{expected_risk}, there exists $\sigma \in (0, + \infty)$ such that
\begin{equation}\label{eq_bounded_variance}
    \mathbb{E} \big[ \norms{ \nabla f(w; \xi ) - \nabla F(w) }^2 \big]  \leq \sigma^2, \quad \forall w \in \dom{F}.
\end{equation}
\end{ass}

We notice that we are able to derive the analysis based on the general bounded variance assumption $\mathbb{E}\big[ \norms{ \nabla f(w; \xi ) - \nabla{F}(w) }^2 \big] \leq \Theta \norms{\nabla{F}(w)}^2 + \sigma^2$ for some $\Theta \geq 0$ and $\sigma > 0$ as in Assumption~\ref{ass_general_bounded_variance}. For simplicity, we only consider the special case where $\Theta = 0$ since Assumption~\ref{ass_bounded_variance} is commonly used in the literature for the standard SGD method. We prove our first result for Algorithm~\ref{sgd_algorithm} to solve \eqref{expected_risk} in the following theorem.

%%% Theorem 7.
\begin{thm}\label{thm_sgd_nonconvex_01}
Assume that Assumptions \ref{ass_basic}$\mathrm{(i)}$, \ref{ass_weaker_smooth_F}, and \ref{ass_bounded_variance} hold for \eqref{expected_risk}. 
Let $\sets{ w_t }$ be generated by Algorithm~\ref{sgd_algorithm} with $0 < \eta_t:= \frac{\gamma}{(t+\beta)^{\alpha}} \leq \frac{1}{L}$ for some $\gamma > 0$, $\beta > 0$, and $\frac{1}{2} < \alpha < 1$, and  $C := [F(w_{1}) - F_*] + \frac{L \sigma^2 \gamma^2}{2(2\alpha - 1)\beta^{2\alpha - 1}} > 0$. 
Then, the following bound holds:
\begin{equation}\label{eq:th7_bounds}
\arraycolsep=0.2em
\begin{array}{lcl}
\displaystyle \frac{1}{T} \sum_{t = 1}^{T} \mathbb{E} \big[ \norm{\nabla F(w_t)}^2 \big]     & \leq & \dfrac{2 (1+\beta)^{\alpha} \left[ F(w_{1}) - F_{*} \right]}{ \gamma} \cdot \dfrac{1}{T} \ + \  \dfrac{C}{ \alpha \gamma} \cdot \dfrac{(T - 1 + \beta)^{\alpha}}{T}  \vspace{1ex}\\
&&  + {~}  \dfrac{L \sigma^2 \gamma}{ (1 - \alpha)} \cdot \dfrac{(T + \beta)^{1 - \alpha}}{T}. 
\end{array}
\end{equation}
\end{thm}

%%% The proof of Theorem 7.
\begin{proof}
Let $\mathcal{F}_t = \sigma(w_1,\cdots,w_t)$ be the $\sigma$-algebra generated by $\set{ w_1, \cdots, w_t}$. 
Then, from the one-side $L$-smoothness of $F$ in \eqref{eq:Lsmooth_2}, we have
\begin{equation*}
\arraycolsep=0.2em
\begin{array}{lcl}
\mathbb{E}[F(w_{t+1}) \mid \mathcal{F}_t]  & \leq &  F(w_t) - \eta_t \norms{ \nabla F(w_t) }^2 +  \frac{\eta_t^2 L}{2}  \mathbb{E}\big[ \norms{ \nabla f(w_t; \xi_t) }^2 \mid  \mathcal{F}_t \big]  \vspace{1ex}\\
 & = & F(w_t) - \eta_t \left(1 - \frac{\eta_t L}{2}  \right) \norms{ \nabla F(w_t) }^2 + \frac{\eta_t^2 L}{2}  \mathbb{E}\big[ \norms{ \nabla f(w_t; \xi_t) - \nabla F(w_t) }^2 \mid \mathcal{F}_t \big]   \vspace{1ex}\\
& \overset{\tiny\eqref{eq_bounded_variance}}{\leq} & F(w_t) - \frac{\eta_t}{2} \norms{ \nabla F(w_t) }^2 + \frac{\eta_t^2 L}{2}  \sigma^2,  
\end{array}
\end{equation*}
where the first equality follows since $\mathbb{E}\big[ \norms{ \nabla f(w_t; \xi_t) - \nabla F(w_t) }^2 \mid \mathcal{F}_t \big] = \mathbb{E}\big[ \norms{ \nabla f(w_t; \xi_t) }^2 \mid \mathcal{F}_t \big] - \norms{ \nabla F(w_t) }^2$; 
and the last inequality follows since $F$ has bounded variance in Assumption~\ref{ass_bounded_variance}. 
Note that $\eta_t \left(1 - \frac{\eta_t L}{2}  \right) \geq \frac{\eta_t}{2}$ since $0 < \eta_t \leq \frac{1}{L}$. 
Subtracting $F_{*}$ from, and then taking full expectation of both sides of the last estimate, we obtain
\begin{equation} \label{eq_recursive_sgd_01}
\arraycolsep = 0.2em
\begin{array}{lcl}
    \Exp{F(w_{t+1}) - F_*} &\leq & \Exp{F(w_{t}) - F_*} - \frac{\eta_t}{2} \Exp{\norm{\nabla F(w_t)}^2} + \frac{\eta_t^2 L \sigma^2}{2}  \vspace{1ex}\\
    & \leq & \Exp{F(w_{t}) - F_*} + \frac{\eta_t^2 L \sigma^2}{2}. 
\end{array}
\end{equation}
Summing up \eqref{eq_recursive_sgd_01}  from $t := 1$ to $t := k \geq 1$, we get
\begin{equation}\label{eq_main_bounded_function_sgd}
\hspace{-1ex}
\arraycolsep=0.2em
\begin{array}{lcl}
    \Exp{F(w_{k+1}) - F_{*} } &\leq & \Exp{F(w_{1}) - F_{*} } + \frac{ L \sigma^2}{2} \sum_{t=1}^k \eta_t^2 = \big[ F(w_{1}) - F_{*} \big] + \frac{L\sigma^2}{2} \sum_{t=1}^k \frac{\gamma^2}{(t + \beta)^{2\alpha}} \vspace{1ex}\\
    & \overset{\tiny \eqref{integral_03}}{\leq} & \big[ F(w_{1}) - F_{*} \big] + \frac{L\sigma^2 \gamma^2}{2} \int_{t=0}^{k} \frac{dt}{(t + \beta)^{2\alpha}}. 
\end{array}
\hspace{-3ex}
\end{equation}
If $\frac{1}{2} < \alpha < 1$, then for $k \geq 1$, we have
\begin{equation}\label{eq_bounded_function_sgd_02}
    \Exp{F(w_{k+1}) - F_{*} } \leq \big[ F(w_{1}) - F_{*} \big] + \frac{L \sigma^2 \gamma^2}{2(2\alpha - 1) \beta^{2\alpha - 1}}.
\end{equation}
Now, let us define $Y_t := \Exp{F(w_{t}) - F_{*}} \geq 0$, $Z_t := \Exp{\norm{\nabla F(w_t)}^2} \geq 0$ for $t \geq 1$, $\rho := \frac{1}{2}$, and $D := \frac{L \sigma^2}{2}$.
Then, the estimate \eqref{eq_recursive_sgd_01} becomes
\begin{equation*}
Y_{t+1} \leq Y_t - \rho \eta_t Z_t + D \eta_t^2. 
\end{equation*}
Let us define $C := \big[ F(w_{1}) - F_{*} \big] + \frac{L \sigma^2 \gamma^2}{2(2\alpha - 1)\beta^{2\alpha - 1}} > 0$. 
By \eqref{eq_bounded_function_sgd_02}, we have $Y_t \leq C$ (note that $H = 0$ in Lemma~\ref{lem_general_framework_02}), $t \geq 1$. 
Applying Lemma~\ref{lem_general_framework_02} with $q = 2$, $m = 1$, and $\frac{1}{2} < \alpha < 1$, we conclude that
\begin{equation*}
\arraycolsep = 0.2em
\begin{array}{lcl}
    \frac{1}{T} \sum_{t = 1}^{T} \mathbb{E}\big[ \norms{\nabla F(w_t)}^2 \big]   & \leq & \frac{2 (1+\beta)^{\alpha} \big[ F(w_{1}) - F_{*} \big]}{ \gamma} \cdot \frac{1}{T} \ + \ \frac{C}{ \alpha \gamma} \cdot \frac{(T - 1 + \beta)^{\alpha}}{T}  \ + \  \frac{L \sigma^2 \gamma}{ (1 - \alpha)} \cdot \frac{(T + \beta)^{1 - \alpha}}{T}, 
    \end{array}
\end{equation*}
which proves \eqref{eq:th7_bounds}.
\end{proof}
%% End of the proof.

%%% Remark 5.
\begin{rem}\label{rem_thm_sgd_nonconvex_01}
In Theorem~\ref{thm_sgd_nonconvex_01}, if we choose $\alpha := \frac{1}{2} + \delta$ for some $0 < \delta < \frac{1}{2}$, then we have
\begin{equation*}
\arraycolsep=0.2em
\begin{array}{lcl}
        \frac{1}{T} \sum_{t = 1}^{T} \mathbb{E}\big[ \norm{\nabla F(w_t)}^2 \big]
        & \leq & \frac{2 (1+\beta)^{\frac{1}{2} + \delta} \left[ F(w_{1}) - F_{*} \right]}{ \gamma} \cdot \frac{1}{T} \ + \ \frac{C}{ \gamma (\frac{1}{2} + \delta)} \cdot \frac{(T - 1 + \beta)^{\frac{1}{2} + \delta}}{T}  +  \frac{L \sigma^2 \gamma}{ (\frac{1}{2} - \delta)} \cdot \frac{(T + \beta)^{\frac{1}{2} - \delta}}{T} \vspace{1ex}\\
    & = &  \Ocal\left( \frac{1}{T^{\frac{1}{2} - \delta}} \right),  
\end{array}
\end{equation*}
where $C := \big[ F(w_{1}) - F_{*} \big] + \frac{L \sigma^2 \gamma^2}{4 \delta \beta^{4\delta}} > 0$.
This rate converges to $\Ocal\left( \frac{1}{\sqrt{T}} \right)$ as $\delta \downarrow 0$.
\end{rem}
%% End of Remark 7.

Finally, if we use $\alpha := 1/2$ and diminishing step-size, then we have the following result.

%%% Theorem 8.
\begin{thm}\label{thm_sgd_nonconvex_02}
Assume that Assumptions \ref{ass_basic}$\mathrm{(i)}$, \ref{ass_weaker_smooth_F}, and \ref{ass_bounded_variance} hold for \eqref{expected_risk}. 
Let $\sets{ w_t }$ be generated by Algorithm~\ref{sgd_algorithm} with the step-size $0 < \eta_t:= \frac{\gamma}{(t+\beta)^{1/2}} \leq \frac{1}{L}$ for some $\gamma > 0$ and $\beta > 0$, 
and  $C := \big[ F( w_{1} ) - F_{*} \big] + \frac{L\sigma^2 \gamma^2}{2(1+\beta)} > 0$. 
Then, the following bound holds:
\begin{equation}\label{eq:th8_bounds}
\arraycolsep=0.2em
\begin{array}{lcl}
    \frac{1}{T} \sum_{t = 1}^{T} \mathbb{E} \big[ \norm{\nabla F(w_t)}^2 \big]  
    & \leq & \frac{2 (1+\beta)^{1/2} \left[ F( w_{1} ) -F_{*} \right]}{ \gamma} \cdot \frac{1}{T} \ + \ \frac{2 C}{ \gamma} \cdot \frac{(T - 1 + \beta)^{1/2}}{T}   \vspace{1ex}\\
   && + {~} L \gamma \sigma^2 \cdot \frac{(T - 1 + \beta)^{1/2} \log(T + 1 + \beta)}{T} \ + \  2 L \gamma \sigma^2 \cdot \frac{(T + \beta)^{1/2}}{T}  \vspace{1ex}\\
    &=  & \Ocal\left( \frac{\log(T)}{T^{1/2}}  \right).
\end{array}
\end{equation}
\end{thm}

%%% The proof of Theorem 8.
\begin{proof}
If $\alpha := \frac{1}{2}$, then, by \eqref{eq_main_bounded_function_sgd}, we can easily show that
\begin{equation*}
\arraycolsep=0.2em
\begin{array}{lcl}
    \Exp{F(w_{k+1}) - F_{*} } &\leq &  \Exp{F(w_{1}) - F_{*} } + \frac{ L \sigma^2}{2} \sum_{t=1}^k \eta_t^2 = \big[ F(w_{1}) - F_{*} \big] + \frac{L\sigma^2}{2} \sum_{t=1}^k \frac{\gamma^2}{(t + \beta)}  \vspace{1ex}\\
    & \overset{\eqref{integral_03}}{\leq}  &  \big[ F(w_{1}) - F_{*} \big] + \frac{L\sigma^2 \gamma^2}{2(1+\beta)} + \frac{L\sigma^2 \gamma^2}{2} \int_{t=1}^{k} \frac{dt}{(t + \beta)^{2\alpha}}.
\end{array}
\end{equation*}
Hence, for $k \geq 1$, we have
\begin{equation}\label{eq_bounded_function_sgd_01}
    \Exp{F(w_{k+1}) - F_{*} } \leq \big[ F(w_{1}) - F_{*} \big] \ + \ \frac{L\sigma^2 \gamma^2}{2(1+\beta)} \ + \ \frac{L \sigma^2 \gamma^2}{2} \cdot \log(k + 2 + \beta).
\end{equation}
Define $Y_t := \Exp{F(w_{t}) - F_{*}} \geq 0$, $Z_t := \Exp{\norms{\nabla F(w_t)}^2} \geq 0$ for $t \geq 1$, $\rho := \frac{1}{2}$, and $D := \frac{L \sigma^2}{2}$.
Then, the estimate \eqref{eq_recursive_sgd_01} becomes
\begin{equation*}
Y_{t+1} \leq Y_t - \rho \eta_t Z_t + D \eta_t^2. 
\end{equation*}
Let us define $C := [F( w_{1} ) -F_{*}] + \frac{L\sigma^2 \gamma^2}{2(1+\beta)} > 0$, $H: = \frac{ L \sigma^2 \gamma^2}{2} > 0$, and $\theta:= 1 + \beta > 0$.
Clearly, $1 + \theta - \beta = 2 > \frac{1}{2}e$. By \eqref{eq_bounded_function_sgd_01}, we have $Y_t \leq C + H \log(t + \theta)$ for $t \geq 1$. Applying Lemma~\ref{lem_general_framework_02} with $q = 2$, $m = 1$, and $\alpha = \frac{1}{2}$, we conclude that
\begin{equation*}
\arraycolsep = 0.2em
\begin{array}{lcl}
    \frac{1}{T} \sum_{t = 1}^{T} \Exp{\norms{\nabla F(w_t)}^2} 
    & \leq  & \frac{2 (1+\beta)^{1/2} \left[ F( w_{1} ) - F_{*} \right]}{ \gamma} \cdot \frac{1}{T} \ + \  \frac{2 C}{ \gamma} \cdot \frac{(T - 1 + \beta)^{1/2}}{T} \vspace{1ex}\\
    && + {~} L \gamma \sigma^2 \cdot \frac{(T - 1 + \beta)^{1/2} \log(T + 1 + \beta)}{T} \ + \ 2 L \gamma \sigma^2 \cdot \frac{(T + \beta)^{1/2}}{T}  \vspace{1ex}\\
    &= & \Ocal\left( \frac{\log(T)}{T^{1/2}}  \right),
\end{array}
\end{equation*}
which proves \eqref{eq:th8_bounds}.
\end{proof}
%%% End of the proof.

%%% References.
\bibliographystyle{unsrtnat}
%%\bibliography{all_refs}

\end{document}